\documentclass[english]{amsart}
\usepackage[utf8]{inputenc} 
\usepackage[T1]{fontenc} 
\usepackage{lmodern}
\usepackage[left=1.9cm,right=1.9cm,top=2cm,bottom=2cm,a4paper]{geometry}
\usepackage[usenames, dvipsnames]{xcolor} 
\usepackage{graphicx} 
\usepackage{epstopdf}
\usepackage{amsthm} 
\usepackage{amssymb} 
\usepackage{fdsymbol} 
\usepackage{stmaryrd} 
\usepackage{mathtools} 
\usepackage{xspace} 
\usepackage{array} 
\usepackage{booktabs} 
\usepackage{multirow} 
\usepackage{xparse} 
\usepackage{tikz} 
\usetikzlibrary{cd} 
\usetikzlibrary{arrows}
\usetikzlibrary{decorations.markings}
\tikzstyle directed=[postaction={decorate,decoration={markings,
		mark=at position .65 with {\arrow{latex}}}}]
\usepackage{amsfonts} 
\usepackage{enumerate} 
\usepackage{csquotes} 
\usepackage[backend=bibtex,style=alphabetic]{biblatex} 
\usepackage{babel} 
\usepackage{varioref} 
\usepackage[colorlinks]{hyperref} 
\usepackage{cleveref} 
\usepackage[all]{hypcap} 
\usepackage[xindy]{imakeidx} 
\makeindex[title=Index des notations,intoc]
\usepackage[xindy]{glossaries} 

\makeatletter
\newcommand{\proofstep}[1]{%
	\par
	\addvspace{\medskipamount}
	\textit{#1\@addpunct{.}}\enspace\ignorespaces
}
\makeatother

\DeclareMathOperator{\Id}{Id} 
\DeclareMathOperator{\Graph}{Graph} 
\DeclareMathOperator{\MatrixSet}{\mathcal{M}} 
\DeclareMathOperator{\Sp}{Sp} 
\DeclareMathOperator{\Opegaldef}{\overset{\text{def}}{=}} 

\DeclareMathOperator{\Ker}{Ker}
\DeclareMathOperator{\RealPart}{Re}
\DeclareMathOperator*{\DirectSumEV}{\oplus}
\newcommand{\N}{\mathbb{N}} 
\newcommand{\R}{\mathbb{R}} 
\newcommand{\C}{\mathbb{C}} 
\newcommand{\dd}{\mathrm{d}} 
\newcommand{\abs}[1]{\left\lvert#1\right\rvert} 
\newcommand{\norme}[1]{\left\lVert#1\right\rVert} 
\newcommand{\normesub}[1]{\left\lVvert#1\right\rVvert} 
\newcommand{\motdef}[1]{\emph{#1}\xspace} 
\newcommand{\egaldef}{\overset{\text{def}}{=}} 
\newcommand{\fonction}[5]{%
	#1\colon\begin{array}{l|rcl}
		& #2 & \to & #3 \\
		& #4 & \mapsto & #5
	\end{array}
}

\newcommand{\Diff}[1]{D#1}
\newcommand{\DiffPartielGeneric}[1]{D_{#1}}
\newcommand{\DiffPoint}[2]{D#1(#2)}	
\newcommand{\DiffSup}[2]{D^{#2}#1}
\newcommand{\DiffPointSup}[3]{D^{#3}#1(#2)}
\newcommand{\DiffPartiel}[2]{D_{#2}#1}
\newcommand{\DiffPartielPoint}[3]{D_{#3}#1(#2)}	

\newcommand{\DiffPartielPointSup}[4]{D_{#4}^{#3}#1(#2)}
\newcommand{\Integral}[4]{\int_{#1}^{#2} #3 \, \dd #4}
\newcommand{\MyBigSet}[2]{\left\{#1 \: | \: #2\right\}}
\newcommand{\OpenInterval}[2]{\left(#1,#2\right)}

\DeclareMathOperator{\II}{II} 

\DeclareDocumentCommand\vpWH{ m o }{%
	\mu_{#1}\IfNoValueF{#2}{(#2)}		
}
\DeclareDocumentCommand\vpWHK{ m o }{%
	\bar{\mu}_{#1} \IfNoValueF {#2} {\left(#2\right)}		
}
\DeclareDocumentCommand\vpWHR{ m o }{%
	\tilde{\mu}_{#1}\IfNoValueF{#2}{\left(#2\right)}		
}

\newcommand{\cKQ}{\mathcal{K}}
\DeclareDocumentCommand\Kparam{ o }{%
	\omega\IfNoValueF{#1}{(#1)}		
}

\DeclareDocumentCommand\KparamEuclid{ o }{%
	v_E\IfNoValueF {#1} {\left(#1\right)}		
}
\DeclareDocumentCommand\KparamCurv{ o }{%
	v_{\cKQ}\IfNoValueF {#1} {\left(#1\right)}		
}

\newcommand{\BphaseQ}{\mathcal{B}}

\DeclareDocumentCommand\ActionNoProper{ o }{%
	\mathscr{T}\IfNoValueF {#1} {_{#1}}		
}

\DeclareDocumentCommand\orbitIIu{ o }{%
	\mathcal{O}_{\II}^u\IfNoValueF {#1} {\left(#1\right)}		
}
\DeclareDocumentCommand\orbitIIs{ o }{%
	\mathcal{O}_{\II}^{s_1}\IfNoValueF {#1} {\left(#1\right)}		
}
\DeclareDocumentCommand\orbitIIss{ o }{%
	\mathcal{O}_{\II}^{s_2}\IfNoValueF {#1} {\left(#1\right)}		
}

\newcommand{\LocalCoorPathMapGeneric}{\Phi}
\DeclareDocumentCommand\LocalCoorPathMap{ m o }{%
	\LocalCoorPathMapGeneric_{#1}^{\IfNoValueF{#2}{#2}}
}

\newcommand{\EpochFirstReturnMapGeneric}{\rho}
\DeclareDocumentCommand\EpochFirstReturnMap{ m o }{%
	\EpochFirstReturnMapGeneric_{#1}^{\text{\IfNoValueF {#2} {#2}}}
}

\DeclareDocumentCommand\AppliPremierRetourEreI{ o }{%
	\hat{\rho}_{\IfNoValueF{#1} {#1} }
}

\newcommand{\VarStableLoc}[2]{W_{#1}^{\text{loc}}(#2)}
\DeclareDocumentCommand\VarStableLoc{ m o }{%
	W_{#1}^{\text{loc}}\IfNoValueF {#2} {\left(#2\right)}
}

\DeclareDocumentCommand\cKgraph{ o }{%
	\hat{\cKQ}_{\text{graph}}\IfNoValueF {#1} {\left(#1\right)}
}
\DeclareDocumentCommand\cSgraph{ o }{%
	\Sigma^{s}_{\text{graph}}\IfNoValueF {#1} {\left(#1\right)}
}

\DeclareDocumentCommand\LocalNormFunc{ m o }{%
	\gamma_{#1}\IfNoValueF{#2}{\left(#2\right)}		
}

\DeclareDocumentCommand\SectionLocale{ m m o m o }{%
	\Sigma_{#1}^{\text{#2\IfNoValueF {#3} {,#3}}}\left(#4\IfNoValueF {#5} {, #5}\right)		
}

\DeclareDocumentCommand\SectionLocaleGeom{ o m o m o }{%
	\mathbf{\Sigma}_{#2}^{\text{#4\IfNoValueF {#1} {,#1}}}\IfNoValueF {#3} {\left(#3\IfNoValueF {#5} {,#5}\right)}		
}
\DeclareDocumentCommand\SectionLocaleCanonique{ m o m o }{%
	\tilde{\Sigma}_{#1}^{\text{#3\IfNoValueF {#2} {,#2}}} \IfNoValueF {#4} {\left(#4\right)} 		
}

\DeclareDocumentCommand\TangCone{ m m o }{%
	V_{#1}^{c}\left(#2\IfNoValueF {#3} {,#3}\right)		
}
\DeclareDocumentCommand\TransCone{ m m o }{%
	V_{#1}^{\bot}\left(#2\IfNoValueF {#3} {,#3}\right)		
}

\DeclareDocumentCommand\FPgraph{ o }{%
	\hat{\zeta}\IfNoValueF {#1} {_{#1}}
}

\DeclareDocumentCommand\EuclidDistance{ o }{%
	d_E\IfNoValueF {#1} {\left(#1\right)}	
}
\DeclareDocumentCommand\CurvilignDistance{ o }{%
	d_{\cKQ}\IfNoValueF {#1} {\left(#1\right)}	
}

\DeclareDocumentCommand\dBphaseQ{ o }{%
	d_{\BphaseQ}\IfNoValueF {#1} {\left(#1\right)}	
}

\newcommand{\SM}[1]{W^{s}(#1)}

\newcommand{\LSM}[2]{W^{s}_{#2}\left(#1\right)}

\newcommand{\SMexp}[2]{W^{s,#2}\left(#1\right)}
\newcommand{\LSMexp}[3]{W^{s,#2}_{#3}(#1)}

\newcommand{\VFParam}{X}
\newcommand{\VFParamStraight}{\tilde{\VFParam}}

\newcommand{\OriginDiff}{A}
\newcommand{\OriginDiffStraight}{\tilde{\OriginDiff}}
\newcommand{\VariablePhaseSpace}{x}
\newcommand{\DimPhaseSpace}{n}
\newcommand{\PhaseSpaceVF}{\R^\DimPhaseSpace}
\newcommand{\Parameter}{\mu}
\newcommand{\DimSetParameter}{s}
\newcommand{\SetParameter}{\R^\DimSetParameter}
\newcommand{\StableSpace}{F}
\newcommand{\UnstableSpace}{G}
\newcommand{\VarPhaseSpaceStable}{z}
\newcommand{\VarPhaseSpaceUnstable}{v}
\newcommand{\LineartPart}{A}
\newcommand{\StableMatrix}{\LineartPart_{|\StableSpace}}
\newcommand{\StableMatrixStraight}{\tilde{\LineartPart}_{1}}
\newcommand{\UnstableMatrix}{\LineartPart_{|\UnstableSpace}}
\newcommand{\UnstableMatrixStraight}{\tilde{\LineartPart}_{2}}
\newcommand{\DimStableSpaceStraight}{p}
\newcommand{\StableSpaceStraight}{\R^\DimStableSpaceStraight}
\newcommand{\DimUnstableSpaceStraight}{q}
\newcommand{\UnstableSpaceStraight}{\R^\DimUnstableSpaceStraight}
\newcommand{\BorneVPStable}{\alpha}
\newcommand{\BorneVPUnstable}{\beta}
\DeclareMathOperator{\OpFixedPoint}{\mathcal{O}}
\newcommand{\OpFamily}[2]{\OpFixedPoint^{#1}_{#2}}
\newcommand{\OpFamilyGeneric}[1]{\OpFixedPoint^{#1}}
\newcommand{\VarConditionInitiale}{\omega}
\DeclareMathOperator{\OpFixedPointII}{\Lambda}
\newcommand{\OpFamilyII}[1]{\OpFixedPointII_{#1}}
\newcommand{\ExpDecaySpeed}{\gamma}
\newcommand{\FunctionSpace}[2]{H^{#1}_{#2}}
\newcommand{\FunctionSpaceGeneric}[1]{H^{#1}}
\newcommand{\NonLinearPartEDI}{f}
\newcommand{\NonLinearPartEDII}{g}

\newcommand{\SMTGraphMap}{\phi}
\newcommand{\SMTGraphMapStraight}{\tilde{\SMTGraphMap}}
\newcommand{\SMTGraphMapLocal}{\phi}
\newcommand{\SMTGraphMapII}{\psi}
\newcommand{\NormSupDerivativesVFParamGeneric}{M}
\newcommand{\NormeSupHigherDerivativesLoc}[2]{\NormSupDerivativesVFParamGeneric_{#1}\left(#2\right)}
\newcommand{\NormeSupHigherDerivativesBounded}[2]{\NormSupDerivativesVFParamGeneric_{#1}\left(#2\right)}
\newcommand{\NormeSupHigherDerivativesLocMax}[2]{\bar{\NormSupDerivativesVFParamGeneric}_{#1}\left(#2\right)}
\newcommand{\NormeSupHigherDerivativesBoundedMax}[2]{\bar{\NormSupDerivativesVFParamGeneric}_{#1}\left(#2\right)}
\newcommand{\SpectreUp}[1]{\lambda_{\text{max}}\left(#1\right)}
\newcommand{\SpectreDown}[1]{\lambda_{\text{min}}\left(#1\right)}
\newcommand{\DistanceBordSpectre}[2]{d_{#1}(#2)}

\newcommand{\FPStableGeneric}{z^{*}}
\newcommand{\FPStable}[2]{\FPStableGeneric_{#1,#2}}
\newcommand{\FPUnstableGeneric}{v^{*}}
\newcommand{\FPUnstable}[2]{\FPUnstableGeneric_{#1,#2}}
\newcommand{\Angle}[1]{\sphericalangle\left(#1\right)}
\newcommand{\ConstanteInegTriangInverse}[1]{m\left(#1\right)}
\newcommand{\ConstanteControleExpMatrix}[1]{M(#1)}
\newcommand{\ConstanteControleExpMatrixNormeII}[1]{\hat{M}(#1)}

\newcommand{\ConstantEstimateStableManifoldThm}{C}

\DeclareDocumentCommand\ConstI{ o }{%
	\ConstantEstimateStableManifoldThm_{1\IfNoValueF {#1} {,#1}} 		
}

\DeclareDocumentCommand\ConstII{ o }{%
	\ConstantEstimateStableManifoldThm_{2\IfNoValueF {#1} {,#1}} 		
}

\DeclareDocumentCommand\ConstIII{ o }{%
	\ConstantEstimateStableManifoldThm_{3\IfNoValueF {#1} {,#1}} 		
}
\newcommand{\PHSIsoStraight}{L}
\newcommand{\TailleTronc}{\xi}
\newcommand{\TroncVFParam}[1]{\VFParam^{#1}}
\newcommand{\gphs}[1]{\sigma\left(#1\right)}
\newcommand{\ExpDecaySpeedInterval}{I_{\OriginDiff}}
\newcommand{\TailleBoule}{r}

\newcommand{\FixedVF}{Y}
\newcommand{\OpenSetDefinVF}{\Omega}
\newcommand{\TraceOpenSetDefinVF}{\OpenSetDefinVF_0}

\newcommand{\SLC}{\xi}
\newcommand{\SLCsmall}{\tilde{\SLC}}

\newcommand{\EVII}[1]{\beta(#1)}
\newcommand{\SMF}{\mathcal{F}^{s}}

\theoremstyle{plain}
\newtheorem{thm}{Theorem}[section]

\newtheorem{prop}[thm]{Proposition}
\newtheorem{lemme}[thm]{Lemma}
\newtheorem{hypo}{Hypothesis}

\newtheorem*{addendum}{Addendum}
\theoremstyle{remark}
\newtheorem{rem}[thm]{Remark}
\theoremstyle{definition}
\newtheorem{defin}[thm]{Definition}

\numberwithin{equation}{section} 

\title{Some estimates for the stable manifold theorem}

\author{Tom \textsc{Dutilleul}
}

\address{LAGA (UMR 7539), Université Paris \textsc{XIII}, Sorbonne Paris Cité.}
\email{\nolinkurl{dutilleul@math.univ-paris13.fr}}
\urladdr{\url{https://www.tomdutilleul.com/}}
\date{\today}
\newcounter{MyCounterProof}
\newcounter{MySubCounterProof}
\begin{document}
	\begin{abstract}
We investigate the standard stable manifold theorem in the context of a partially hyperbolic singularity of a vector field depending on a parameter. We prove some estimates on the size of the neighbourhood where the local stable manifold is known to be the graph of a function, and some estimates about the derivatives of all orders of this function. We explicitate the different constants arising and their dependance on the vector field.
As an application, we consider the situation where a vector field vanishes on a submanifold $N$ and contracts a direction transverse to $N$. We prove some estimates on the size of the neighbourhood of $N$ where there are some charts straightening the stable foliation while giving some controls on the derivatives of all orders of the charts. 
	\end{abstract}

\maketitle



\section{Introduction}
Fix a smooth vector field $\FixedVF$ on a Riemannian manifold $M$ and let $x$ be a \motdef{singularity} of the vector field $\FixedVF$, that is, a point of $M$ such that $\FixedVF(x)=0$. For any $\ExpDecaySpeed <0$ and for any $\eta >0$, the \motdef{local $\ExpDecaySpeed$-stable set} $\LSMexp{x,\FixedVF}{\ExpDecaySpeed}{\eta}$ of $x$ for $\FixedVF$ is the set of points in $M$ whose forward orbit under the flow of $\FixedVF$ stay in the $\eta$-neighbourhood of $x$ and converge to $x$ faster than $e^{\ExpDecaySpeed t}$ as $t\to + \infty$ (see~\eqref{e.local-gamma-stable-manifold}). This is one of the most fundamental objects when one tries to understand the asymptotic dynamics of the flow of $\FixedVF$ near $x$. Its geometry is very well understood in the context of a hyperbolic (or partially hyperbolic) singularity, as explained in what follows.

\subsection{Stable manifold theorem}

Assume that the singularity $x$ is partially hyperbolic: up to replacing $\FixedVF$ by $-\FixedVF$, this means that there exists a non trivial decomposition $T_x M = \StableSpace  \oplus \UnstableSpace $ of the tangent space at $x$ such that $\StableSpace$ and $\UnstableSpace$ are stabilized by $\DiffPoint{\FixedVF}{x}$ and there exists a negative real $\ExpDecaySpeed$ such that the real parts of the eigenvalues of $\DiffPoint{\FixedVF}{x}_{|\StableSpace}$ are strictly less than $\ExpDecaySpeed$ and the real parts of the eigenvalues of $\DiffPoint{\FixedVF}{x}_{|\UnstableSpace}$ are strictly more than $\ExpDecaySpeed$. In this context, the Stable Manifold Theorem asserts that for any positive $\eta$ small enough, the local $\ExpDecaySpeed$-stable set $\LSMexp{x,X}{\ExpDecaySpeed}{\eta}$ is an embedded submanifold of $M$ tangent to $\StableSpace$ at $x$, called the local $\ExpDecaySpeed$-stable manifold. It can be seen as the graph of a smooth map $\phi: U \subset \StableSpace \to V \subset \UnstableSpace$, from a neighbourhood $U$ of $0$ in $\StableSpace$ to a neighbourhood $V$ of $0$ in $\UnstableSpace$, satisfying $\phi(0)=0$ and $\DiffPoint{\phi}{0}=0$. Moreover, if $\FixedVF$ depends smoothly on a parameter $\Parameter \in \SetParameter$, then this is also the case for the submanifold described earlier, that is, $\phi_\Parameter(z)=\phi(z,\Parameter)$ is smooth as a map of the two variables $z \in U$, $\Parameter \in \SetParameter$.

Though this standard theorem has been presented and generalized in many articles (see \emph{e.g.} \cite{IrwinStableManifoldTheorem1970}, \cite{HirschStablemanifoldshyperbolic1970}) and books (see \emph{e.g.} \cite{KatokIntroductionModernTheory1997}, \cite{IrwinSmoothDynamicalSystems2001}, \cite{RuelleElementsDifferentiableDynamics1989}, \cite{RobinsonDynamicalSystemsStability1999}, \cite{brin_introduction_2002} for classical introductory readings and \cite{HirschInvariantManifolds2006} for a deeper treatment but a tougher reading), we have not found a version of this result that gives explicit estimates on the $C^k$-norms of $\phi(z,\Parameter)$ ($k\in \N^*$) and on the size of the neighbourhood where these estimates hold true. In most of the books, authors state that if $\FixedVF$ is $C^r$, then $\phi_\Parameter$ is also $C^r$ and $\Parameter \mapsto \phi_\Parameter$ is a continuous map from $\SetParameter$ to the space of $C^r$ maps equipped with the $C^r$ topology, which is a weaker statement than saying that $\phi(z,\Parameter)$ is smooth. The closest result to what we were looking for has been found in \autocite{chua_methods_1998} (chapter $5$). They prove rigorously that the map $\phi(z,\Parameter)$ is smooth but do not provide explicit estimates. This is the reference that motivated the writing of this paper, whose purpose is to give such estimates. Since we are only interested by local estimates, we may (and do) assume that $M = \PhaseSpaceVF$ and $x=0$ (it suffices to work in a local chart and to multiply the vector field by a smooth plateau map in the neighbourhood of $0$). 

The classical stable manifold theorem (which can be found in the above references) can be stated as following:
\begin{thm}[Stable manifold theorem with parameters]\label{thm.variete-gamma-stable-local-version-intro}
	Let $\VFParam=(\VFParam_\Parameter)_{\Parameter \in \SetParameter}$ be a smooth family of smooth vector fields on $\PhaseSpaceVF$ such that
\begin{enumerate}
	\item For every $\Parameter \in \SetParameter$, the origin of $\PhaseSpaceVF$ is a singularity of $\VFParam_\Parameter$, \emph{i.e.}
	\begin{equation*}
	\VFParam_\Parameter(0)=0
	\end{equation*}
	\item The endomorphism $\OriginDiff:=\DiffPartielPoint{\VFParam}{0,0}{\VariablePhaseSpace}$ admits a \motdef{partially hyperbolic splitting} $R^n = F \oplus G$ such that
	\begin{equation*}
	\SpectreUp{\OriginDiff_{|\StableSpace}} < \min \left(0 , \SpectreDown{\OriginDiff_{|\UnstableSpace}}\right)
	\end{equation*}
\end{enumerate}
where $\SpectreUp{\OriginDiff_{|\StableSpace}}$ (resp. $\SpectreDown{\OriginDiff_{|\UnstableSpace}}$) denotes the max (resp. min) of the real parts of the eigenvalues of $\OriginDiff_{|\StableSpace}$ (resp. $\OriginDiff_{|\UnstableSpace}$).
Let $\ExpDecaySpeed \in \OpenInterval{\SpectreUp{\OriginDiff_{|\StableSpace}}}{\min\left(0,\SpectreDown{\OriginDiff_{|\UnstableSpace}}\right)}$.
Then there exists $\epsilon > 0$ and $\eta >0$ such that for every $\Parameter \in B_{\SetParameter}(0,\epsilon)$, the local $\ExpDecaySpeed$-stable set $\LSMexp{0,\VFParam_\Parameter}{\ExpDecaySpeed}{\eta}$ is the graph of a smooth function $\SMTGraphMap_\Parameter: \StableSpace \to \UnstableSpace$ intersected with the ball $B_{\PhaseSpaceVF}(0,\eta)$. Moreover, the map $\SMTGraphMap: (z,\Parameter) \in \StableSpace \times B_{\SetParameter}(0,\epsilon) \mapsto \SMTGraphMap_\Parameter(z) \in \UnstableSpace$ is smooth, for every $\Parameter \in B_{\SetParameter}(0,\epsilon)$, $\SMTGraphMap_\Parameter(0)=0$ and $\DiffPoint{\SMTGraphMap_0}{0}=0$.
\end{thm}
As explained above, our goal is to supplement this result by providing explicit estimates on $\epsilon$, $\eta$ and the derivatives of all orders of $\SMTGraphMap$. What we prove is summarized in the following addendum (for a precise version, see theorem~\ref{thm.variete-gamma-stable-local-version}):
\begin{addendum}
	For every $\TailleBoule >0$, one can find a radius $\epsilon$, a size $\eta$ and a map $\SMTGraphMap$ as above satisfying the following properties:
\begin{itemize}
	\item the radius $\epsilon$ is linear in $\TailleBoule$; polynomial on the distance between $\ExpDecaySpeed$ and the real part of the spectrum of $\OriginDiff$; inversely linear on the norm of the second derivative of $\VFParam$ on the closed ball $\overline{B}_{\PhaseSpaceVF \times \SetParameter}(0,\TailleBoule)$ and inversely polynomial on the norm of $\OriginDiff$ and the angle between the generalized eigenspaces of $\OriginDiff$.
	\item the size $\eta$ is linear in $\TailleBoule$; polynomial on the spectral gap $ \min \left(0 , \SpectreDown{\OriginDiff_{|\UnstableSpace}}\right)-\SpectreUp{\OriginDiff_{|\StableSpace}}$; inversely linear on the norm of the second derivative of $\VFParam$ on the closed ball $\overline{B}_{\PhaseSpaceVF \times \SetParameter}(0,\TailleBoule)$ and inversely polynomial on the norm of $\OriginDiff$ and the angle between the generalized eigenspaces of $\OriginDiff$.
	\item the norm of the $k$-th derivative of $\SMTGraphMap$ on $B_{\StableSpace}(0,\eta) \times B_{\SetParameter}(0,\eta)$ is a polynomial function of degree~$\simeq n k^2$ depending on the norm of $\OriginDiff$, the angle between the generalized eigenspaces of $\OriginDiff$, the inverse of $\TailleBoule$, the norms of the $(k+1)$ first derivatives of $\VFParam$ on the closed ball $\overline{B}_{\PhaseSpaceVF \times \SetParameter}(0,\TailleBoule)$ and the inverse of the spectral gap.
\end{itemize}	
\end{addendum}

\begin{rem}
	The parameter $r$ describes quantitatively how the local $\ExpDecaySpeed$-stable manifold is, indeed, a local object. It allows one to get some information on the size of the local $\ExpDecaySpeed$-stable manifold when one is only using a control of $\VFParam$ over the ball of radius $r$.
\end{rem}

\begin{rem}
	The strategy used to prove theorem~\ref{thm.variete-gamma-stable-local-version-intro} is standard. We find the orbits contained in a stable manifold as the fixed points of an "integral" operator (depending on the parameter $\Parameter$) on a suitable space of functions. The construction of the operator is natural and gives the desired description of the stable manifolds as graphs of some family of maps $\SMTGraphMap_\Parameter$. This is the technique used in \autocite{chua_methods_1998}, but with a major simplification. We directly prove that on the one hand the operator is smooth with respect to all variables including the parameters and on the other hand it is a contraction mapping with respect to the space of functions, thus we obtain that the family of graphs $\SMTGraphMap$ is smooth with respect to the variable in the phase space and the parameter, using a global version of the implicit function theorem (which can be seen as a contraction mapping theorem with parameters). This makes the proof easier and more natural compared to the one in \autocite{chua_methods_1998}. Indeed, in this reference, the authors do not prove that the operator is smooth and thus need to use a family of truncated operators to obtain the smoothness of the fixed point.
\end{rem}

\subsection{Vector fields vanishing on submanifolds}
Theorem~\ref{thm.variete-gamma-stable-local-version-intro} allows us to describe the stable foliation associated with a normally contracted submanifold on which a vector field vanishes. The context is as follows. Let $M$ be a smooth manifold of dimension $\DimPhaseSpace$ and let $N$ be a smooth submanifold of $M$. Let $\FixedVF$ be a smooth vector field on $M$ vanishing on $N$ such that for every point $x \in N$, there exists a direction transverse to $T_xN$ which is stabilized and contracted by $\DiffPoint{\FixedVF}{x}$. Recall that, given $x \in N$, the \motdef{stable set} $\SM{x,\FixedVF}$ of $x$ for $\FixedVF$ is the set of points in $M$ whose forward orbit under the flow of $\FixedVF$ converge to $x$. It is well known (this is an easy consequence of theorem~\ref{thm.variete-gamma-stable-local-version-intro}) that the family of stable manifolds $(\SM{x,\FixedVF})_{x \in N}$ foliates a neighbourhood $W$ of $N$ and the stable foliation
	\begin{equation*}
\SMF \egaldef \MyBigSet{\SM{x,\FixedVF} \cap W}{x \in N}
\end{equation*}
can be locally smoothly straightened.

Fixing a point $x \in N$ and a local chart (independantly of $\FixedVF$) centered around $x$ which straightens $N$, and looking at the situation in this chart, we "can assume that" $M$ is an open set $\OpenSetDefinVF$ of $\PhaseSpaceVF$ and $N$ is the set $\TraceOpenSetDefinVF:= \OpenSetDefinVF \cap \UnstableSpace \neq \emptyset$ where $\UnstableSpace$ is a linear subspace of $\PhaseSpaceVF$.
The standard result explained above can be stated as following:

\begin{thm}[Straightening of a stable foliation]
\label{thm.SMF-straightening-intro}
Let $\OpenSetDefinVF$ be an open neighbourhood of $0$ in $\PhaseSpaceVF$, $\UnstableSpace$ be a linear subspace of $\PhaseSpaceVF$ and $\FixedVF: \OpenSetDefinVF \to \PhaseSpaceVF$ be a smooth vector field such that
\begin{enumerate}
	\item $\FixedVF$ vanishes on $\TraceOpenSetDefinVF:= \OpenSetDefinVF \cap \UnstableSpace$;
	\item For every $\Parameter \in \TraceOpenSetDefinVF$, there exists a decomposition $\StableSpace_{\Parameter} \oplus \UnstableSpace = \R^n$ stabilized by $\OriginDiff_{\Parameter}:=\DiffPoint{\FixedVF}{\Parameter}$ and such that
	\begin{equation*}
	\SpectreUp{(\OriginDiff_{\Parameter})_{|\StableSpace_{\Parameter}}} <0
	\end{equation*}
\end{enumerate}
Let $\Parameter_0 \in \TraceOpenSetDefinVF$. Then there exists a smooth local coordinate system $\SLC$ defined on a ball $B_{\PhaseSpaceVF}\left(\Parameter_0,R\right)$ such that the family of stable manifolds $(\SM{\Parameter,\FixedVF})_{\Parameter \in \TraceOpenSetDefinVF \cap B_{\PhaseSpaceVF}\left(\Parameter_0,R\right)}$ foliates $B_{\PhaseSpaceVF}\left(\Parameter,R\right)$ and such that the local coordinate system straightens the stable foliation: for every $\Parameter \in \TraceOpenSetDefinVF \cap B_{\PhaseSpaceVF}\left(\Parameter_0,R\right)$,
\begin{equation*}
\SLC\left( \SM{\Parameter,\FixedVF} \cap B_{\PhaseSpaceVF}\left(\Parameter_0,R\right) \right) = \left(\Parameter + \StableSpace_{\Parameter_0}\right) \cap \SLC\left( B_{\PhaseSpaceVF}\left(\Parameter_0,R\right) \right)
\end{equation*}
\end{thm}
Once again, our goal is to provide some explicit estimates on the radius $R$ and on the derivatives of all orders of $\SLC$ and $\SLC^{-1}$. What we prove is summarized in the following addendum (for a precise version, see theorem~\ref{thm.redressement-local-feuilletage}):
\begin{addendum}
For every $\TailleBoule >0$ such that $\overline{B}_{\PhaseSpaceVF}(\Parameter_0,\TailleBoule) \subset \OpenSetDefinVF$, one can find a radius $R$ and a local coordinate system $\SLC$ as above satisfying the following properties:
\begin{itemize}
	\item The radius $R$ is linear in $\TailleBoule$; polynomial on the spectral gap $\abs{\SpectreUp{(\OriginDiff_{\Parameter_0})_{|\StableSpace_{\Parameter_0}}}}$; inversely linear on the norm of the second derivative of $\FixedVF$ on the closed ball $\overline{B}_{\PhaseSpaceVF}(\Parameter_0,\TailleBoule)$; inversely polynomial on the norm of $\OriginDiff_{\Parameter_0}$ and the angle between the generalized eigenspaces of $\OriginDiff_{\Parameter_0}$.
	\item For every $\epsilon >0$, $\SLC$ restricted to $B_{\PhaseSpaceVF}\left(\Parameter_0,\epsilon R\right)$ is $\epsilon$-close to the identity with respect to the $C^1$-norm.
	\item The norms of the $k$-th derivatives of $\SLC$ and $\SLC^{-1}$ are polynomial on the norm of $\OriginDiff_{\Parameter_0}$, the angle between the generalized eigenspaces of $\OriginDiff_{\Parameter_0}$ and the norms of the $(k+1)$ first derivatives of $\FixedVF$  on the closed ball $\overline{B}_{\PhaseSpaceVF}(\Parameter_0,\TailleBoule)$ and inversely polynomial on the spectral gap and $\TailleBoule$.
\end{itemize}
\end{addendum}

\begin{rem}
	In order to deduce this from theorem~\ref{thm.variete-gamma-stable-local-version-intro}, one must choose a compact ball $\overline{B}(\Parameter_0,\TailleBoule) \subset \OpenSetDefinVF$ on which one controls the derivatives of all orders of $\FixedVF$. There is no canonical choice and one can use the parameter $\TailleBoule$ to make a choice depending on its needs.
\end{rem}

The paper is organized as follows.
Section~\ref{section.notations} compiles some notations used throughout the paper. In section~\ref{section.stable-manifold-theorem-notations-setup}, we prove theorem~\ref{thm.variete-gamma-stable-local-version-intro}. We first treat the global case (see proposition~\ref{prop.variete-gamma-stable-global-version}), which is the main technical result of this paper, and then we apply it to the local case. In section~\ref{e.SMF}, we prove theorem~\ref{thm.SMF-straightening-intro} using theorem~\ref{thm.variete-gamma-stable-local-version-intro}. Appendix~\ref{section.appendix-lemmas-algebra} recalls some well-known estimates of linear algebra that are extensively used throughout the paper.

\section{General notations.}
\label{section.notations}
We introduce here some notations that will be used throughout this paper.
For any $n \in \N$, we denote by $\norme{.}$ the Euclidean norm on $\R^n$. For any family $(E_1,\norme{.}_1),\dots,(E_r,\norme{.}_r)$, $(F,\norme{.}_F)$ of normed vector spaces (possibly of infinite dimension), for any continuous $r$-linear map $L:E_1 \times \dots \times E_r \to F$, we will usually denote by $\normesub{L}$ its subordinate norm, that is,
\begin{equation*}
\normesub{L} = \sup_{(x_1,\dots,x_r) \in \prod_{i=1}^{r} E_i}\frac{\norme{L(x_1,\dots,x_r)}_F}{\prod_{i=1}^{r}\norme{x_i}_i}
\end{equation*}

For any linear subspaces $F,G$ of $\R^n$, let us recall that the \motdef{angle} between $F$ and $G$, denoted by $\Angle{F,G}$, is defined as the minimal (unsigned) angle between a vector in $F$ and a vector in $G$.
The angle between $F$ and $G$ is strictly positive if and only if $F \cap G = \{0\}$. If this is the case, let
\begin{equation*}
\ConstanteInegTriangInverse{F,G} \egaldef \left(\frac{2}{1-\cos \Angle{F,G}}\right)^{\frac{1}{2}}
\end{equation*}
We generalise this notion by defining the angle between a finite family $E_1,\dots,E_r$ of linear subspaces of $\R^n$ as following
\begin{equation*}
\Angle{E_1,\dots,E_r} \egaldef \min_{1 \leq j \leq r} \Angle{E_j, \DirectSumEV_{i \neq j} E_i}
\end{equation*}

For any $A \in \MatrixSet_n(\R)$, let
\begin{subequations}
	\label{e.max-min-spectre}
\begin{align}
\label{e.max-spectre}
\SpectreUp{A} &\egaldef \max_{\lambda \in \Sp_{\C}(A)} \RealPart(\lambda) \\
\label{e.min-spectre}
\SpectreDown{A} &\egaldef \min_{\lambda \in \Sp_{\C}(A)} \RealPart(\lambda)
\end{align}
\end{subequations}

\begin{subequations}
\begin{align}
\begin{split}
\label{e.ConstanteInegTriangInverse}
\ConstanteInegTriangInverse{A} &\egaldef \left(\frac{2}{1-\cos \Angle{E_1,\dots,E_r}}\right)^{\frac{r-1}{2}} \\
& \text{where $E_1,\dots,E_r$ are the generalized eigenspaces of $A$}
\end{split} \\
\begin{split}
\label{e.ConstanteControleExpMatrixNorme}
\ConstanteControleExpMatrix{A} &\egaldef \max\left(1,\normesub{A}\right)^{n-1} \ConstanteInegTriangInverse{A}\\
& \text{where $\normesub{.}$ is the subordinate norm with respect to the Euclidean norm}
\end{split} \\
\label{e.ConstanteControleExpMatrixNormeII}
\ConstanteControleExpMatrixNormeII{A} &\egaldef 
2^{2n-2} (n-1)^{n-1} \ConstanteControleExpMatrix{A} 
\end{align}
\end{subequations}

Given a Riemannian manifold $M$ with distance $d$, a smooth vector field $\FixedVF$ on $M$ with flow $\FixedVF^t$ and a singularity $x$ of $\FixedVF$, we define the following \emph{stable} sets:
\begin{subequations}
\begin{itemize}
\item The \motdef{global stable set} $\SM{x,X}$ of $x$ for $\FixedVF$ is the set of points in $M$ whose forward orbit under the flow of $\FixedVF$ converge to $x$, that is,
\begin{equation}\label{e.global-stable-set}
\SM{x,X}= \MyBigSet{y \in M}{ \lim_{{t \to +\infty}} d\left(X^t(y),x\right) = 0}
\end{equation}
\item For any $\ExpDecaySpeed <0$, the \motdef{global $\ExpDecaySpeed$-stable set} $\SMexp{x,X}{\ExpDecaySpeed}$ of $x$ for $\FixedVF$ is the set of points in $M$ whose forward orbit under the flow of $\FixedVF$ converge to $x$ at least as fast as $e^{\ExpDecaySpeed t}$, that is,
\begin{equation}
\SMexp{x,X}{\ExpDecaySpeed} = \MyBigSet{y \in M}{d\left(X^t(y),x\right) = O_{t \to +\infty}\left(e^{\ExpDecaySpeed t}\right)}
\end{equation}
\item For any $\eta >0$, the \motdef{local stable set} $\LSM{x,X}{\eta}$ of $x$ for $\FixedVF$ is the set of points in $\SM{x,X}$ whose forward orbit under the flow of $\FixedVF$ stay in the $\eta$-neighbourhood of $x$, that is,
\begin{equation}
\LSM{x,X}{\eta} = \MyBigSet{y \in \SM{x,X}}{\forall t \geq 0, \: d\left(X^t(y),x\right) < \eta}
\end{equation}
\item For any $\ExpDecaySpeed <0$, for any $\eta >0$, the \motdef{local $\ExpDecaySpeed$-stable set} $\LSMexp{x,X}{\ExpDecaySpeed}{\eta}$ of $x$ for $\FixedVF$ is the set of points in $\SMexp{x,X}{\ExpDecaySpeed}$ whose forward orbit under the flow of $\FixedVF$ stay in the $\eta$-neighbourhood of $x$, that is,
\begin{equation}\label{e.local-gamma-stable-manifold}
\LSMexp{x,X}{\ExpDecaySpeed}{\eta} = \MyBigSet{y \in \SMexp{x,X}{\ExpDecaySpeed}}{\forall t \geq 0, \: d\left(X^t(y),x\right) < \eta}
\end{equation}
\end{itemize}
\end{subequations}
One can remark that if one chooses a distance $d'$ equivalent to $d$, then the stable sets for $d'$ coincide with the stable sets for $d$.

\section{Estimates for the stable manifold theorem with parameters}
\label{section.stable-manifold-theorem-notations-setup}
\subsection{Setup}
Fix an integer $\DimPhaseSpace \geq 2$ and an integer $\DimSetParameter \in \N^*$.
We define a \motdef{smooth family of vector fields $(\VFParam_\Parameter)_{\Parameter \in \SetParameter}$} as a smooth map
\begin{equation*}
\fonction{\VFParam}{\PhaseSpaceVF \times \SetParameter}{\PhaseSpaceVF}{(\VariablePhaseSpace,\Parameter)}{\VFParam_\Parameter(\VariablePhaseSpace)}
\end{equation*}
where $\PhaseSpaceVF$ is the phase space and $\SetParameter$ is the set of parameters.
Given such a $\VFParam$, let us consider some hypotheses:
\begin{hypo} \label{hypo.singularity}
	For every $\Parameter \in \SetParameter$, the origin is a singularity of $\VFParam_\Parameter$, \emph{i.e.}
	\begin{equation*}
	\VFParam_\Parameter(0)=0
	\end{equation*} 
\end{hypo}

\begin{hypo} \label{hypo.PHS}
The endomorphism $\OriginDiff:=\DiffPartielPoint{\VFParam}{0,0}{\VariablePhaseSpace}$ admits a \motdef{partially hyperbolic splitting} $(\StableSpace,\UnstableSpace)$, \emph{i.e.} there exists a non trivial decomposition $R^n = F \oplus G$ such that $F$ and $G$ are stabilized by $\OriginDiff$ and
\begin{equation*}
\SpectreUp{\OriginDiff_{|\StableSpace}} < \min \left(0 , \SpectreDown{\OriginDiff_{|\UnstableSpace}}\right)
\end{equation*}
\end{hypo}
Given such a partially hyperbolic splitting, we will consider the interval:
\begin{equation}
\label{e.exp-decay-speed-interval}
\ExpDecaySpeedInterval \egaldef \OpenInterval{\SpectreUp{\OriginDiff_{|\StableSpace}}}{\min\left(0,\SpectreDown{\OriginDiff_{|\UnstableSpace}}\right)} 
\end{equation}
and the "spectral gap":
\begin{equation}
	\label{e.gap-PHS}
	\gphs{\OriginDiff} \egaldef \min\left(1,\min\left(0,\SpectreDown{\OriginDiff_{|\UnstableSpace}}\right)-\SpectreUp{\OriginDiff_{|\StableSpace}}\right)^{-(\DimPhaseSpace-1)}
\end{equation}
\begin{hypo} \label{hypo.PHS-spectral-gap}
	Given a partially hyperbolic splitting $(\StableSpace,\UnstableSpace)$, the first derivative of $\VFParam$ satisfies
	\begin{equation*}
\sup_{(\VariablePhaseSpace,\Parameter) \in \PhaseSpaceVF \times \SetParameter} \normesub{\DiffPartiel{\VFParam}{\VariablePhaseSpace}(\VariablePhaseSpace,\Parameter)-\OriginDiff}  \leq \left(2^{3\DimPhaseSpace-1}(\DimPhaseSpace-1)^{\DimPhaseSpace-1}\sqrt{2} \: \ConstanteControleExpMatrix{\OriginDiff}\gphs{\OriginDiff}\right)^{-1}
	\end{equation*}
\end{hypo}

\begin{hypo} \label{hypo.Hyp-derivatives-bounded}
	The derivatives of all orders of $\VFParam$ are bounded, \emph{i.e.} for every $k \geq 1$,
	\begin{equation*}
\sup_{(\VariablePhaseSpace,\Parameter)\in \PhaseSpaceVF \times \SetParameter} \normesub{\DiffPartielPointSup{\VFParam}{\VariablePhaseSpace,\Parameter}{k}{\VariablePhaseSpace,\Parameter}} < + \infty
	\end{equation*}
\end{hypo}
In section~\ref{subsection.global-estimates}, we will assume that $\VFParam$ satisfies all the above hypotheses and we will prove a global stable manifold theorem with global estimates while in section~\ref{section.local-estimates}, we will only assume that the first two hypotheses hold true and we will prove a local stable manifold theorem, expliciting the local estimates and the size of the neighbourhood where these estimates hold true. The local theorem will be a consequence of the global one. The idea is to multiply the non linear part of $\VFParam$ by a smooth plateau map on a small neighbourhood of $(0,0)$ such that the new $\VFParam$ satisfies all the above hypotheses.

\subsection{Global estimates}
\label{subsection.global-estimates}
In this section, we state and prove a (global) stable manifold theorem with parameters for smooth families of vector fields $(\VFParam_\Parameter)_{\Parameter \in \SetParameter}$ satisfying the hypotheses~\ref{hypo.singularity},~\ref{hypo.PHS},~\ref{hypo.PHS-spectral-gap} and~\ref{hypo.Hyp-derivatives-bounded}. For such a $\VFParam$, let
\begin{equation*}
\NormeSupHigherDerivativesBounded{1}{\VFParam}\egaldef \sup_{(\VariablePhaseSpace,\Parameter) \in \PhaseSpaceVF \times \SetParameter} \normesub{\DiffPartiel{\VFParam}{\VariablePhaseSpace}(\VariablePhaseSpace,\Parameter)-\OriginDiff} 
\end{equation*}
where $\OriginDiff:=\DiffPartielPoint{\VFParam}{0,0}{\VariablePhaseSpace}$,
and for every integer $k \geq 2$, let
\begin{equation*}
\NormeSupHigherDerivativesBounded{k}{\VFParam} \egaldef  \sup_{2 \leq j \leq k} \sup_{(\VariablePhaseSpace,\Parameter)\in \PhaseSpaceVF \times \SetParameter} \normesub{\DiffPartielPointSup{\VFParam}{\VariablePhaseSpace,\Parameter}{j}{\VariablePhaseSpace,\Parameter}}
\end{equation*}
and
\begin{equation*}
	\NormeSupHigherDerivativesBoundedMax{k}{\VFParam} \egaldef \max\left(1,\NormeSupHigherDerivativesBounded{k}{\VFParam}\right)
\end{equation*}
Let us recall that in the current context, for any $\Parameter \in \SetParameter$ and $\ExpDecaySpeed <0$, the global $\ExpDecaySpeed$-stable set of $0$ for $\VFParam_\Parameter$ is
\begin{equation}\label{e.gamma-global-stable-set-context-VF-parameter}
\SMexp{0,\VFParam_\Parameter}{\ExpDecaySpeed}=\MyBigSet{\VariablePhaseSpace_0\in \PhaseSpaceVF}{\norme{\VFParam_\Parameter^t(\VariablePhaseSpace_0)}=O_{t \to +\infty}\left(e^{\ExpDecaySpeed t}\right)}
\end{equation}
For any $\ExpDecaySpeed <0$, let
\begin{equation*}
\DistanceBordSpectre{\OriginDiff}{\ExpDecaySpeed} \egaldef \min\bigl(1, d\left(\ExpDecaySpeed, \RealPart (\Sp_{\C}(\OriginDiff))\right)\bigr)^{\DimPhaseSpace-1}
\end{equation*}
where $d$ is the usual distance on $\R$.

\begin{prop}[Global estimates for the stable manifold theorem with parameters]\label{prop.variete-gamma-stable-global-version}
There exists a positive constant $\ConstantEstimateStableManifoldThm$ and a sequence of positive constants $(\ConstI[k])_{k \in \N}$ (depending only on the dimension $\DimPhaseSpace$ of the phase space) such that for every smooth family of vector fields $(\VFParam_\Parameter)_{\Parameter \in \SetParameter}$ satisfying the hypotheses~\ref{hypo.singularity},~\ref{hypo.PHS},~\ref{hypo.PHS-spectral-gap} and~\ref{hypo.Hyp-derivatives-bounded}, for every partially hyperbolic splitting $(\StableSpace,\UnstableSpace)$ of $\OriginDiff:= \DiffPartielPoint{\VFParam}{0,0}{\VariablePhaseSpace}$, there exists a unique smooth map
\begin{equation*}
	\fonction{\SMTGraphMap}{\StableSpace \times \SetParameter}{\UnstableSpace}{(\VarPhaseSpaceStable,\Parameter)}{\SMTGraphMap_\Parameter(\VarPhaseSpaceStable)}
\end{equation*}
such that
\begin{enumerate}
	\item\label{item.graph-structure-global-estimate} Graph structure of the global $\ExpDecaySpeed$-stable set: for every $\Parameter \in \SetParameter$, for every $\ExpDecaySpeed \in \ExpDecaySpeedInterval$ (see~\eqref{e.exp-decay-speed-interval}) satisfying
\begin{equation} \label{e-hyp-norme-differential-exp-decay}
\NormeSupHigherDerivativesBounded{1}{\VFParam}\leq \frac{1}{\ConstI} \frac{\DistanceBordSpectre{\OriginDiff}{\ExpDecaySpeed}}{\ConstanteControleExpMatrix{\OriginDiff}} 
\end{equation}
where $\ConstI = 2^{2\DimPhaseSpace}(\DimPhaseSpace-1)^{\DimPhaseSpace-1} \sqrt{2}$, the stable set $\SMexp{0,\VFParam_\Parameter}{\ExpDecaySpeed}$ is exactly the graph of the map $\SMTGraphMap_\Parameter : \StableSpace \to \UnstableSpace$. In particular, $\SMexp{0,\VFParam_\Parameter}{\ExpDecaySpeed}$ does not depend on the choice of such a $\ExpDecaySpeed$.

\item\label{item.local-stable-set-cas-global} Local $\ExpDecaySpeed$-stable set: for every $\ExpDecaySpeed \in \ExpDecaySpeedInterval$ satisfying~\eqref{e-hyp-norme-differential-exp-decay}, for every $\Parameter \in \SetParameter$, for every $ \eta >0$, for every $0 < \delta \leq \frac{\eta}{\ConstantEstimateStableManifoldThm \ConstanteControleExpMatrix{\OriginDiff} \gphs{\OriginDiff}}$, we have
\begin{equation}\label{e.local-stable-set-cas-global}
	\LSMexp{0,\VFParam_\Parameter}{\ExpDecaySpeed}{\eta} \cap B_{\PhaseSpaceVF}(0,\delta) = \Graph \left(\SMTGraphMapLocal_\Parameter\right) \cap B_{\PhaseSpaceVF}(0,\delta)
\end{equation}

\item \label{item.estimates-SMT-graph-global-estimates} Controls on $\SMTGraphMap$: for every $(\VarPhaseSpaceStable,\Parameter) \in \StableSpace \times \SetParameter$,
	\begin{subequations}
		\label{e.estimates-SMT-graph-global-estimates}
		\begin{align}
		\label{e-controle-phi-0}
		\norme{\SMTGraphMap(\VarPhaseSpaceStable,\Parameter)} &\leq \ConstI[0] \gphs{\OriginDiff}\ConstanteControleExpMatrix{\OriginDiff} \NormeSupHigherDerivativesBounded{1}{\VFParam} \norme{\VarPhaseSpaceStable} \\
		\label{e-controle-phi-1-derivative-stable-variable}
		\normesub{\DiffPartielPoint{\SMTGraphMap}{\VarPhaseSpaceStable,\Parameter}{z}} &\leq \ConstI[1]
		\gphs{\OriginDiff}
		\ConstanteControleExpMatrix{\OriginDiff} \NormeSupHigherDerivativesBounded{1}{\VFParam}\\
		\label{e-controle-phi-1-derivative-parameter-variable}
		\normesub{\DiffPartielPoint{\SMTGraphMap}{\VarPhaseSpaceStable,\Parameter}{\Parameter}} &\leq
		\ConstI[1]
		\gphs{\OriginDiff}
		\ConstanteControleExpMatrix{\OriginDiff} \NormeSupHigherDerivativesBounded{2}{\VFParam} \norme{\VarPhaseSpaceStable}
		\end{align}
		and more generally, using the norm $\norme{(\VarPhaseSpaceStable,\Parameter)}=\norme{\VarPhaseSpaceStable}+\norme{\Parameter}$ on $\StableSpace \times \SetParameter$, we have, for all $k \geq 2$,
		\begin{equation}\label{e-controle-phi-2}
		\normesub{\DiffPointSup{\SMTGraphMap}{\VarPhaseSpaceStable,\Parameter}{k}} \leq
		\ConstI[k] \left(\gphs{\OriginDiff}^2\ConstanteControleExpMatrix{\OriginDiff}^2\NormeSupHigherDerivativesBoundedMax{k+1}{\VFParam} \max\left(1,\norme{\VarPhaseSpaceStable}\right)  \right)^{2k-1}
		\end{equation}
	\end{subequations}
	where $\gphs{\OriginDiff}$ is defined by~\eqref{e.gap-PHS}.
\end{enumerate}
\end{prop}

\begin{rem}
	If one is working with a different norm than the Euclidean one, one will have the same result but with different constants $\ConstI,\ConstI[0], \ConstI[1],\dots$
\end{rem}

\begin{rem}
	Hypothesis~\ref{hypo.PHS-spectral-gap} is not fundamentally necessary for proposition~\ref{prop.variete-gamma-stable-global-version} to be true. This hypothesis implies that there exists a $\ExpDecaySpeed \in \ExpDecaySpeedInterval$ satisfying~\eqref{e-hyp-norme-differential-exp-decay} in item~\eqref{item.graph-structure-global-estimate}, so it is only a convenient and explicit sufficient condition for the proposition to not be empty. When proving the local version in section~\ref{section.local-estimates}, we will not check that hypothesis~\ref{hypo.PHS-spectral-gap} holds true, we will directly work with a given $\ExpDecaySpeed$ and check that~\eqref{e-hyp-norme-differential-exp-decay} holds true.
\end{rem}
The proof of proposition~\ref{prop.variete-gamma-stable-global-version} is heavily based on the contraction mapping theorem, applied in the Banach space introduced in definition~\ref{defin.function-space} below.

\begin{defin}[$\ExpDecaySpeed$-norm]
	For any $\ExpDecaySpeed \in \R$ and $(\VarPhaseSpaceStable,\VarPhaseSpaceUnstable):[0,+\infty[ \to \StableSpaceStraight \times \UnstableSpaceStraight$, we define the $\ExpDecaySpeed$-norm of $(\VarPhaseSpaceStable,\VarPhaseSpaceUnstable)$ by
	\begin{equation*}
	\norme{(\VarPhaseSpaceStable,\VarPhaseSpaceUnstable)}_\ExpDecaySpeed \Opegaldef \sup_{t \geq 0} \max(\norme{z(t)},\norme{v(t)}) e^{-\ExpDecaySpeed t} \in [0,+\infty]
	\end{equation*}
\end{defin}

\begin{defin}[Function space $\FunctionSpaceGeneric{\ExpDecaySpeed}$]
	\label{defin.function-space}
	Let $\ExpDecaySpeed \in \R$. Denote by $\FunctionSpaceGeneric{\ExpDecaySpeed}$ the vector space of continuous maps $(\VarPhaseSpaceStable,\VarPhaseSpaceUnstable): [0, +\infty[ \to \StableSpaceStraight \times \UnstableSpaceStraight$ whose $\ExpDecaySpeed$-norm are finite. The vector space $\FunctionSpaceGeneric{\ExpDecaySpeed}$ endowed with the $\ExpDecaySpeed$-norm is a Banach space.
\end{defin}

\begin{rem} \label{rem.espaces-H}
	For any $ \ExpDecaySpeed < \ExpDecaySpeed' $, we have $\FunctionSpaceGeneric{\ExpDecaySpeed} \subset \FunctionSpaceGeneric{\ExpDecaySpeed'}$ and for every $(\VarPhaseSpaceStable,\VarPhaseSpaceUnstable) \in \FunctionSpaceGeneric{\ExpDecaySpeed}$, we have $\norme{(\VarPhaseSpaceStable,\VarPhaseSpaceUnstable)}_{\ExpDecaySpeed'} \leq \norme{(\VarPhaseSpaceStable,\VarPhaseSpaceUnstable)}_\ExpDecaySpeed$.
\end{rem}

\begin{rem}
	It will be useful to see $\FunctionSpace{\ExpDecaySpeed}{\DimPhaseSpace}:= \FunctionSpaceGeneric{\ExpDecaySpeed}$ as the cartesian product $\FunctionSpace{\ExpDecaySpeed}{\DimStableSpaceStraight} \times \FunctionSpace{\ExpDecaySpeed}{\DimUnstableSpaceStraight}$ when $\DimPhaseSpace=\DimStableSpaceStraight+\DimUnstableSpaceStraight$.
\end{rem}

\begin{proof}[Proof of proposition~\ref{prop.variete-gamma-stable-global-version}]
Before expliciting the strategy of the proof, we need some preparatory work, stated below.
Fix a smooth family of vector fields $(\VFParam_\Parameter)_{\Parameter \in \SetParameter}$ satisfying the hypotheses~\ref{hypo.singularity},~\ref{hypo.PHS},~\ref{hypo.PHS-spectral-gap} and~\ref{hypo.Hyp-derivatives-bounded} and a partially hyperbolic splitting $(\StableSpace,\UnstableSpace)$ of $\OriginDiff:=\DiffPartielPoint{\VFParam}{0,0}{\VariablePhaseSpace}$.
Let $\DimStableSpaceStraight = \dim \StableSpace$ and $\DimUnstableSpaceStraight = \dim \UnstableSpace$.
Fix $\ExpDecaySpeed \in \ExpDecaySpeedInterval$ satisfying
\begin{equation} \label{e-hyp-norme-differential-exp-decay-Euclidean-norm}
\NormeSupHigherDerivativesBounded{1}{\VFParam}\leq \frac{1}{2^{2\DimPhaseSpace}(\DimPhaseSpace-1)^{\DimPhaseSpace-1} \sqrt{2}} \frac{\DistanceBordSpectre{\OriginDiff}{\ExpDecaySpeed}}{\ConstanteControleExpMatrix{\OriginDiff}}
\end{equation}
\proofstep{Conjugation of $\VFParam$}
We start by conjugating $\VFParam$ in such a way that $\StableSpace$ and $\UnstableSpace$ become orthogonal linear subspaces of $\PhaseSpaceVF$.
For that purpose, let us fix an isomorphism $\PHSIsoStraight:\PhaseSpaceVF \to \PhaseSpaceVF \simeq \StableSpaceStraight \times \UnstableSpaceStraight$ such that $\PHSIsoStraight_{|\StableSpace}$ (resp. $\PHSIsoStraight_{|\UnstableSpace}$) is an isometry from $\StableSpace$ (resp. $\UnstableSpace$) to $\StableSpaceStraight \times \{0\} \simeq \StableSpaceStraight$ (resp. $\{0\} \times \UnstableSpaceStraight \simeq \UnstableSpaceStraight$).
We have (using lemma~\ref{lemma.triangular-inequality-inverse-angle-two-spaces})
\begin{equation}\label{e.norme-conjuguaison}
\norme{\PHSIsoStraight} \leq \ConstanteInegTriangInverse{\StableSpace,\UnstableSpace} ,\quad \norme{\PHSIsoStraight^{-1}} \leq \sqrt{2}
\end{equation}
We now define
\begin{equation}\label{e.defin-VFParam-Straight}
\fonction{\VFParamStraight}{\StableSpaceStraight \times \UnstableSpaceStraight \times \SetParameter}{\StableSpaceStraight \times \UnstableSpaceStraight}{(\VarPhaseSpaceStable,\VarPhaseSpaceUnstable,\Parameter)}{\PHSIsoStraight\bigl(\VFParam_\Parameter(\PHSIsoStraight^{-1}(\VarPhaseSpaceStable,\VarPhaseSpaceUnstable))\bigr)}
\end{equation}
One can remark that for every $k \geq 1$, 
\begin{equation} \label{e.estimate-sup-derivatives-straight}
\NormeSupHigherDerivativesBounded{k}{\VFParamStraight}\leq \sqrt{2}^k \: \ConstanteInegTriangInverse{\StableSpace,\UnstableSpace} \NormeSupHigherDerivativesBounded{k}{\VFParam} 
\end{equation}
Let
\begin{equation*}
\OriginDiffStraight \egaldef \DiffPartielPoint{\VFParamStraight}{0,0,0}{\VarPhaseSpaceStable,\VarPhaseSpaceUnstable} = \PHSIsoStraight \OriginDiff \PHSIsoStraight^{-1}=
\begin{pmatrix}
\StableMatrixStraight & 0   \\
0 & \UnstableMatrixStraight \\
\end{pmatrix}
\end{equation*}
where $\StableMatrixStraight=\tilde{\LineartPart}_{|\StableSpaceStraight}$ and $\UnstableMatrixStraight=\tilde{\LineartPart}_{|\UnstableSpaceStraight}$, with respect to the canonical basis. Using the fact that $\PHSIsoStraight$ is an isometry in restriction to $\StableSpace$ and $\UnstableSpace$, we have the following properties on $\StableMatrixStraight$ and $\UnstableMatrixStraight$ :
\begin{subequations}
	\label{e.properties-linear-part-EDO}
	\begin{align}
	\label{e-spectre-conjuguated}
	\Sp_{\C}(\StableMatrixStraight)&= \Sp_{\C}(\StableMatrix),  &\Sp_{\C}(\UnstableMatrixStraight)&= \Sp_{\C}(\UnstableMatrix) \\
	\label{e.constantes-conjuguated}
	\ConstanteInegTriangInverse{\StableMatrixStraight} &= \ConstanteInegTriangInverse{\StableMatrix}, &\ConstanteInegTriangInverse{\UnstableMatrixStraight}&= \ConstanteInegTriangInverse{\UnstableMatrix} \\
	\label{e-constantes-conjuguated-2}
	\ConstanteControleExpMatrix{\StableMatrixStraight} &= \ConstanteControleExpMatrix{\StableMatrix},  &\ConstanteControleExpMatrix{\UnstableMatrixStraight} &= \ConstanteControleExpMatrix{\UnstableMatrix}
	\end{align}
	Property~\eqref{e-spectre-conjuguated} implies that $\DistanceBordSpectre{\OriginDiffStraight}{\ExpDecaySpeed}=\DistanceBordSpectre{\OriginDiff}{\ExpDecaySpeed}$.
\end{subequations}
\proofstep{Differential equation view-point}
Let $\Parameter \in \SetParameter$.
The differential equation associated with the vector field $\VFParamStraight_\Parameter$ can be written in the following form
\begin{equation} \label{e-EDO_initiale}
\begin{cases}
\VarPhaseSpaceStable' &= \StableMatrixStraight \VarPhaseSpaceStable + \NonLinearPartEDI(\VarPhaseSpaceStable,\VarPhaseSpaceUnstable,\Parameter) \\
\VarPhaseSpaceUnstable' &= \UnstableMatrixStraight \VarPhaseSpaceUnstable + \NonLinearPartEDII(\VarPhaseSpaceStable,\VarPhaseSpaceUnstable,\Parameter)
\end{cases}
\end{equation}
where $(\NonLinearPartEDI,\NonLinearPartEDII):\StableSpaceStraight \times \UnstableSpaceStraight \times \SetParameter \to \StableSpaceStraight \times \UnstableSpaceStraight$ is a smooth map defined by
\begin{equation*}
\begin{pmatrix}
\NonLinearPartEDI(\VarPhaseSpaceStable,\VarPhaseSpaceUnstable,\Parameter)   \\
\NonLinearPartEDII(\VarPhaseSpaceStable,\VarPhaseSpaceUnstable,\Parameter) \\
\end{pmatrix}
= \VFParamStraight(\VarPhaseSpaceStable,\VarPhaseSpaceUnstable,\Parameter) - \DiffPartielPoint{\VFParamStraight}{0,0}{\VarPhaseSpaceStable,\VarPhaseSpaceUnstable}.(\VarPhaseSpaceStable,\VarPhaseSpaceUnstable)
\end{equation*}
and satisfying
\begin{subequations}
\label{e.properties-non-linear-part-EDO}
	\begin{align}
	\label{e-ligne-critique}
	\forall \Parameter \in \SetParameter, \;  (\NonLinearPartEDI,\NonLinearPartEDII)(0,0,\Parameter) & = 0 \\
	\label{e-derivee-f-g-0}
	\DiffPoint{(\NonLinearPartEDI,\NonLinearPartEDII)}{0,0,0} &=0 \\
	\label{e.controle-derivee-premiere-f-g-sup}
\forall (\VarPhaseSpaceStable,\VarPhaseSpaceUnstable,\Parameter) \in \StableSpaceStraight \times \UnstableSpaceStraight \times \SetParameter, \normesub{\DiffPartielPoint{(\NonLinearPartEDI,\NonLinearPartEDII)}{\VarPhaseSpaceStable,\VarPhaseSpaceUnstable,\Parameter}{\VarPhaseSpaceStable,\VarPhaseSpaceUnstable}} &\leq \NormeSupHigherDerivativesBounded{1}{\VFParamStraight} \\
	\label{e-controle-derivee-f-g-sup}
	\forall N \geq 2, \forall \: 2 \leq k \leq N, \; \forall (\VarPhaseSpaceStable,\VarPhaseSpaceUnstable,\Parameter) \in \StableSpaceStraight \times \UnstableSpaceStraight \times \SetParameter,  \normesub{\DiffPointSup{(\NonLinearPartEDI,\NonLinearPartEDII)}{\VarPhaseSpaceStable,\VarPhaseSpaceUnstable,\Parameter}{k}} & \leq \NormeSupHigherDerivativesBounded{N}{\VFParamStraight}
	\end{align}
Property~\eqref{e-ligne-critique} implies that
	\begin{equation} \label{e-controle-derivee-f-g-mu}
	\forall \Parameter \in \SetParameter, \forall k \in \N, \; D^k_{\Parameter}(\NonLinearPartEDI,\NonLinearPartEDII)(0,0,\Parameter) =0 
	\end{equation}
\end{subequations}
\proofstep{Control of exponential matrices}
We now state an estimate that will be used several times throughout this proof.
Let
\begin{equation*}
\BorneVPStable = \frac{\ExpDecaySpeed+\SpectreUp{\StableMatrix}}{2}, \quad \BorneVPUnstable = \frac{\ExpDecaySpeed+\SpectreDown{\UnstableMatrix}}{2}
\end{equation*}
According to lemma~\ref{lemma.controle-polynomial-exp-matrices} and~\eqref{e-constantes-conjuguated-2}, we have, for every $s \geq 0$,
\begin{equation}\label{e.controle-polynomial-matrice-exp}
\begin{split}
\normesub{e^{s\StableMatrixStraight}} &\leq  \frac{\ConstanteControleExpMatrixNormeII{\StableMatrix}}{\DistanceBordSpectre{\OriginDiff}{\ExpDecaySpeed}^{\frac{\DimStableSpaceStraight-1}{\DimPhaseSpace-1}}} e^{\BorneVPStable s}	 \\
\normesub{e^{-s\UnstableMatrixStraight}} &\leq 	\frac{\ConstanteControleExpMatrixNormeII{\UnstableMatrix}}{\DistanceBordSpectre{\OriginDiff}{\ExpDecaySpeed}^{\frac{\DimUnstableSpaceStraight-1}{\DimPhaseSpace-1}}} e^{-\BorneVPUnstable s}
\end{split}
\end{equation}
where $\ConstanteControleExpMatrixNormeII{.}$ is defined by~\eqref{e.ConstanteControleExpMatrixNormeII}. Beware of the fact that the integer $\DimPhaseSpace$ must be replaced by $\DimStableSpaceStraight$ (resp. $\DimUnstableSpaceStraight$) for $\ConstanteControleExpMatrixNormeII{\StableMatrix}$ (resp. $\ConstanteControleExpMatrixNormeII{\UnstableMatrix}$).
\proofstep{Main operator of the proof}
Let us define the operator
\begin{equation*}
\fonction{\OpFamilyGeneric{\ExpDecaySpeed}}{\FunctionSpaceGeneric{\ExpDecaySpeed} \times \StableSpaceStraight \times \SetParameter}{\FunctionSpaceGeneric{\ExpDecaySpeed}}{((\VarPhaseSpaceStable,\VarPhaseSpaceUnstable),\VarConditionInitiale,\Parameter)}{\OpFamily{\ExpDecaySpeed}{\VarConditionInitiale,\Parameter}(\VarPhaseSpaceStable,\VarPhaseSpaceUnstable)}
\end{equation*}
by the formula
\begin{equation*}
\OpFamily{\ExpDecaySpeed}{\VarConditionInitiale,\Parameter}(\VarPhaseSpaceStable,\VarPhaseSpaceUnstable)(t)=
\begin{pmatrix}
e^{t \StableMatrixStraight}\VarConditionInitiale +\Integral{0}{t}{e^{(t-s)\StableMatrixStraight}\NonLinearPartEDI(\VarPhaseSpaceStable(s),\VarPhaseSpaceUnstable(s),\Parameter)}{s} \\
- \Integral{t}{+\infty}{e^{-(s-t)\UnstableMatrixStraight} \NonLinearPartEDII(\VarPhaseSpaceStable(s),\VarPhaseSpaceUnstable(s),\Parameter)}{s}
\end{pmatrix}
\end{equation*}
\proofstep{Strategy of the proof}
Fix $\Parameter \in \SetParameter$.
We want to prove that the global $\ExpDecaySpeed$-stable set $\SMexp{0,\VFParamStraight_\Parameter}{\ExpDecaySpeed}$ is a graph over $\StableSpaceStraight$. This amounts to prove that for every $\VarConditionInitiale \in \StableSpaceStraight$, there exists a unique $\VarPhaseSpaceUnstable_0 \in \UnstableSpaceStraight$ such that $(\VarConditionInitiale,\VarPhaseSpaceUnstable_0) \in \SMexp{0,\VFParamStraight_\Parameter}{\ExpDecaySpeed}$. This is also equivalent to say that for every $\VarConditionInitiale \in \StableSpaceStraight$, there exists a unique solution $(\VarPhaseSpaceStable,\VarPhaseSpaceUnstable)$ of~\eqref{e-EDO_initiale} such that $\VarPhaseSpaceStable(0)= \VarConditionInitiale$ and $(\VarPhaseSpaceStable,\VarPhaseSpaceUnstable) \in \FunctionSpaceGeneric{\ExpDecaySpeed}$.
We introduced the operator $\OpFamily{\ExpDecaySpeed}{\VarConditionInitiale,\Parameter}$ because its fixed points are exactly the solutions $(\VarPhaseSpaceStable,\VarPhaseSpaceUnstable)$ of~\eqref{e-EDO_initiale} such that $\VarPhaseSpaceStable(0)= \VarConditionInitiale$ and $(\VarPhaseSpaceStable,\VarPhaseSpaceUnstable) \in \FunctionSpaceGeneric{\ExpDecaySpeed}$ (see lemma~\ref{lemma.equivalence-stable-set-graph}).
It is then enough to prove that $\OpFamily{\ExpDecaySpeed}{\VarConditionInitiale,\Parameter}$ admits a unique fixed point in $\FunctionSpaceGeneric{\ExpDecaySpeed}$, denoted by $(\FPStable{\VarConditionInitiale}{\Parameter},\FPUnstable{\VarConditionInitiale}{\Parameter})$ (see lemma~\ref{lemma.FP-smooth-estimates}).
See figure~\ref{fig.orbits-in-stable-manifold}.
The estimates on the graph follow from estimates on $\FPUnstable{\VarConditionInitiale}{\Parameter}$ (see lemma~\ref{lemma.FP-smooth-estimates}) which themselves follow from estimates on $\OpFamilyGeneric{\ExpDecaySpeed}$ (see lemma~\ref{lemma.smoothness-FP}).
\begin{figure}
	\begin{center}
		\begin{tikzpicture}[scale=2]
		\draw[->] (-1/2,0) -- (2,0);
		\draw (2,0) node[right] {$\StableSpaceStraight$};
		\draw [->] (0,-1/2) -- (0,2);
		\draw (0,2) node[above] {$\UnstableSpaceStraight$};
		\draw (0,0) node[below left] {$0$};
		\draw[red,directed] (2,1.5) to[bend left=37] (0,0);
		\draw (1.5,0.6) node {$\bullet$};
		\draw (1.5,0.6) node[right] {$(\FPStable{\VarConditionInitiale}{\Parameter}(0),\FPUnstable{\VarConditionInitiale}{\Parameter}(0))$};
		\draw[dashed] (1.5,0) -- (1.5,2);
		\draw (1.5,0) node[below] {$\VarConditionInitiale$};
		\end{tikzpicture}
	\end{center}
	\caption{$(\FPStable{\VarConditionInitiale}{\Parameter},\FPUnstable{\VarConditionInitiale}{\Parameter})$ is the unique orbit of $\VFParamStraight_\Parameter$ contained in the global $\ExpDecaySpeed$-stable set $\SMexp{0,\VFParamStraight_\Parameter}{\ExpDecaySpeed}$ with initial condition of the form $(\VarConditionInitiale,\VarPhaseSpaceUnstable_0)$, $\VarPhaseSpaceUnstable_0 \in \UnstableSpaceStraight$.}
	\label{fig.orbits-in-stable-manifold}
\end{figure}
\proofstep{Technical details of the proof} We now state and prove three lemmas which constitute the main part of the proof.

For every $k \geq 0$, we denote by $\mathcal{L}_k \left(\FunctionSpaceGeneric{\ExpDecaySpeed}\times \SetParameter,\FunctionSpaceGeneric{\ExpDecaySpeed}\right)$ the space of $k$-linear maps from $\left(\FunctionSpaceGeneric{\ExpDecaySpeed}\times \SetParameter\right)^k$ to $\FunctionSpaceGeneric{\ExpDecaySpeed}$ and we define the operator
\begin{equation*}
	\OpFamilyII{k}: \FunctionSpaceGeneric{\ExpDecaySpeed}\times \SetParameter \to \mathcal{L}_k \left(\FunctionSpaceGeneric{\ExpDecaySpeed}\times \SetParameter,\FunctionSpaceGeneric{\ExpDecaySpeed}\right)
\end{equation*}
by the following formula: for every $((\VarPhaseSpaceStable,\VarPhaseSpaceUnstable),\Parameter) \in \FunctionSpaceGeneric{\ExpDecaySpeed}\times \SetParameter$, $((z_i,v_i),\Parameter_i)_{1 \leq i \leq k} \in \left(\FunctionSpaceGeneric{\ExpDecaySpeed}\times \SetParameter\right)^k$, $t \geq 0$,
\begin{equation*}
	\OpFamilyII{k}((\VarPhaseSpaceStable,\VarPhaseSpaceUnstable),\Parameter).((z_i,v_i),\Parameter_i)(t) = 
	\begin{pmatrix}
	\Integral{0}{t}{e^{(t-s)\StableMatrixStraight}\DiffPointSup{\NonLinearPartEDI}{\VarPhaseSpaceStable(s),\VarPhaseSpaceUnstable(s),\Parameter}{k}.((z_i(s),v_i(s)),\Parameter_i)}{s} \\
	- \Integral{t}{+\infty}{e^{-(s-t)\UnstableMatrixStraight} \DiffPointSup{\NonLinearPartEDII}{\VarPhaseSpaceStable(s),\VarPhaseSpaceUnstable(s),\Parameter}{k}.((z_i(s),v_i(s)),\Parameter_i)}{s}
	\end{pmatrix}
\end{equation*}
We also define the operator
\begin{equation*}
	\fonction{\Gamma}{\StableSpaceStraight}{\FunctionSpaceGeneric{\ExpDecaySpeed}}{\VarConditionInitiale}{\left[t \mapsto
		\begin{pmatrix}
		e^{t \StableMatrixStraight}\VarConditionInitiale \\
		0
		\end{pmatrix}	
		\right]}
\end{equation*}

\begin{lemme}\label{lemma.smoothness-FP}
The operator $\OpFamilyGeneric{\ExpDecaySpeed}$ is smooth.

For all $k \geq 1$, $((\VarPhaseSpaceStable,\VarPhaseSpaceUnstable),\VarConditionInitiale,\Parameter) \in \FunctionSpaceGeneric{\ExpDecaySpeed}\times \StableSpaceStraight \times \SetParameter$, $((z_i,v_i),\VarConditionInitiale_i,\Parameter_i)_{1 \leq i \leq k} \in \left(\FunctionSpaceGeneric{\ExpDecaySpeed}\times \StableSpaceStraight \times \SetParameter\right)^k$,
\begin{equation}\label{e.derivatives-operator-FP}
	\DiffPointSup{\OpFamilyGeneric{\ExpDecaySpeed}}{(\VarPhaseSpaceStable,\VarPhaseSpaceUnstable),\VarConditionInitiale,\Parameter}{k}. ((z_i,v_i),\VarConditionInitiale_i,\Parameter_i) = 
\begin{cases}
\Gamma(\VarConditionInitiale_1) + \OpFamilyII{1}((\VarPhaseSpaceStable,\VarPhaseSpaceUnstable),\Parameter).((z_1,v_1),\Parameter_1) & \text{if $k=1$}\\
\OpFamilyII{k}((\VarPhaseSpaceStable,\VarPhaseSpaceUnstable),\Parameter).((z_i,v_i),\Parameter_i) & \text{if $k \geq 2$}
\end{cases}
\end{equation}
Moreover, using the following norm on $\FunctionSpaceGeneric{\ExpDecaySpeed}\times \StableSpaceStraight \times \SetParameter$:
\begin{equation*}
	\norme{((\VarPhaseSpaceStable,\VarPhaseSpaceUnstable),\VarConditionInitiale,\Parameter)} = \norme{(\VarPhaseSpaceStable,\VarPhaseSpaceUnstable)}_\ExpDecaySpeed + \norme{\VarConditionInitiale}+\norme{\Parameter}
\end{equation*}
we have the following estimates : for every $((\VarPhaseSpaceStable,\VarPhaseSpaceUnstable),\VarConditionInitiale,\Parameter) \in \FunctionSpaceGeneric{\ExpDecaySpeed}\times \StableSpaceStraight \times \SetParameter$,
\begin{subequations}
	\label{e.estimate-operator-FP-all-cases}
\begin{align}
	\label{e.partial-derivarive-operator-bounded-1demi}
	\normesub{\DiffPartielPoint{\OpFamilyGeneric{\ExpDecaySpeed}}{(\VarPhaseSpaceStable,\VarPhaseSpaceUnstable),\VarConditionInitiale,\Parameter}{\VarPhaseSpaceStable,\VarPhaseSpaceUnstable}}_\ExpDecaySpeed &\leq \frac{1}{2} \\
	\label{e.partial-derivarive-operator-direction-cond-ini}
	\normesub{\DiffPartielPoint{\OpFamilyGeneric{\ExpDecaySpeed}}{(\VarPhaseSpaceStable,\VarPhaseSpaceUnstable),\VarConditionInitiale,\Parameter}{\VarConditionInitiale}}_\ExpDecaySpeed &\leq \frac{\ConstanteControleExpMatrixNormeII{\StableMatrix}}{\DistanceBordSpectre{\OriginDiff}{\ExpDecaySpeed}^{\frac{\DimStableSpaceStraight-1}{\DimPhaseSpace-1}}} \\
	\label{e.partial-derivarive-operator-direction-parameter}
	\normesub{\DiffPartielPoint{\OpFamilyGeneric{\ExpDecaySpeed}}{(\VarPhaseSpaceStable,\VarPhaseSpaceUnstable),\VarConditionInitiale,\Parameter}{\Parameter}}_\ExpDecaySpeed &\leq \frac{2\max\left(\ConstanteControleExpMatrixNormeII{\StableMatrix},\ConstanteControleExpMatrixNormeII{\UnstableMatrix}\right)}{\DistanceBordSpectre{\OriginDiff}{\ExpDecaySpeed}}\NormeSupHigherDerivativesBounded{2}{\VFParamStraight} \norme{(\VarPhaseSpaceStable,\VarPhaseSpaceUnstable)}_\ExpDecaySpeed
\end{align}
where $\normesub{.}_\ExpDecaySpeed$ denotes the standard norm of continuous linear maps from $\FunctionSpaceGeneric{\ExpDecaySpeed}$ (resp. $\StableSpaceStraight$, resp. $\SetParameter$) to $\FunctionSpaceGeneric{\ExpDecaySpeed}$ and, more generally, for every $k \geq 2$,
\begin{equation}\label{e.estimates-derivatives-operatorFP-high-order}
\normesub{\DiffPointSup{\OpFamilyGeneric{\ExpDecaySpeed}}{(\VarPhaseSpaceStable,\VarPhaseSpaceUnstable),\VarConditionInitiale,\Parameter}{k}}_\ExpDecaySpeed \leq 
	\frac{2\max\left(\ConstanteControleExpMatrixNormeII{\StableMatrix},\ConstanteControleExpMatrixNormeII{\UnstableMatrix}\right)}{\DistanceBordSpectre{\OriginDiff}{\ExpDecaySpeed}}   \left(\NormeSupHigherDerivativesBounded{k}{\VFParamStraight} + \NormeSupHigherDerivativesBounded{k+1}{\VFParamStraight} \norme{(\VarPhaseSpaceStable,\VarPhaseSpaceUnstable)}_\ExpDecaySpeed \right)
\end{equation}
\end{subequations}
where $\normesub{.}_\ExpDecaySpeed$ denotes the standard norm of continuous $k$-linear maps from $\left(\FunctionSpaceGeneric{\ExpDecaySpeed}\times \StableSpaceStraight \times \SetParameter\right)^k$ to $\FunctionSpaceGeneric{\ExpDecaySpeed}$.
\end{lemme}

\begin{proof}[Proof of lemma~\ref{lemma.smoothness-FP}]
\setcounter{MyCounterProof}{1}
One can remark that $\OpFamilyGeneric{\ExpDecaySpeed}$ is the sum of two operators, the first one being the linear map $\Gamma$ and the second one being $\OpFamilyII{0}$. Since $\StableSpaceStraight$ is a finite dimensional vector space, $\Gamma$ is smooth. It follows that we only need to prove that the operator $\OpFamilyII{0}: \FunctionSpaceGeneric{\ExpDecaySpeed} \times \SetParameter \to \FunctionSpaceGeneric{\ExpDecaySpeed}$ is smooth to prove the first part of the lemma.
Using the classical algebraic identification
\begin{equation*}
	\mathcal{L}_{k+1} \left(\FunctionSpaceGeneric{\ExpDecaySpeed}\times \SetParameter,\FunctionSpaceGeneric{\ExpDecaySpeed}\right) \simeq \mathcal{L}\left(\FunctionSpaceGeneric{\ExpDecaySpeed}\times \SetParameter,\mathcal{L}_k \left(\FunctionSpaceGeneric{\ExpDecaySpeed}\times \SetParameter,\FunctionSpaceGeneric{\ExpDecaySpeed}\right)\right)
\end{equation*}
we are going to prove that for every $k \geq 0$, $\OpFamilyII{k}$ is differentiable and $\Diff{\OpFamilyII{k}} = \OpFamilyII{k+1}$.

\proofstep{Step \theMyCounterProof: for every $k \geq 0$, $\OpFamilyII{k}$ is well defined and for every $k \geq 1$, for every $((\VarPhaseSpaceStable,\VarPhaseSpaceUnstable),\Parameter) \in \FunctionSpaceGeneric{\ExpDecaySpeed}\times \SetParameter$, $\OpFamilyII{k}((\VarPhaseSpaceStable,\VarPhaseSpaceUnstable),\Parameter)$ is a continuous $k$-linear map}
Let $k \geq 0$, $((\VarPhaseSpaceStable,\VarPhaseSpaceUnstable),\Parameter) \in \FunctionSpaceGeneric{\ExpDecaySpeed}\times \SetParameter$ and $((z_i,v_i),\Parameter_i)_{1 \leq i \leq k} \in \left(\FunctionSpaceGeneric{\ExpDecaySpeed}\times \SetParameter\right)^k$. For every $s\geq 0$,
\begin{multline*}
\DiffPointSup{(\NonLinearPartEDI,\NonLinearPartEDII)}{\VarPhaseSpaceStable(s),\VarPhaseSpaceUnstable(s),\Parameter}{k}.((z_i(s),v_i(s)),\Parameter_i)_{1 \leq i \leq k} = \\
\sum_{\substack{
		0\leq l\leq k\\
		\sigma \in \mathfrak{S}_k(l)}}
\DiffPartielGeneric{\VarPhaseSpaceStable,\VarPhaseSpaceUnstable}^l\DiffPartielGeneric{\Parameter}^{k-l}(\NonLinearPartEDI,\NonLinearPartEDII)(\VarPhaseSpaceStable(s),\VarPhaseSpaceUnstable(s),\Parameter).\bigl(\sigma.((z_i,v_i),\Parameter_i)_{1 \leq i \leq k}\bigr)
\end{multline*}
where
\begin{equation*}
	\sigma.((z_i,v_i),\Parameter_i)_{1 \leq i \leq k} = \biggl(\left(z_{\sigma(1)}(s),v_{\sigma(1)}(s)\right),\dots,\left(z_{\sigma(l)}(s),v_{\sigma(l)}(s)\right),\Parameter_{\sigma(l+1)},\dots,\Parameter_{\sigma(k)}\biggr)
\end{equation*}
and $\mathfrak{S}_k(l)$ is the set of all permutations of $\{1,\dots, k \}$ which are increasing on both the integer intervals $\left \llbracket 1,l \right \rrbracket$ and $\left \llbracket l+1,k \right \rrbracket$.
According to~\eqref{e-ligne-critique},~\eqref{e.controle-derivee-premiere-f-g-sup} and the mean value theorem, we have, for any $s \geq 0$,
\begin{equation*}
\norme{(\NonLinearPartEDI,\NonLinearPartEDII)(\VarPhaseSpaceStable(s),\VarPhaseSpaceUnstable(s),\Parameter)} \leq \NormeSupHigherDerivativesBounded{1}{\VFParamStraight} \norme{(\VarPhaseSpaceStable(s),\VarPhaseSpaceUnstable(s))} \leq e^{\ExpDecaySpeed s} \NormeSupHigherDerivativesBounded{1}{\VFParamStraight} \norme{(\VarPhaseSpaceStable,\VarPhaseSpaceUnstable)}_\ExpDecaySpeed 
\end{equation*}
According to~\eqref{e.controle-derivee-premiere-f-g-sup}, we have, for any $s \geq 0$,
\begin{equation*}
	\norme{\DiffPartielPoint{(\NonLinearPartEDI,\NonLinearPartEDII)}{\VarPhaseSpaceStable(s),\VarPhaseSpaceUnstable(s),\Parameter}{\VarPhaseSpaceStable,\VarPhaseSpaceUnstable}.(z_1(s),v_1(s))} \leq \NormeSupHigherDerivativesBounded{1}{\VFParamStraight} \norme{(\VarPhaseSpaceStable_1(s),\VarPhaseSpaceUnstable_1(s))} \leq e^{\ExpDecaySpeed s} \NormeSupHigherDerivativesBounded{1}{\VFParamStraight} \norme{(\VarPhaseSpaceStable_1,\VarPhaseSpaceUnstable_1)}_\ExpDecaySpeed 
\end{equation*}
According to~\eqref{e-controle-derivee-f-g-sup},~\eqref{e-controle-derivee-f-g-mu} and the mean value theorem, we have, for any $s \geq 0$,
\begin{equation*}
\norme{\DiffPartielPoint{(\NonLinearPartEDI,\NonLinearPartEDII)}{\VarPhaseSpaceStable(s),\VarPhaseSpaceUnstable(s),\Parameter}{\Parameter}.\Parameter_1} \leq \NormeSupHigherDerivativesBounded{2}{\VFParamStraight} \norme{(\VarPhaseSpaceStable(s),\VarPhaseSpaceUnstable(s))} \norme{\Parameter_1} \leq e^{\ExpDecaySpeed s} \NormeSupHigherDerivativesBounded{2}{\VFParamStraight} \norme{(\VarPhaseSpaceStable,\VarPhaseSpaceUnstable)}_\ExpDecaySpeed \norme{\Parameter_1} 
\end{equation*}
According to~\eqref{e-controle-derivee-f-g-sup}, we have, for any $s \geq 0$, $0 \leq l \leq k$ and $\sigma \in \mathfrak{S}_k(l)$,
\begin{multline*}
	\norme{\DiffPartielGeneric{\VarPhaseSpaceStable,\VarPhaseSpaceUnstable}^l\DiffPartielGeneric{\Parameter}^{k-l}(\NonLinearPartEDI,\NonLinearPartEDII)(\VarPhaseSpaceStable(s),\VarPhaseSpaceUnstable(s),\Parameter).\bigl(\sigma.((z_i,v_i),\Parameter_i)_{1 \leq i \leq k}\bigr)} \leq \\
	\NormeSupHigherDerivativesBounded{k}{\VFParamStraight} \prod_{i=1}^{l}\norme{(z_{\sigma(i)}(s),v_{\sigma(i)}(s))}_\ExpDecaySpeed \prod_{j=l+1}^{k} \norme{\Parameter_{\sigma(j)}} \\
	\leq  e^{l \ExpDecaySpeed s} \NormeSupHigherDerivativesBounded{k}{\VFParamStraight} \prod_{i=1}^{l}\norme{(z_{\sigma(i)},v_{\sigma(i)})}_\ExpDecaySpeed \prod_{j=l+1}^{k} \norme{\Parameter_{\sigma(j)}}
\end{multline*}
When $l=0$, the above estimate is not useful since there is no exponential decay, so we replace it with an estimate using $\NormeSupHigherDerivativesBounded{k+1}{\VFParamStraight}$ instead of $\NormeSupHigherDerivativesBounded{k}{\VFParamStraight}$. According to~\eqref{e-controle-derivee-f-g-sup},~\eqref{e-controle-derivee-f-g-mu} and the mean value theorem, we have, for any $s \geq 0$,
\begin{align*}
\norme{\DiffPartielPointSup{(\NonLinearPartEDI,\NonLinearPartEDII)}{\VarPhaseSpaceStable(s),\VarPhaseSpaceUnstable(s),\Parameter}{k}{\Parameter}.(\Parameter_i)_{1 \leq i \leq k}}
&\leq \NormeSupHigherDerivativesBounded{k+1}{\VFParamStraight} \norme{(\VarPhaseSpaceStable(s),\VarPhaseSpaceUnstable(s))} \prod_{i=1}^{k} \norme{\Parameter_i} \\
&\leq  e^{\ExpDecaySpeed s} \NormeSupHigherDerivativesBounded{k+1}{\VFParamStraight} \norme{(\VarPhaseSpaceStable,\VarPhaseSpaceUnstable)}_\ExpDecaySpeed \prod_{i=1}^{k} \norme{\Parameter_i}
\end{align*}
We now summarize the above (useful) estimates, using the inequality $e^{l \ExpDecaySpeed s}\leq e^{\ExpDecaySpeed s}$ for $l \neq 0$. For any $s \geq 0$, $0 \leq l \leq k$ and $\sigma \in \mathfrak{S}_k(l)$, we have
\begin{multline}\label{e.control-lemma-operator-all-cases}
	\norme{\DiffPartielGeneric{\VarPhaseSpaceStable,\VarPhaseSpaceUnstable}^l\DiffPartielGeneric{\Parameter}^{k-l}(\NonLinearPartEDI,\NonLinearPartEDII)(\VarPhaseSpaceStable(s),\VarPhaseSpaceUnstable(s),\Parameter).\bigl(\sigma.((z_i,v_i),\Parameter_i)_{1 \leq i \leq k}\bigr)} \leq \\
\begin{cases}
e^{\ExpDecaySpeed s}\NormeSupHigherDerivativesBounded{k+1}{\VFParamStraight} \prod_{i=1}^{k} \norme{\Parameter_i} \norme{(\VarPhaseSpaceStable,\VarPhaseSpaceUnstable)}_\ExpDecaySpeed   & \text{if $k \geq 0,l=0$} \\
e^{\ExpDecaySpeed s} \NormeSupHigherDerivativesBounded{k}{\VFParamStraight} \prod_{i=1}^{l}\norme{(z_{\sigma(i)},v_{\sigma(i)})}_\ExpDecaySpeed \prod_{j=l+1}^{k} \norme{\Parameter_{\sigma(j)}}   & \text{if $k \geq 1,l \neq 0$} 
\end{cases}
\end{multline}
It follows from~\eqref{e.control-lemma-operator-all-cases} that the map $s \mapsto e^{-(s-t)\UnstableMatrixStraight} \DiffPointSup{\NonLinearPartEDII}{\VarPhaseSpaceStable(s),\VarPhaseSpaceUnstable(s),\Parameter}{k}.((z_i(s),v_i(s)),\Parameter_i)_{1 \leq i \leq k}$ is integrable on $[t,+\infty[$, so $\OpFamilyII{k}$ is well defined. 

According to~\eqref{e.controle-polynomial-matrice-exp} and~\eqref{e.control-lemma-operator-all-cases}, and using the inequality $e^{l \ExpDecaySpeed s}\leq e^{\ExpDecaySpeed s}$, we have, for every $t \geq 0$,
\begin{multline}\label{e.control-Non-linear-part-EDO-ordre-k-I}
	\norme{\Integral{0}{t}{e^{(t-s)\StableMatrixStraight}\DiffPointSup{\NonLinearPartEDI}{\VarPhaseSpaceStable(s),\VarPhaseSpaceUnstable(s),\Parameter}{k}.((z_i(s),v_i(s)),\Parameter_i)_{1 \leq i \leq k}}{s}} \leq \\
	\frac{1}{\ExpDecaySpeed - \BorneVPStable}\frac{\ConstanteControleExpMatrixNormeII{\StableMatrix}}{\DistanceBordSpectre{\OriginDiff}{\ExpDecaySpeed}^{\frac{\DimStableSpaceStraight-1}{\DimPhaseSpace-1}}} e^{\ExpDecaySpeed t} \times
\begin{cases}
\NormeSupHigherDerivativesBounded{1}{\VFParamStraight} \norme{(\VarPhaseSpaceStable,\VarPhaseSpaceUnstable)}_\ExpDecaySpeed & \text{if $k=0$} \\
\left(\NormeSupHigherDerivativesBounded{k}{\VFParamStraight} + \NormeSupHigherDerivativesBounded{k+1}{\VFParamStraight} \norme{(\VarPhaseSpaceStable,\VarPhaseSpaceUnstable)}_\ExpDecaySpeed \right)  \prod_{i=1}^{k}\norme{\left((z_{i},v_{i}),\Parameter_i\right)}  & \text{if $k \geq 1$}
\end{cases}
\end{multline}
where we used the equality
\begin{equation*}
		\sum_{\substack{
			0\leq l\leq k\\
			\sigma \in \mathfrak{S}_k(l)}}
	\prod_{i=1}^{l}\norme{(z_{\sigma(i)},v_{\sigma(i)})}_\ExpDecaySpeed \prod_{j=l+1}^{k} \norme{\Parameter_{\sigma(j)}} = 
	\prod_{i=1}^{k} \left(\norme{(z_{i},v_{i})}_\ExpDecaySpeed + \norme{\Parameter_i}\right) = 
\prod_{i=1}^{k}\norme{\left((z_{i},v_{i}),\Parameter_i\right)}
\end{equation*}
Analogously, for every $t \geq 0$,
\begin{multline}\label{e.control-Non-linear-part-EDO-ordre-k-II}
\norme{\Integral{t}{+\infty}{e^{-(s-t)\UnstableMatrixStraight} \DiffPointSup{\NonLinearPartEDII}{\VarPhaseSpaceStable(s),\VarPhaseSpaceUnstable(s),\Parameter}{k}.((z_i(s),v_i(s)),\Parameter_i)_{1 \leq i \leq k}}{s}} \leq \\
\frac{1}{\BorneVPUnstable - \ExpDecaySpeed}\frac{\ConstanteControleExpMatrixNormeII{\UnstableMatrix}}{\DistanceBordSpectre{\OriginDiff}{\ExpDecaySpeed}^{\frac{\DimUnstableSpaceStraight-1}{\DimPhaseSpace-1}}} e^{\ExpDecaySpeed t} \times
\begin{cases}
\NormeSupHigherDerivativesBounded{1}{\VFParamStraight} \norme{(\VarPhaseSpaceStable,\VarPhaseSpaceUnstable)}_\ExpDecaySpeed & \text{if $k=0$} \\
 \left(\NormeSupHigherDerivativesBounded{k}{\VFParamStraight} + \NormeSupHigherDerivativesBounded{k+1}{\VFParamStraight} \norme{(\VarPhaseSpaceStable,\VarPhaseSpaceUnstable)}_\ExpDecaySpeed \right)  \prod_{i=1}^{k}\norme{\left((z_{i},v_{i}),\Parameter_i\right)}  & \text{if $k \geq 1$}
\end{cases}
\end{multline}
According to~\eqref{e.control-Non-linear-part-EDO-ordre-k-I}, \eqref{e.control-Non-linear-part-EDO-ordre-k-II} and the fact that $\max(\DimStableSpaceStraight,\DimUnstableSpaceStraight) \leq \DimPhaseSpace-1$,
\begin{subequations}
		\begin{equation} \label{e.control-second-part-operatorFP-firstlemma-first-case}
	\norme{\OpFamilyII{0}((\VarPhaseSpaceStable,\VarPhaseSpaceUnstable),\Parameter)}_\ExpDecaySpeed \leq 
	\frac{2\max\left(\ConstanteControleExpMatrixNormeII{\StableMatrix},\ConstanteControleExpMatrixNormeII{\UnstableMatrix}\right)}{\DistanceBordSpectre{\OriginDiff}{\ExpDecaySpeed}}  
	\NormeSupHigherDerivativesBounded{1}{\VFParamStraight} \norme{(\VarPhaseSpaceStable,\VarPhaseSpaceUnstable)}_\ExpDecaySpeed 
	\end{equation}
and for all $k \geq 1$,	
	\begin{multline} \label{e.control-second-part-operatorFP-firstlemma}
	\norme{\OpFamilyII{k}((\VarPhaseSpaceStable,\VarPhaseSpaceUnstable),\Parameter).((z_i,v_i),\Parameter_i)_{1 \leq i \leq k}}_\ExpDecaySpeed \leq \\
	\frac{2\max\left(\ConstanteControleExpMatrixNormeII{\StableMatrix},\ConstanteControleExpMatrixNormeII{\UnstableMatrix}\right)}{\DistanceBordSpectre{\OriginDiff}{\ExpDecaySpeed}} 
	\left(\NormeSupHigherDerivativesBounded{k}{\VFParamStraight} + \NormeSupHigherDerivativesBounded{k+1}{\VFParamStraight} \norme{(\VarPhaseSpaceStable,\VarPhaseSpaceUnstable)}_\ExpDecaySpeed \right)  \prod_{i=1}^{k}\norme{\left((z_{i},v_{i}),\Parameter_i\right)}
	\end{multline}
\end{subequations}
According to~\eqref{e.control-second-part-operatorFP-firstlemma}, for every $k \geq 1$, $\OpFamilyII{k}((\VarPhaseSpaceStable,\VarPhaseSpaceUnstable),\Parameter)$ is a continuous $k$-linear map whose subordinate norm satisfies
\begin{equation}\label{e.estimates-operatorFP-firstlemma}
	\normesub{\OpFamilyII{k}((\VarPhaseSpaceStable,\VarPhaseSpaceUnstable),\Parameter)}_\ExpDecaySpeed \leq 
	\frac{2\max\left(\ConstanteControleExpMatrixNormeII{\StableMatrix},\ConstanteControleExpMatrixNormeII{\UnstableMatrix}\right)}{\DistanceBordSpectre{\OriginDiff}{\ExpDecaySpeed}}   \left(\NormeSupHigherDerivativesBounded{k}{\VFParamStraight} + \NormeSupHigherDerivativesBounded{k+1}{\VFParamStraight} \norme{(\VarPhaseSpaceStable,\VarPhaseSpaceUnstable)}_\ExpDecaySpeed \right)
\end{equation}

\stepcounter{MyCounterProof}
\proofstep{Step \theMyCounterProof: for every $k \geq 0$, for every $((\VarPhaseSpaceStable,\VarPhaseSpaceUnstable),\Parameter) \in \FunctionSpaceGeneric{\ExpDecaySpeed}\times \SetParameter$, $\OpFamilyII{k}$ is differentiable at the point $((\VarPhaseSpaceStable,\VarPhaseSpaceUnstable),\Parameter)$ and $\DiffPoint{\OpFamilyII{k}}{(\VarPhaseSpaceStable,\VarPhaseSpaceUnstable),\Parameter} = \OpFamilyII{k+1}((\VarPhaseSpaceStable,\VarPhaseSpaceUnstable),\Parameter)$}
Let $k \geq 0$, $((\VarPhaseSpaceStable,\VarPhaseSpaceUnstable),\Parameter),((\Delta z,\Delta v),\Delta \Parameter)\in \FunctionSpaceGeneric{\ExpDecaySpeed}\times \SetParameter$ and $((z_i,v_i),\Parameter_i)_{1 \leq i \leq k} \in \left(\FunctionSpaceGeneric{\ExpDecaySpeed}\times \SetParameter\right)^k$. According to Taylor-Lagrange formula, for every $s\geq 0$,
\begin{multline*}
\left\lVert \DiffPointSup{(\NonLinearPartEDI,\NonLinearPartEDII)}{z(s)+\Delta \VarPhaseSpaceStable(s),\VarPhaseSpaceUnstable(s)+\Delta v(s),\Parameter+\Delta \Parameter}{k}.((z_i(s),v_i(s)),\Parameter_i)_{1 \leq i \leq k} - \right. \\
\DiffPointSup{(\NonLinearPartEDI,\NonLinearPartEDII)}{\VarPhaseSpaceStable(s),\VarPhaseSpaceUnstable(s),\Parameter}{k}.\bigl(((z_i(s),v_i(s)),\Parameter_i)_{1 \leq i \leq k}\bigr) -  \\
\left. \DiffPointSup{(\NonLinearPartEDI,\NonLinearPartEDII)}{\VarPhaseSpaceStable(s),\VarPhaseSpaceUnstable(s),\Parameter}{k+1}.\bigl(((z_i(s),v_i(s)),\Parameter_i)_{1 \leq i \leq k},(\Delta z(s),\Delta v(s),\Delta \Parameter)\bigr) \right\rVert \leq \\
\frac{1}{2} \sup_{w \in [0,1]} \norme{\Phi_s''(w)}
\end{multline*}
where
\begin{equation*}
	\Phi_s(w)= \DiffPointSup{(\NonLinearPartEDI,\NonLinearPartEDII)}{(\VarPhaseSpaceStable(s),\VarPhaseSpaceUnstable(s),\Parameter)+w(\Delta z(s),\Delta v(s),\Delta \Parameter)}{k}.((z_i(s),v_i(s)),\Parameter_i)_{1 \leq i \leq k}
\end{equation*}
By~\eqref{e.control-lemma-operator-all-cases} and computations similar to the ones done in the preceding step,
\begin{equation*}
	\sup_{w \in [0,1]} \norme{\Phi_s''(w)} \leq e^{\ExpDecaySpeed s}O_{((\Delta z,\Delta v),\Delta \Parameter) \to 0}\left(\norme{((\Delta z,\Delta v),\Delta \Parameter)}^2\right) \prod_{i=1}^{k}\norme{\left((z_{i},v_{i}),\Parameter_i\right)}
\end{equation*}
so
\begin{multline*}
	\normesub{\OpFamilyII{k}(((\VarPhaseSpaceStable,\VarPhaseSpaceUnstable),\Parameter)+((\Delta z,\Delta v),\Delta \Parameter))-\OpFamilyII{k}((\VarPhaseSpaceStable,\VarPhaseSpaceUnstable),\Parameter)-\OpFamilyII{k+1}((\VarPhaseSpaceStable,\VarPhaseSpaceUnstable),\Parameter).((\Delta z,\Delta v),\Delta \Parameter)}_\ExpDecaySpeed = \\
	O_{((\Delta z,\Delta v),\Delta \Parameter) \to 0}\left(\norme{((\Delta z,\Delta v),\Delta \Parameter)}^2\right)
\end{multline*}
By a staightforward induction on $k$, this implies that $\OpFamilyII{0}$ is smooth and for every $k \geq 1$, $\DiffSup{\OpFamilyII{0}}{k} = \OpFamilyII{k}$. As a further consequence, $\OpFamilyGeneric{\ExpDecaySpeed}$ is smooth and formula~\eqref{e.derivatives-operator-FP} holds true.

\stepcounter{MyCounterProof}
\proofstep{Step \theMyCounterProof: proof of estimates~\eqref{e.estimate-operator-FP-all-cases}}
First, estimate~\eqref{e.estimates-derivatives-operatorFP-high-order} is a direct consequence of~\eqref{e.estimates-operatorFP-firstlemma}.
Let $((\VarPhaseSpaceStable,\VarPhaseSpaceUnstable),\VarConditionInitiale,\Parameter) \in \FunctionSpaceGeneric{\ExpDecaySpeed} \times \StableSpaceStraight \times \SetParameter$ and $((\Delta z,\Delta v), \Delta \Parameter) \in \FunctionSpaceGeneric{\ExpDecaySpeed} \times \SetParameter$.
According to~\eqref{e.derivatives-operator-FP}, we have
\begin{equation*}
\DiffPartiel{\OpFamilyGeneric{\ExpDecaySpeed}}{\VarPhaseSpaceStable,\VarPhaseSpaceUnstable}((\VarPhaseSpaceStable,\VarPhaseSpaceUnstable),\VarConditionInitiale,\Parameter).(\Delta z,\Delta v) = \OpFamilyII{1}((\VarPhaseSpaceStable,\VarPhaseSpaceUnstable),\Parameter).((\Delta z,\Delta v),0)
\end{equation*}
By~\eqref{e.control-lemma-operator-all-cases} and similar computations to the ones done in the first step,
\begin{equation*}
\norme{\OpFamilyII{1}((\VarPhaseSpaceStable,\VarPhaseSpaceUnstable),\Parameter).((\Delta z,\Delta v),0)}_\ExpDecaySpeed \leq \frac{2\max\left(\ConstanteControleExpMatrixNormeII{\StableMatrix},\ConstanteControleExpMatrixNormeII{\UnstableMatrix}\right)}{\DistanceBordSpectre{\OriginDiff}{\ExpDecaySpeed}} \NormeSupHigherDerivativesBounded{1}{\VFParamStraight} \norme{(\Delta z,\Delta v)}_\ExpDecaySpeed
\end{equation*}
so, by~\eqref{e-hyp-norme-differential-exp-decay-Euclidean-norm}, \eqref{e.estimate-sup-derivatives-straight}, \eqref{e.ConstanteControleExpMatrixNormeII} and the fact that $\max\left(\ConstanteControleExpMatrix{\StableMatrix},\ConstanteControleExpMatrix{\UnstableMatrix}\right) \ConstanteInegTriangInverse{\StableSpace,\UnstableSpace} \leq \ConstanteControleExpMatrix{\OriginDiff}$, we have
\begin{equation} \label{e.control-estimee-importante-1-2}
	\frac{2\max\left(\ConstanteControleExpMatrixNormeII{\StableMatrix},\ConstanteControleExpMatrixNormeII{\UnstableMatrix}\right)}{\DistanceBordSpectre{\OriginDiff}{\ExpDecaySpeed}} \NormeSupHigherDerivativesBounded{1}{\VFParamStraight} \leq \frac{1}{2}
\end{equation}
so estimate~\eqref{e.partial-derivarive-operator-bounded-1demi} holds true.
By similar computations, we obtain
\begin{equation*}
\norme{\OpFamilyII{1}((\VarPhaseSpaceStable,\VarPhaseSpaceUnstable),\Parameter).((0,0),\Delta \Parameter)}_\ExpDecaySpeed \leq \frac{2\max\left(\ConstanteControleExpMatrixNormeII{\StableMatrix},\ConstanteControleExpMatrixNormeII{\UnstableMatrix}\right)}{\DistanceBordSpectre{\OriginDiff}{\ExpDecaySpeed}} \NormeSupHigherDerivativesBounded{2}{\VFParamStraight} \norme{(\VarPhaseSpaceStable,\VarPhaseSpaceUnstable)}_\ExpDecaySpeed \norme{\Delta \Parameter}
\end{equation*}
which implies the estimate~\eqref{e.partial-derivarive-operator-direction-parameter}. 
Estimate~\eqref{e.partial-derivarive-operator-direction-cond-ini} is a straightforward consequence of~\eqref{e.controle-polynomial-matrice-exp} and~\eqref{e.derivatives-operator-FP}. This concludes the proof of lemma~\ref{lemma.smoothness-FP}.
\end{proof}

\begin{lemme} \label{lemma.FP-smooth-estimates}
	For every $(\VarConditionInitiale,\Parameter) \in \StableSpaceStraight \times \SetParameter$, the operator $\OpFamily{\ExpDecaySpeed}{\VarConditionInitiale,\Parameter}$ admits a unique fixed point, which is independant of the choice of $\ExpDecaySpeed$ and is denoted by $(\FPStable{\VarConditionInitiale}{\Parameter},\FPUnstable{\VarConditionInitiale}{\Parameter})$. This fixed point satisifes the following control: for every $t \geq 0$,	\begin{equation}\label{e.estimate-FP-lemma}
		\norme{(\FPStable{\VarConditionInitiale}{\Parameter}(t),\FPUnstable{\VarConditionInitiale}{\Parameter}(t))} \leq \frac{2 \sqrt{2} \ConstanteControleExpMatrixNormeII{\StableMatrix}}{\DistanceBordSpectre{\OriginDiff}{\ExpDecaySpeed}^{\frac{\DimStableSpaceStraight-1}{\DimPhaseSpace-1}}} 	\norme{(\FPStable{\VarConditionInitiale}{\Parameter}(0),\FPUnstable{\VarConditionInitiale}{\Parameter}(0))} 
	\end{equation}
Moreover, the map
	\begin{equation*}
	\fonction{(\FPStableGeneric,\FPUnstableGeneric)}{\StableSpaceStraight \times \SetParameter}{\FunctionSpaceGeneric{\ExpDecaySpeed}}{(\VarConditionInitiale,\Parameter)}{(\FPStable{\VarConditionInitiale}{\Parameter},\FPUnstable{\VarConditionInitiale}{\Parameter})}
	\end{equation*}
	is smooth and, using the norm $\norme{(\VarConditionInitiale,\Parameter)}=\norme{\VarConditionInitiale}+\norme{\Parameter}$ on $\StableSpaceStraight \times \SetParameter$, we have the following estimates: for every $(\VarConditionInitiale,\Parameter) \in \StableSpaceStraight \times \SetParameter$,
\begin{subequations}
\label{e.estimates-FP-unstable-all-cases}
\begin{align}
\label{e.estimate-FPUnstable-0}
	\norme{\FPUnstable{\VarConditionInitiale}{\Parameter}}_\ExpDecaySpeed &\leq \frac{4\ConstanteControleExpMatrixNormeII{\StableMatrix}\ConstanteControleExpMatrixNormeII{\UnstableMatrix}}{ \DistanceBordSpectre{\OriginDiff}{\ExpDecaySpeed}} \NormeSupHigherDerivativesBounded{1}{\VFParamStraight} \norme{\VarConditionInitiale} \\
\label{e.estimate-FPUnstable-1-direction-VarCondIni}
	\normesub{\DiffPartiel{\FPUnstable{\VarConditionInitiale}{\Parameter}}{\VarConditionInitiale}}_\ExpDecaySpeed &\leq \frac{4\ConstanteControleExpMatrixNormeII{\StableMatrix}\ConstanteControleExpMatrixNormeII{\UnstableMatrix}}{ \DistanceBordSpectre{\OriginDiff}{\ExpDecaySpeed}} \NormeSupHigherDerivativesBounded{1}{\VFParamStraight} \\
\label{e.estimate-FPUnstable-1-direction-Parameter}
	\normesub{\DiffPartiel{\FPUnstable{\VarConditionInitiale}{\Parameter}}{\Parameter}}_\ExpDecaySpeed &\leq \frac{8\ConstanteControleExpMatrixNormeII{\StableMatrix}\ConstanteControleExpMatrixNormeII{\UnstableMatrix}}{\DistanceBordSpectre{\OriginDiff}{\ExpDecaySpeed}} \NormeSupHigherDerivativesBounded{2}{\VFParamStraight} \norme{\VarConditionInitiale} 
\end{align}
and, more generally, for every $k \geq 2$,
	\begin{equation}\label{e.estimateFP-all-derivatives}
	\normesub{\DiffSup{\FPUnstable{\VarConditionInitiale}{\Parameter}}{k}}_\ExpDecaySpeed \leq a_k\left(\frac{\max\left(\ConstanteControleExpMatrixNormeII{\StableMatrix},\ConstanteControleExpMatrixNormeII{\UnstableMatrix}\right)^2}{\DistanceBordSpectre{\OriginDiff}{\ExpDecaySpeed}^2} \NormeSupHigherDerivativesBoundedMax{k+1}{\VFParamStraight} \max\left(1,\norme{\VarConditionInitiale}\right)  \right)^{2k-1}
	\end{equation}
where $a_k$ is a positive constant independant of $\VFParam$, $(\StableSpace,\UnstableSpace)$, $\VarConditionInitiale$ and $\Parameter$.
\end{subequations}
\end{lemme}

\begin{rem}
	To conclude the proof of proposition~\ref{prop.variete-gamma-stable-global-version}, we only need estimates on $\FPUnstableGeneric$, this is why we did not give estimates on $\FPStableGeneric$ in the above statement. Such estimates will be used in the following proof though.
\end{rem}

\begin{proof}[Proof of lemma~\ref{lemma.FP-smooth-estimates}]
	\setcounter{MyCounterProof}{1}
	According to~\eqref{e.partial-derivarive-operator-bounded-1demi} and the contraction mapping theorem, for all $(\VarConditionInitiale,\Parameter) \in \StableSpaceStraight \times \SetParameter$, $\OpFamily{\ExpDecaySpeed}{\VarConditionInitiale,\Parameter}$ admits a unique fixed point
	\begin{equation*}
	(\FPStable{\ExpDecaySpeed,\VarConditionInitiale}{\Parameter},\FPUnstable{\ExpDecaySpeed,\VarConditionInitiale}{\Parameter})
	\end{equation*}
	Let $\ExpDecaySpeed' \in ] \SpectreUp{\StableMatrix},\min(0,\SpectreDown{\UnstableMatrix})[$ satisfying~\eqref{e-hyp-norme-differential-exp-decay-Euclidean-norm}. According to remark~\ref{rem.espaces-H}, we have
	\begin{equation*}
	\FunctionSpaceGeneric{\min(\ExpDecaySpeed,\ExpDecaySpeed')} \subset \FunctionSpaceGeneric{\max(\ExpDecaySpeed,\ExpDecaySpeed')}
	\end{equation*}
	so the fixed point $\left(\FPStable{\min(\ExpDecaySpeed,\ExpDecaySpeed'),\VarConditionInitiale}{\Parameter},\FPUnstable{\min(\ExpDecaySpeed,\ExpDecaySpeed'),\VarConditionInitiale}{\Parameter}\right)$ is also a fixed point of $\OpFamily{\max(\ExpDecaySpeed,\ExpDecaySpeed')}{\VarConditionInitiale,\Parameter}$. By uniqueness, this proves that the two fixed points coincide. Denote this unique fixed point by $	(\FPStable{\VarConditionInitiale}{\Parameter},\FPUnstable{\VarConditionInitiale}{\Parameter})$.
	
	To prove that the fixed point depends smoothly on $(\VarConditionInitiale,\Parameter)$, the idea is to apply the global inverse function theorem to the map
	\begin{equation*}
	\fonction{G^\ExpDecaySpeed}{\FunctionSpaceGeneric{\ExpDecaySpeed} \times \StableSpaceStraight \times \SetParameter}{\FunctionSpaceGeneric{\ExpDecaySpeed} \times \StableSpaceStraight \times \SetParameter}{((\VarPhaseSpaceStable,\VarPhaseSpaceUnstable),\VarConditionInitiale,\Parameter)}{\left(\OpFamily{\ExpDecaySpeed}{\VarConditionInitiale,\Parameter}(\VarPhaseSpaceStable,\VarPhaseSpaceUnstable)-(\VarPhaseSpaceStable,\VarPhaseSpaceUnstable),\VarConditionInitiale,\Parameter\right)}
	\end{equation*}
	Indeed, according to lemma~\ref{lemma.smoothness-FP}, $G^\ExpDecaySpeed$ is smooth and according to~\eqref{e.partial-derivarive-operator-bounded-1demi}, $G^\ExpDecaySpeed$ is injective and its differential is everywhere invertible. According to the global inverse function theorem, $V^\ExpDecaySpeed:=G^\ExpDecaySpeed(\FunctionSpaceGeneric{\ExpDecaySpeed} \times \StableSpaceStraight \times \SetParameter)$ is an open set of $\FunctionSpaceGeneric{\ExpDecaySpeed} \times \StableSpaceStraight \times \SetParameter$, the map $G^\ExpDecaySpeed:\FunctionSpaceGeneric{\ExpDecaySpeed} \times \StableSpaceStraight \times \SetParameter \to V^\ExpDecaySpeed$ is a diffeomorphism and its inverse is smooth.
	Denote by 
	\begin{equation*}
		\left(G^\ExpDecaySpeed\right)^{-1}_1: V^\ExpDecaySpeed \to \FunctionSpaceGeneric{\ExpDecaySpeed}
	\end{equation*}
	the first coordinate of $\left(G^\ExpDecaySpeed\right)^{-1}$.
	By definition of $(\FPStableGeneric,\FPUnstableGeneric)$,  for every $(\VarConditionInitiale,\Parameter) \in \StableSpaceStraight \times \SetParameter$,
	\begin{equation*}
	G^\ExpDecaySpeed\bigl((\FPStable{\VarConditionInitiale}{\Parameter},\FPUnstable{\VarConditionInitiale}{\Parameter}),\VarConditionInitiale,\Parameter\bigr) = \bigl((0,0),\VarConditionInitiale,\Parameter\bigr)
	\end{equation*}
	so
	\begin{equation*}
	(\FPStable{\VarConditionInitiale}{\Parameter},\FPUnstable{\VarConditionInitiale}{\Parameter})=\left(G^\ExpDecaySpeed\right)^{-1}_1\bigl((0,0),\VarConditionInitiale,\Parameter\bigr)
	\end{equation*}
	Since $\left(G^\ExpDecaySpeed\right)^{-1}$ is smooth, this completes the first part of the proof of lemma~\ref{lemma.FP-smooth-estimates}.

Fix $(\VarConditionInitiale,\Parameter) \in \StableSpaceStraight \times \SetParameter$. From the fixed point equation
\begin{equation}\label{e.FP-equation}
	(\FPStable{\VarConditionInitiale}{\Parameter},\FPUnstable{\VarConditionInitiale}{\Parameter}) = \OpFamily{\ExpDecaySpeed}{\VarConditionInitiale,\Parameter}\left(\FPStable{\VarConditionInitiale}{\Parameter},\FPUnstable{\VarConditionInitiale}{\Parameter}\right)
\end{equation}
and~\eqref{e.controle-polynomial-matrice-exp}, \eqref{e.control-second-part-operatorFP-firstlemma-first-case}, it follows that for all $t \geq 0$,
\begin{equation*}
	\norme{(\FPStable{\VarConditionInitiale}{\Parameter},\FPUnstable{\VarConditionInitiale}{\Parameter})(t)} \leq \frac{\ConstanteControleExpMatrixNormeII{\StableMatrix}}{\DistanceBordSpectre{\OriginDiff}{\ExpDecaySpeed}^{\frac{\DimStableSpaceStraight-1}{\DimPhaseSpace-1}}} e^{\BorneVPStable t} \norme{\VarConditionInitiale} + \frac{2\max\left(\ConstanteControleExpMatrixNormeII{\StableMatrix},\ConstanteControleExpMatrixNormeII{\UnstableMatrix}\right)}{\DistanceBordSpectre{\OriginDiff}{\ExpDecaySpeed}}  \NormeSupHigherDerivativesBounded{1}{\VFParamStraight} e^{\ExpDecaySpeed t}\norme{(\FPStable{\VarConditionInitiale}{\Parameter},\FPUnstable{\VarConditionInitiale}{\Parameter})}_\ExpDecaySpeed
\end{equation*}
so, according to~\eqref{e.control-estimee-importante-1-2} and the inequality $e^{\BorneVPStable t} \leq e^{\ExpDecaySpeed t}$,
\begin{equation}\label{e.estimate-FP-0}
	\norme{(\FPStable{\VarConditionInitiale}{\Parameter},\FPUnstable{\VarConditionInitiale}{\Parameter})}_\ExpDecaySpeed \leq \frac{2\ConstanteControleExpMatrixNormeII{\StableMatrix}}{\DistanceBordSpectre{\OriginDiff}{\ExpDecaySpeed}^{\frac{\DimStableSpaceStraight-1}{\DimPhaseSpace-1}}} \norme{\VarConditionInitiale}
\end{equation}
Plugging~\eqref{e.control-Non-linear-part-EDO-ordre-k-II} (case $k=0$) into~\eqref{e.FP-equation}, we obtain, for all $t \geq 0$,
\begin{equation}\label{e.estimate-FP-unstable-0}
\norme{\FPUnstable{\VarConditionInitiale}{\Parameter}(t)} \leq \frac{1}{\BorneVPUnstable - \ExpDecaySpeed}\frac{\ConstanteControleExpMatrixNormeII{\UnstableMatrix}}{\DistanceBordSpectre{\OriginDiff}{\ExpDecaySpeed}^{\frac{\DimUnstableSpaceStraight-1}{\DimPhaseSpace-1}}} e^{\ExpDecaySpeed t}
\NormeSupHigherDerivativesBounded{1}{\VFParamStraight} 	\norme{(\FPStable{\VarConditionInitiale}{\Parameter},\FPUnstable{\VarConditionInitiale}{\Parameter})}_\ExpDecaySpeed
\end{equation}
Plugging~\eqref{e.estimate-FP-0} into~\eqref{e.estimate-FP-unstable-0} and using the equality $\DimStableSpaceStraight+\DimUnstableSpaceStraight=\DimPhaseSpace$, it follows that estimate~\eqref{e.estimate-FPUnstable-0} holds true. As a byproduct, we obtain from~\eqref{e.estimate-FP-0} that~\eqref{e.estimate-FP-lemma} holds true.

Taking the derivative of~\eqref{e.FP-equation} with respect to the variable $\VarConditionInitiale$ and using~\eqref{e.partial-derivarive-operator-bounded-1demi} and~\eqref{e.partial-derivarive-operator-direction-cond-ini}, we get
\begin{equation}\label{e.estimate-FP-1-direction-VarCondIni}
	\normesub{\DiffPartiel{(\FPStable{\VarConditionInitiale}{\Parameter},\FPUnstable{\VarConditionInitiale}{\Parameter})}{\VarConditionInitiale}}_\ExpDecaySpeed \leq \frac{2\ConstanteControleExpMatrixNormeII{\StableMatrix}}{\DistanceBordSpectre{\OriginDiff}{\ExpDecaySpeed}^{\frac{\DimStableSpaceStraight-1}{\DimPhaseSpace-1}}}
\end{equation}
Moreover, taking the derivative of~\eqref{e.FP-equation} with respect to the variable $\VarConditionInitiale$, we obtain, for every $\VarConditionInitiale_1 \in \StableSpaceStraight$ and every $t \geq 0$,
\begin{equation*}
\DiffPartiel{\FPUnstable{\VarConditionInitiale}{\Parameter}}{\VarConditionInitiale}.\VarConditionInitiale_1 (t) = -\Integral{t}{+\infty}{e^{-(s-t)\UnstableMatrixStraight} \left(\DiffPartielPoint{\NonLinearPartEDII}{\FPStable{\VarConditionInitiale}{\Parameter}(s),\FPUnstable{\VarConditionInitiale}{\Parameter}(s),\Parameter}{\VarPhaseSpaceStable,\VarPhaseSpaceUnstable}. \DiffPartiel{(\FPStable{\VarConditionInitiale}{\Parameter},\FPUnstable{\VarConditionInitiale}{\Parameter})}{\VarConditionInitiale}.\VarConditionInitiale_1 (s) \right)}{s}
\end{equation*}
so using~\eqref{e.controle-polynomial-matrice-exp} and~\eqref{e.control-lemma-operator-all-cases}, we obtain
\begin{equation}\label{e.estimate-FP-Unstable-1-direction-VarCondIni}
\normesub{\DiffPartiel{\FPUnstable{\VarConditionInitiale}{\Parameter}}{\VarConditionInitiale}}_\ExpDecaySpeed \leq \frac{1}{\BorneVPUnstable-\ExpDecaySpeed}\frac{\ConstanteControleExpMatrixNormeII{\UnstableMatrix}}{\DistanceBordSpectre{\OriginDiff}{\ExpDecaySpeed}^{\frac{\DimUnstableSpaceStraight-1}{\DimPhaseSpace-1}}} \NormeSupHigherDerivativesBounded{1}{\VFParamStraight} \normesub{\DiffPartiel{(\FPStable{\VarConditionInitiale}{\Parameter},\FPUnstable{\VarConditionInitiale}{\Parameter})}{\VarConditionInitiale}}_\ExpDecaySpeed
\end{equation}
Plugging~\eqref{e.estimate-FP-1-direction-VarCondIni} into~\eqref{e.estimate-FP-Unstable-1-direction-VarCondIni} and using the equality $\DimStableSpaceStraight+\DimUnstableSpaceStraight=\DimPhaseSpace$, it follows that estimate~\eqref{e.estimate-FPUnstable-1-direction-VarCondIni} holds true.

Taking the derivative of~\eqref{e.FP-equation} with respect to the variable $\Parameter$ and using~\eqref{e.partial-derivarive-operator-bounded-1demi} and~\eqref{e.partial-derivarive-operator-direction-parameter}, we get
\begin{equation}\label{e.estimate-FP-1-direction-Parameter}
\normesub{\DiffPartiel{(\FPStable{\VarConditionInitiale}{\Parameter},\FPUnstable{\VarConditionInitiale}{\Parameter})}{\Parameter}}_\ExpDecaySpeed \leq \frac{4\max\left(\ConstanteControleExpMatrixNormeII{\StableMatrix},\ConstanteControleExpMatrixNormeII{\UnstableMatrix}\right)}{\DistanceBordSpectre{\OriginDiff}{\ExpDecaySpeed}}\NormeSupHigherDerivativesBounded{2}{\VFParamStraight} \norme{(\FPStable{\VarConditionInitiale}{\Parameter},\FPUnstable{\VarConditionInitiale}{\Parameter})}_\ExpDecaySpeed
\end{equation}
Moreover, taking the derivative of~\eqref{e.FP-equation} with respect to the variable $\Parameter$, we obtain, for every $\Parameter_1 \in \SetParameter$ and every $t \geq 0$,
\begin{multline*}
	\DiffPartiel{\FPUnstable{\VarConditionInitiale}{\Parameter}}{\Parameter}.\Parameter_1 (t) = -\Integral{t}{+\infty}{e^{-(s-t)\UnstableMatrixStraight} \left(\DiffPartielPoint{\NonLinearPartEDII}{\FPStable{\VarConditionInitiale}{\Parameter}(s),\FPUnstable{\VarConditionInitiale}{\Parameter}(s),\Parameter}{\VarPhaseSpaceStable,\VarPhaseSpaceUnstable}. \DiffPartiel{(\FPStable{\VarConditionInitiale}{\Parameter},\FPUnstable{\VarConditionInitiale}{\Parameter})}{\Parameter}.\Parameter_1 (s) \right)}{s} \\
	-\Integral{t}{+\infty}{e^{-(s-t)\UnstableMatrixStraight} \left(\DiffPartielPoint{\NonLinearPartEDII}{\FPStable{\VarConditionInitiale}{\Parameter}(s),\FPUnstable{\VarConditionInitiale}{\Parameter}(s),\Parameter}{\Parameter}.  \Parameter_1 \right)}{s}
\end{multline*}
so using~\eqref{e.controle-polynomial-matrice-exp} and~\eqref{e.control-lemma-operator-all-cases}, we obtain
\begin{equation} \label{e.estimate-FP-unstable-1-direction-Parameter-proof-1}
	\normesub{\DiffPartiel{\FPUnstable{\VarConditionInitiale}{\Parameter}}{\Parameter}}_\ExpDecaySpeed \leq \frac{1}{\BorneVPUnstable - \ExpDecaySpeed}\frac{\ConstanteControleExpMatrixNormeII{\UnstableMatrix}}{\DistanceBordSpectre{\OriginDiff}{\ExpDecaySpeed}^{\frac{\DimUnstableSpaceStraight-1}{\DimPhaseSpace-1}}} \left(\NormeSupHigherDerivativesBounded{1}{\VFParamStraight}\normesub{\DiffPartiel{(\FPStable{\VarConditionInitiale}{\Parameter},\FPUnstable{\VarConditionInitiale}{\Parameter})}{\Parameter}}_\ExpDecaySpeed+\NormeSupHigherDerivativesBounded{2}{\VFParamStraight}\norme{(\FPStable{\VarConditionInitiale}{\Parameter},\FPUnstable{\VarConditionInitiale}{\Parameter})}_\ExpDecaySpeed\right)
\end{equation}
Plugging~\eqref{e.estimate-FP-1-direction-Parameter} into~\eqref{e.estimate-FP-unstable-1-direction-Parameter-proof-1}, we have
\begin{equation*} 
\normesub{\DiffPartiel{\FPUnstable{\VarConditionInitiale}{\Parameter}}{\Parameter}}_\ExpDecaySpeed \leq \frac{1}{\BorneVPUnstable - \ExpDecaySpeed}\frac{\ConstanteControleExpMatrixNormeII{\UnstableMatrix}}{\DistanceBordSpectre{\OriginDiff}{\ExpDecaySpeed}^{\frac{\DimUnstableSpaceStraight-1}{\DimPhaseSpace-1}}}\NormeSupHigherDerivativesBounded{2}{\VFParamStraight}\norme{(\FPStable{\VarConditionInitiale}{\Parameter},\FPUnstable{\VarConditionInitiale}{\Parameter})}_\ExpDecaySpeed \left(1+\frac{4\max\left(\ConstanteControleExpMatrixNormeII{\StableMatrix},\ConstanteControleExpMatrixNormeII{\UnstableMatrix}\right)}{\DistanceBordSpectre{\OriginDiff}{\ExpDecaySpeed}} \NormeSupHigherDerivativesBounded{1}{\VFParamStraight}\right)
\end{equation*}
Using~\eqref{e.ConstanteControleExpMatrixNormeII},~\eqref{e-hyp-norme-differential-exp-decay-Euclidean-norm},~\eqref{e.estimate-sup-derivatives-straight}, the fact that
\begin{equation*}
	\max\left(\ConstanteControleExpMatrix{\StableMatrix},\ConstanteControleExpMatrix{\UnstableMatrix}\right) \ConstanteInegTriangInverse{\StableSpace,\UnstableSpace} \leq \ConstanteControleExpMatrix{\OriginDiff}
\end{equation*}
and the inequality $\max(\DimStableSpaceStraight,\DimUnstableSpaceStraight) <n$, we get
\begin{equation*}
	\frac{4\max\left(\ConstanteControleExpMatrixNormeII{\StableMatrix},\ConstanteControleExpMatrixNormeII{\UnstableMatrix}\right)}{\DistanceBordSpectre{\OriginDiff}{\ExpDecaySpeed}} \NormeSupHigherDerivativesBounded{1}{\VFParamStraight} \leq 1
\end{equation*}
so
\begin{equation} \label{e.estimate-FP-unstable-1-direction-Parameter-proof-2}
\normesub{\DiffPartiel{\FPUnstable{\VarConditionInitiale}{\Parameter}}{\Parameter}}_\ExpDecaySpeed \leq 2\frac{1}{\BorneVPUnstable - \ExpDecaySpeed}\frac{\ConstanteControleExpMatrixNormeII{\UnstableMatrix}}{\DistanceBordSpectre{\OriginDiff}{\ExpDecaySpeed}^{\frac{\DimUnstableSpaceStraight-1}{\DimPhaseSpace-1}}}\NormeSupHigherDerivativesBounded{2}{\VFParamStraight}\norme{(\FPStable{\VarConditionInitiale}{\Parameter},\FPUnstable{\VarConditionInitiale}{\Parameter})}_\ExpDecaySpeed 
\end{equation}
Plugging~\eqref{e.estimate-FP-0} into~\eqref{e.estimate-FP-unstable-1-direction-Parameter-proof-2} and using the fact that $\DimStableSpaceStraight + \DimUnstableSpaceStraight = \DimPhaseSpace$, it follows that estimate~\eqref{e.estimate-FPUnstable-1-direction-Parameter} holds true.

We are now going to prove~\eqref{e.estimateFP-all-derivatives}.
To avoid clutter with constants independant of $\VFParam$, $(\StableSpace,\UnstableSpace)$, $\VarConditionInitiale$ and $\Parameter$ in the following estimates, we introduce the following notation: for any real positive functions $\delta_1,\delta_2$ depending on $(\VFParam,\StableSpace,\UnstableSpace,\VarConditionInitiale,\Parameter)$ we define the binary relationship $\precsim$ by
\begin{equation}\label{e.binary-relationship-constant-free}
\delta_1 \precsim \delta_2 \iff \exists C >0, \; \delta_1 \leq C \delta_2
\end{equation}
We will use the abuse of notation $\delta_1(\VFParam,\StableSpace,\UnstableSpace,\VarConditionInitiale,\Parameter) \precsim \delta_2(\VFParam,\StableSpace,\UnstableSpace,\VarConditionInitiale,\Parameter)$.
For every $k \geq 2$, for every $(\VarConditionInitiale,\Parameter) \in \StableSpaceStraight \times \SetParameter$, let
\begin{equation}\label{e.terme-principal-control-derivatives}
u_k \egaldef \frac{\max\left(\ConstanteControleExpMatrixNormeII{\StableMatrix},\ConstanteControleExpMatrixNormeII{\UnstableMatrix}\right)}{\DistanceBordSpectre{\OriginDiff}{\ExpDecaySpeed}} \max\left(1,\norme{(\FPStable{\VarConditionInitiale}{\Parameter},\FPUnstable{\VarConditionInitiale}{\Parameter})}_\ExpDecaySpeed\right) \NormeSupHigherDerivativesBoundedMax{k}{\VFParamStraight}
\end{equation}
We will use the fact that $(u_k)$ is increasing.
We are now going to prove by induction on $k$ that, for every $k \geq 1$,
\begin{equation}\label{e.conjecture-norme-FP-high-derivative}
\normesub{\DiffSup{(\FPStable{\VarConditionInitiale}{\Parameter},\FPUnstable{\VarConditionInitiale}{\Parameter})}{k}}_\ExpDecaySpeed \precsim u_{k+1}^{2k-1}
\end{equation}
According to~\eqref{e.estimate-FP-1-direction-VarCondIni} and~\eqref{e.estimate-FP-1-direction-Parameter}, we have
\begin{equation*}
\normesub{\Diff{(\FPStable{\VarConditionInitiale}{\Parameter},\FPUnstable{\VarConditionInitiale}{\Parameter})}}_\ExpDecaySpeed \precsim  u_2
\end{equation*}
which proves~\eqref{e.conjecture-norme-FP-high-derivative} in the case $k=1$. Let $k \geq 2$.
Deriving~\eqref{e.FP-equation}, we have
\begin{multline}\label{e.Faa-di-Bruno-formula-FP}
\DiffSup{(\FPStable{\VarConditionInitiale}{\Parameter},\FPUnstable{\VarConditionInitiale}{\Parameter})}{k} =
\DiffPartielPoint{\OpFamilyGeneric{\ExpDecaySpeed}}{(\FPStable{\VarConditionInitiale}{\Parameter},\FPUnstable{\VarConditionInitiale}{\Parameter}),\VarConditionInitiale,\Parameter}{\VarPhaseSpaceStable,\VarPhaseSpaceUnstable}.\DiffSup{(\FPStable{\VarConditionInitiale}{\Parameter},\FPUnstable{\VarConditionInitiale}{\Parameter})}{k} \\
+\sum_{j=2}^{k}  \sum_{\substack{i_1,\dots,i_j \geq 1 \\ i_1 + \dots +i_j=k}} C_{i_1,\dots,i_j} \DiffPointSup{\OpFamilyGeneric{\ExpDecaySpeed}}{(\FPStable{\VarConditionInitiale}{\Parameter},\FPUnstable{\VarConditionInitiale}{\Parameter}),\VarConditionInitiale,\Parameter}{j}. \\
\left(\DiffSup{((\FPStable{\VarConditionInitiale}{\Parameter},\FPUnstable{\VarConditionInitiale}{\Parameter}),\VarConditionInitiale,\Parameter)}{i_1},\dots,\DiffSup{((\FPStable{\VarConditionInitiale}{\Parameter},\FPUnstable{\VarConditionInitiale}{\Parameter}),\VarConditionInitiale,\Parameter)}{i_j}\right)
\end{multline}
where the $C_{i_1,\dots,i_j}$ are the constants appearing in the standard Fa\`a di Bruno's formula.
According to~\eqref{e.estimates-derivatives-operatorFP-high-order}, for all $j \geq 2$
\begin{equation}\label{e-derivative-operator-FP-special-case-in-FP}
	\normesub{\DiffPointSup{\OpFamilyGeneric{\ExpDecaySpeed}}{(\FPStable{\VarConditionInitiale}{\Parameter},\FPUnstable{\VarConditionInitiale}{\Parameter}),\VarConditionInitiale,\Parameter}{j}}_\ExpDecaySpeed \precsim u_{j+1}
\end{equation}
Plugging estimates~\eqref{e.partial-derivarive-operator-bounded-1demi} and~\eqref{e-derivative-operator-FP-special-case-in-FP}   into~\eqref{e.Faa-di-Bruno-formula-FP} and using the induction hypothesis, we have
\begin{align*}
	\normesub{\DiffSup{(\FPStable{\VarConditionInitiale}{\Parameter},\FPUnstable{\VarConditionInitiale}{\Parameter})}{k}}_\ExpDecaySpeed &\precsim \max_{\substack{2 \leq j \leq k \\ i_1,\dots,i_j \geq 1 \\ i_1 + \dots +i_j=k}} u_{j+1} \prod_{l=1}^{j} u_{i_l+1}^{2i_l-1} \\
	& \precsim u_{k+1} \max_{\substack{2 \leq j \leq k \\ i_1,\dots,i_j \geq 1 \\ i_1 + \dots +i_j=k}}  u_{k+1}^{\sum_{l=1}^{j} (2i_l-1)} \\
	& \precsim u_{k+1}^{2k-1}
\end{align*}
which proves~\eqref{e.conjecture-norme-FP-high-derivative} for all $k \geq 1$ by induction.
According to~\eqref{e.estimate-FP-0}, we have
\begin{equation}\label{e.estimate-FP-0-max}
	\max\left(1,\norme{(\FPStable{\VarConditionInitiale}{\Parameter},\FPUnstable{\VarConditionInitiale}{\Parameter})}_\ExpDecaySpeed\right) \precsim \frac{\max\left(\ConstanteControleExpMatrixNormeII{\StableMatrix},\ConstanteControleExpMatrixNormeII{\UnstableMatrix}\right)}{\DistanceBordSpectre{\OriginDiff}{\ExpDecaySpeed}} \max\left(1,\norme{\VarConditionInitiale}\right)
\end{equation}
Plugging~\eqref{e.estimate-FP-0-max} into~\eqref{e.conjecture-norme-FP-high-derivative}, it follows that~\eqref{e.estimateFP-all-derivatives} holds true. This concludes the proof of lemma~\ref{lemma.FP-smooth-estimates}.
\end{proof}

Let us define the map
\begin{equation*}
\fonction{\SMTGraphMapStraight}{\StableSpaceStraight \times \SetParameter}{\UnstableSpaceStraight}{(\VarConditionInitiale,\Parameter)}{\SMTGraphMapStraight_\Parameter(\VarConditionInitiale):=\FPUnstable{\VarConditionInitiale}{\Parameter}(0)}
\end{equation*}

\begin{lemme}\label{lemma.equivalence-stable-set-graph}
	For every $\Parameter \in \SetParameter$, the global $\gamma$-stable set $\SMexp{0,\VFParamStraight_\Parameter}{\ExpDecaySpeed}$ is exactly the graph of the map $\SMTGraphMapStraight_\Parameter : \StableSpaceStraight \to \UnstableSpaceStraight$. Moreover, $\SMTGraphMapStraight$ is smooth and for every $k \geq 0$, $(\VarConditionInitiale,\Parameter) \in \StableSpaceStraight \times \SetParameter$, $(\VarConditionInitiale_i,\Parameter_i)_{1 \leq i \leq k} \in \left( \StableSpaceStraight \times \SetParameter\right)^k$
	\begin{equation}\label{e.controle-graph-straight}
		\DiffPointSup{\SMTGraphMapStraight}{\VarConditionInitiale,\Parameter}{k}.(\VarConditionInitiale_i,\Parameter_i)_{1 \leq i \leq k} = \DiffSup{\FPUnstable{\VarConditionInitiale}{\Parameter}}{k}.(\VarConditionInitiale_i,\Parameter_i)_{1 \leq i \leq k}(0)
	\end{equation}
\end{lemme}

\begin{proof}[Proof of lemma~\ref{lemma.equivalence-stable-set-graph}]
Fix $\Parameter \in \SetParameter$.
Let $(\VarConditionInitiale,\VarPhaseSpaceUnstable_0) \in \StableSpaceStraight \times \UnstableSpaceStraight$.
We are going to prove the following equivalence:
\begin{equation*}
	(\VarConditionInitiale,\VarPhaseSpaceUnstable_0) \in \SMexp{0,\VFParamStraight_\Parameter}{\ExpDecaySpeed} \iff
	\VarPhaseSpaceUnstable_0=\SMTGraphMapStraight_\Parameter(\VarConditionInitiale)
\end{equation*}
Let $(\VarPhaseSpaceStable,\VarPhaseSpaceUnstable):\left[0,+\infty\right) \to \StableSpaceStraight \times \UnstableSpaceStraight$ be a continuous map such that $\VarPhaseSpaceStable(0)=\VarConditionInitiale$. By a straighforward computation, $(\VarPhaseSpaceStable,\VarPhaseSpaceUnstable)$ is an orbit of $\VFParamStraight_\Parameter$ (that is, a solution of~\eqref{e-EDO_initiale}) if and only if for every $0 \leq t \leq \tau$,
\begin{equation*}
\begin{split}
z(t) &= e^{t \StableMatrixStraight}\VarConditionInitiale + \int_0^t e^{(t-s)\StableMatrixStraight}\NonLinearPartEDI(\VarPhaseSpaceStable(s),\VarPhaseSpaceUnstable(s),\Parameter)ds \\
v(t)& = e^{-(\tau-t)\UnstableMatrixStraight}v(\tau) - \int_t^\tau e^{-(s-t)\UnstableMatrixStraight}\NonLinearPartEDII(\VarPhaseSpaceStable(s),\VarPhaseSpaceUnstable(s),\Parameter)ds
\end{split}
\end{equation*}
If we assume that $(\VarPhaseSpaceStable,\VarPhaseSpaceUnstable) \in \FunctionSpaceGeneric{\ExpDecaySpeed}$, then, according to~\eqref{e-ligne-critique},~\eqref{e.controle-derivee-premiere-f-g-sup} and~\eqref{e.controle-polynomial-matrice-exp}, the second integral above converges as $\tau$ goes to $+\infty$. Letting $\tau$ tend to $+\infty$, we get that $(\VarPhaseSpaceStable,\VarPhaseSpaceUnstable) \in \FunctionSpaceGeneric{\ExpDecaySpeed}$ and $(\VarPhaseSpaceStable,\VarPhaseSpaceUnstable)$ is a solution of~\eqref{e-EDO_initiale} if and only if $(\VarPhaseSpaceStable,\VarPhaseSpaceUnstable) \in \FunctionSpaceGeneric{\ExpDecaySpeed}$ and for every $t \geq 0$,
\begin{equation}\label{e-pt-fixe-0-lemma-equivalence-stable-set-graph}
\begin{split}
z(t) &= e^{t \StableMatrixStraight}\VarConditionInitiale + \int_0^t e^{(t-s)\StableMatrixStraight}\NonLinearPartEDI(\VarPhaseSpaceStable(s),\VarPhaseSpaceUnstable(s),\Parameter)ds \\
v(t)& = - \int_t^{+ \infty} e^{-(s-t)\UnstableMatrixStraight}\NonLinearPartEDII(\VarPhaseSpaceStable(s),\VarPhaseSpaceUnstable(s),\Parameter)ds
\end{split}
\end{equation}
\emph{i.e.} if and only if $(\VarPhaseSpaceStable,\VarPhaseSpaceUnstable)$ is a fixed point of $\OpFamily{\ExpDecaySpeed}{\VarConditionInitiale,\Parameter}$.

From now on $(\VarPhaseSpaceStable,\VarPhaseSpaceUnstable)$ denotes the orbit of $\VFParamStraight_\Parameter$ starting from $(\VarConditionInitiale,\VarPhaseSpaceUnstable_0)$ at $t=0$, that is, $(z(t),v(t))=\VFParamStraight_\Parameter^t(\VarConditionInitiale,\VarPhaseSpaceUnstable_0)$. 
We have the following equivalences:
\begin{align*}
	 & (\VarConditionInitiale,\VarPhaseSpaceUnstable_0) \in \SMexp{0,\VFParamStraight_\Parameter}{\ExpDecaySpeed} \\
\iff & (\VarPhaseSpaceStable,\VarPhaseSpaceUnstable) \in \FunctionSpaceGeneric{\ExpDecaySpeed} &\text{by~\eqref{e.gamma-global-stable-set-context-VF-parameter}} \\
\iff & \text{$(\VarPhaseSpaceStable,\VarPhaseSpaceUnstable)$ is a fixed point of $\OpFamily{\ExpDecaySpeed}{\VarConditionInitiale,\Parameter}$} &\text{by the above reasoning} \\
\iff & (\VarPhaseSpaceStable,\VarPhaseSpaceUnstable) = (\FPStable{\VarConditionInitiale}{\Parameter},\FPUnstable{\VarConditionInitiale}{\Parameter}) & \text{by lemma~\ref{lemma.smoothness-FP}} \\
\iff & (z(0),v(0))= (\FPStable{\VarConditionInitiale}{\Parameter}(0),\FPUnstable{\VarConditionInitiale}{\Parameter}(0)) & \text{by uniqueness in Cauchy-Lipschitz theorem} \\
\iff & (\VarConditionInitiale,\VarPhaseSpaceUnstable_0)=(\VarConditionInitiale,\SMTGraphMapStraight_\Parameter(\VarConditionInitiale)) & \text{by definition of $\SMTGraphMapStraight_\Parameter$}
\end{align*}
which concludes the first part of the proof. Let $E_0$ be the "evaluation at time $t=0$" map
\begin{equation*}
\fonction{E_0}{\FunctionSpace{\ExpDecaySpeed}{\DimUnstableSpaceStraight}}{\UnstableSpaceStraight}{v}{v(0)}
\end{equation*}
By definition of the $\ExpDecaySpeed$-norm, $E_0$ is a linear continuous map (with $\norme{E_0}\leq 1$) and as such is smooth. Since
\begin{equation*}
	\SMTGraphMapStraight = E_0 \circ \FPUnstableGeneric
\end{equation*}
it follows from lemma~\ref{lemma.smoothness-FP} that $\SMTGraphMapStraight$ is smooth and~\eqref{e.controle-graph-straight} holds true. This concludes the proof of lemma~\ref{lemma.equivalence-stable-set-graph}.
\end{proof}

Plugging estimates~\eqref{e.estimates-FP-unstable-all-cases} into~\eqref{e.controle-graph-straight} and using~\eqref{e.estimate-sup-derivatives-straight}, \eqref{e.ConstanteControleExpMatrixNormeII} and the fact that
\begin{equation*}
	\ConstanteControleExpMatrix{\StableMatrix}\ConstanteControleExpMatrix{\UnstableMatrix} \ConstanteInegTriangInverse{\StableSpace,\UnstableSpace} \leq \ConstanteControleExpMatrix{\OriginDiff}
\end{equation*}
we have,
for every $(\VarConditionInitiale,\Parameter) \in \StableSpaceStraight \times \SetParameter$,
\begin{subequations}
	\label{e.estimates-SMTgraph-straight-all-cases}
	\begin{align}
	\label{e.estimate-SMTgraph-straight-0}
	\norme{\SMTGraphMapStraight(\VarConditionInitiale,\Parameter)} &\precsim \frac{\ConstanteControleExpMatrix{\OriginDiff}}{ \DistanceBordSpectre{\OriginDiff}{\ExpDecaySpeed}} \NormeSupHigherDerivativesBounded{1}{\VFParam} \norme{\VarConditionInitiale} \\
	\label{e.estimate-SMTgraph-straight-1-direction-VarCondIni}
	\normesub{\DiffPartielPoint{\SMTGraphMapStraight}{\VarConditionInitiale,\Parameter}{\VarConditionInitiale}} &\precsim \frac{\ConstanteControleExpMatrix{\OriginDiff}}{ \DistanceBordSpectre{\OriginDiff}{\ExpDecaySpeed}} \NormeSupHigherDerivativesBounded{1}{\VFParam} \\
	\label{e.estimate-SMTgraph-straight-1-direction-Parameter}
	\normesub{\DiffPartielPoint{\SMTGraphMapStraight}{\VarConditionInitiale,\Parameter}{\Parameter}} &\precsim \frac{\ConstanteControleExpMatrix{\OriginDiff}}{ \DistanceBordSpectre{\OriginDiff}{\ExpDecaySpeed}} \NormeSupHigherDerivativesBounded{2}{\VFParam} \norme{\VarConditionInitiale} 
	\end{align}
	and, more generally, for every $k \geq 2$,
	\begin{equation}\label{e.estimate-SMTgraph-straight-all-derivatives}
	\normesub{\DiffPointSup{\SMTGraphMapStraight}{\VarConditionInitiale,\Parameter}{k}} \precsim \left(\frac{\ConstanteControleExpMatrix{\OriginDiff}^2}{\DistanceBordSpectre{\OriginDiff}{\ExpDecaySpeed}^2}\NormeSupHigherDerivativesBoundedMax{k+1}{\VFParam} \max\left(1,\norme{\VarConditionInitiale}\right)  \right)^{2k-1}
	\end{equation}
\end{subequations}
where $\precsim$ is defined by~\eqref{e.binary-relationship-constant-free}. Let us now define
\begin{equation*}
\fonction{\SMTGraphMap}{\StableSpace \times \SetParameter}{\UnstableSpace}{(\VarConditionInitiale,\Parameter)}{\PHSIsoStraight^{-1}\left(\SMTGraphMapStraight_\Parameter(\PHSIsoStraight(\VarConditionInitiale))\right)} 
\end{equation*}
One can remark that
\begin{equation*}
	\PHSIsoStraight\left(\Graph \SMTGraphMap\right) = \Graph \SMTGraphMapStraight
\end{equation*}
and, according to~\eqref{e.defin-VFParam-Straight}, we have, for every $\Parameter \in \SetParameter$,
\begin{equation*}
\PHSIsoStraight\left(\SMexp{0,\VFParam_\Parameter}{\ExpDecaySpeed}\right) = \SMexp{0,\VFParamStraight_\Parameter}{\ExpDecaySpeed}
\end{equation*}
so, according to lemma~\ref{lemma.equivalence-stable-set-graph}, we get that item~\eqref{item.graph-structure-global-estimate} of proposition~\ref{prop.variete-gamma-stable-global-version} holds true. 

We are now going to prove that the estimates~\eqref{e.estimates-SMT-graph-global-estimates} hold true.
According to the fact that $\PHSIsoStraight_{|\StableSpace}$ and $(\PHSIsoStraight^{-1})_{|\UnstableSpaceStraight}$ are isometries, it follows that estimates~\eqref{e.estimates-SMTgraph-straight-all-cases} hold true for $\SMTGraphMap$ instead of $\SMTGraphMapStraight$, up to a formal replacement of $\VarConditionInitiale \in \StableSpaceStraight$ by $\VarPhaseSpaceStable \in \StableSpace$. To conclude, it suffices to remark that these estimates are valid for all $\ExpDecaySpeed \in \ExpDecaySpeedInterval$ satisfying~\eqref{e-hyp-norme-differential-exp-decay-Euclidean-norm}. It is straightforward to check that the function $\DistanceBordSpectre{\OriginDiff}{\ExpDecaySpeed}$ defined for all $\ExpDecaySpeed \in ] \SpectreUp{\StableMatrix},\min(0,\SpectreDown{\UnstableMatrix})]$ satisfying~\eqref{e-hyp-norme-differential-exp-decay-Euclidean-norm} is maximal at the point $\min\left(0,(\SpectreUp{\StableMatrix}+\SpectreDown{\UnstableMatrix}\right)/2)$ and its maximum is more than $(2^{\DimPhaseSpace-1} \gphs{\OriginDiff})^{-1}$, where $\gphs{\OriginDiff}$ is defined by~\eqref{e.gap-PHS}. Letting $\ExpDecaySpeed$ tend to $\min\left(0,(\SpectreUp{\StableMatrix}+\SpectreDown{\UnstableMatrix}\right)/2)$ in estimates~\eqref{e.estimates-SMTgraph-straight-all-cases}, it follows that estimates~\eqref{e.estimates-SMT-graph-global-estimates} hold true for some constants $\ConstI[0], \ConstI[1],\dots$ independant of $\VFParam$, $(\StableSpace,\UnstableSpace)$, $\VarConditionInitiale$ and $\Parameter$.

It remains to prove that item~\eqref{item.local-stable-set-cas-global} holds true.
Fix $\Parameter \in \SetParameter$. Let $(z,v)$ be an orbit of $\VFParam_\Parameter$. By definition of $\VFParamStraight_\Parameter$, $L(z,v)$ is an orbit of $\VFParamStraight_\Parameter$. According to~\eqref{e.estimate-FP-lemma}, we have, for all $t \geq 0$,
\begin{equation*}
	\norme{L(z(t),v(t))} \leq \frac{2 \sqrt{2} \ConstanteControleExpMatrixNormeII{\StableMatrix}}{\DistanceBordSpectre{\OriginDiff}{\ExpDecaySpeed}^{\frac{\DimStableSpaceStraight-1}{\DimPhaseSpace-1}}} 	\norme{L(z(0),v(0))} 
\end{equation*}
so, using~\eqref{e.norme-conjuguaison}, we get
\begin{equation*}
\norme{(z(t),v(t))} \leq \frac{4 \ConstanteControleExpMatrixNormeII{\StableMatrix} \ConstanteInegTriangInverse{\StableSpace,\UnstableSpace}}{\DistanceBordSpectre{\OriginDiff}{\ExpDecaySpeed}^{\frac{\DimStableSpaceStraight-1}{\DimPhaseSpace-1}}} \norme{(z(0),v(0))} \\
\end{equation*}
Letting $\ExpDecaySpeed$ tend to $\min\left(0,(\SpectreUp{\StableMatrix}+\SpectreDown{\UnstableMatrix}\right)/2)$ in the above estimate, there exists a positive constant $\ConstantEstimateStableManifoldThm$ (independant of $\VFParam$, $(\StableSpace,\UnstableSpace)$, $\VarConditionInitiale$, $\Parameter$ and $(z,v)$) such that for all $t \geq 0$,
\begin{equation*}
\norme{(z(t),v(t))}	\leq \ConstantEstimateStableManifoldThm \ConstanteControleExpMatrix{\OriginDiff} \gphs{\OriginDiff} \norme{(z(0),v(0))}
\end{equation*}
The above estimate implies that for every $ \eta >0$, for every $0 < \delta \leq \frac{\eta}{\ConstantEstimateStableManifoldThm \ConstanteControleExpMatrix{\OriginDiff} \gphs{\OriginDiff}}$, we have
\begin{equation*}
\SMexp{0,\VFParam_\Parameter}{\ExpDecaySpeed} \cap B_{\PhaseSpaceVF}(0,\delta) \subset \LSMexp{0,\VFParam_\Parameter}{\ExpDecaySpeed}{\eta} \cap B_{\PhaseSpaceVF}(0,\delta)
\end{equation*}
The other inclusion being straightforward, item~\eqref{item.local-stable-set-cas-global} holds true.
This concludes the proof of proposition~\ref{prop.variete-gamma-stable-global-version}.
\end{proof}

\subsection{Local estimates}
\label{section.local-estimates}
In this section, we state and prove a precise version of the local stable manifold theorem~\ref{thm.variete-gamma-stable-local-version}. Given a parameter $\TailleBoule >0$ and a smooth family of vector fields $(\VFParam_\Parameter)_{\Parameter \in \SetParameter}$ satisfying the hypotheses~\ref{hypo.singularity} and~\ref{hypo.PHS}, let
\begin{equation*}
\NormeSupHigherDerivativesLoc{1}{\VFParam,\TailleBoule}\egaldef \sup_{(\VariablePhaseSpace,\Parameter) \in \overline{B}((0,0),\TailleBoule)} \normesub{\DiffPartiel{\VFParam}{\VariablePhaseSpace}(\VariablePhaseSpace,\Parameter)-\OriginDiff} 
\end{equation*}
where $\OriginDiff:=\DiffPartielPoint{\VFParam}{0,0}{\VariablePhaseSpace}$, and for every integer $k \geq 2$, let
\begin{equation*}
\NormeSupHigherDerivativesLoc{k}{\VFParam,\TailleBoule} \egaldef  \sup_{2 \leq j \leq k} \sup_{(\VariablePhaseSpace,\Parameter)\in \overline{B}((0,0),\TailleBoule)} \normesub{\DiffPointSup{\VFParam}{\VariablePhaseSpace,\Parameter}{j}}
\end{equation*}
where $\overline{B}((0,0),\TailleBoule)$ is the closed ball in $\PhaseSpaceVF \times \SetParameter$ of center $(0,0)$ and radius $\TailleBoule$ and let
\begin{equation}\label{e.norme-max-cas-local}
	\NormeSupHigherDerivativesLocMax{k}{\VFParam,\TailleBoule} \egaldef \max\left(1, \NormeSupHigherDerivativesLoc{k}{\VFParam,\TailleBoule}\right)
\end{equation}

\begin{thm}[Local estimates for the stable manifold theorem with parameters]\label{thm.variete-gamma-stable-local-version}
There exists a positive constant $\ConstII \geq 1$ and a sequence of positive constants $(\ConstII[k])_{k \in \N}$ (both depending on the dimension $\DimPhaseSpace$) such that for every smooth family of vector fields $(\VFParam_\Parameter)_{\Parameter \in \SetParameter}$ satisfying the hypotheses~\ref{hypo.singularity} and~\ref{hypo.PHS}, for every partially hyperbolic splitting $(\StableSpace,\UnstableSpace)$ of $\OriginDiff:= \DiffPartielPoint{\VFParam}{0,0}{\VariablePhaseSpace}$ and for every $\TailleBoule >0$, we have:
\begin{enumerate}
\item\label{item.uniqueness} Uniqueness of the stable sets: for every $\ExpDecaySpeed, \ExpDecaySpeed' \in \ExpDecaySpeedInterval$ (see~\eqref{e.exp-decay-speed-interval}), for every $\Parameter \in \SetParameter$ such that
\begin{equation*}
	\norme{\Parameter} \leq \frac{1}{\ConstII} \min\left(\frac{\min\left(\DistanceBordSpectre{\OriginDiff}{\ExpDecaySpeed},\DistanceBordSpectre{\OriginDiff}{\ExpDecaySpeed'}\right)}{\ConstanteControleExpMatrix{\OriginDiff} \NormeSupHigherDerivativesLocMax{2}{\VFParam,\TailleBoule}}, \TailleBoule\right)
\end{equation*}
we have
	\begin{equation}
	\SMexp{0,\VFParam_\Parameter}{\ExpDecaySpeed} = \SMexp{0,\VFParam_\Parameter}{\ExpDecaySpeed'}
	\end{equation}
\item\label{item.graph-structure-local-estimate} Graph structure: there exists a (non unique) smooth map
\begin{equation*}
\fonction{\SMTGraphMapLocal}{\StableSpace \times \SetParameter}{\UnstableSpace}{(\VarPhaseSpaceStable,\Parameter)}{\SMTGraphMapLocal_\Parameter(\VarPhaseSpaceStable)}
\end{equation*}
such that for every $\ExpDecaySpeed \in \ExpDecaySpeedInterval$, for every $\Parameter \in \SetParameter$ satisfying
\begin{equation*}
\norme{\Parameter} \leq \frac{1}{\ConstII} \min\left(\frac{\min\left(\DistanceBordSpectre{\OriginDiff}{\ExpDecaySpeed},\gphs{\OriginDiff}^{-1}\right)}{\ConstanteControleExpMatrix{\OriginDiff} \NormeSupHigherDerivativesLocMax{2}{\VFParam,\TailleBoule}}, \TailleBoule \right)
\end{equation*}
for every $0< \eta \leq \tilde{\eta}$ and for every $0 < \delta \leq \frac{\eta}{\ConstantEstimateStableManifoldThm \ConstanteControleExpMatrix{\OriginDiff} \gphs{\OriginDiff}}$, we have
\begin{equation}\label{e.graph-local-structure}
\LSMexp{0,\VFParam_\Parameter}{\ExpDecaySpeed}{\eta} \cap B_{\PhaseSpaceVF}(0,\delta) = \Graph \left(\SMTGraphMapLocal_\Parameter\right) \cap B_{\PhaseSpaceVF}(0,\delta)
\end{equation}
where
\begin{equation*}
\tilde{\eta} = \frac{1}{\ConstII} \min \left( \left(\gphs{\OriginDiff} \ConstanteControleExpMatrix{\OriginDiff} \NormeSupHigherDerivativesLocMax{2}{\VFParam,\TailleBoule}\right)^{-1}, r \right) \quad \text{(see~\eqref{e.ConstanteControleExpMatrixNorme} and~\eqref{e.gap-PHS})}
\end{equation*}
\item \label{item.estimates-SMT-graph-local-estimates} Controls on $\SMTGraphMapLocal$: for every $(\VarPhaseSpaceStable,\Parameter) \in B_{\StableSpace}\left(0,\tilde{\delta}\right) \times B_{\SetParameter}\left(0,\tilde{\delta}\right)$,
	\begin{subequations}
		\label{e.estimates-SMT-graph-local-estimates}
		\begin{align}
		\label{e-controle-phi-0-local-estimates}
		\norme{\SMTGraphMapLocal(\VarPhaseSpaceStable,\Parameter)} &\leq \ConstII[0] \gphs{\OriginDiff}^2 \ConstanteControleExpMatrix{\OriginDiff}^2 \NormeSupHigherDerivativesLoc{2}{\VFParam,\TailleBoule} (\norme{\VarPhaseSpaceStable}+\norme{\Parameter}) \norme{\VarPhaseSpaceStable} \\
		\label{e-controle-phi-1-derivative-stable-variable-local-estimates}
		\normesub{\DiffPartielPoint{\SMTGraphMapLocal}{\VarPhaseSpaceStable,\Parameter}{z}} &\leq \ConstII[1]
		\gphs{\OriginDiff}^2
		\ConstanteControleExpMatrix{\OriginDiff}^2 \NormeSupHigherDerivativesLoc{2}{\VFParam,\TailleBoule} (\norme{\VarPhaseSpaceStable}+\norme{\Parameter})\\
		\label{e-controle-phi-1-derivative-parameter-variable-local-estimates}
		\normesub{\DiffPartielPoint{\SMTGraphMapLocal}{\VarPhaseSpaceStable,\Parameter}{\Parameter}} &\leq
		\ConstII[1]
		\gphs{\OriginDiff}
		\ConstanteControleExpMatrix{\OriginDiff} \NormeSupHigherDerivativesLoc{2}{\VFParam,\TailleBoule} \norme{\VarPhaseSpaceStable}
		\end{align}
where
\begin{equation*}
	\tilde{\delta} = \frac{1}{\ConstII \gphs{\OriginDiff} \ConstanteControleExpMatrix{\OriginDiff}} \min \left( \left(\gphs{\OriginDiff} \ConstanteControleExpMatrix{\OriginDiff} \NormeSupHigherDerivativesLocMax{2}{\VFParam,\TailleBoule}\right)^{-1}, r \right) 
\end{equation*}
		and more generally, using the norm $\norme{(\VarPhaseSpaceStable,\Parameter)}=\norme{\VarPhaseSpaceStable}+\norme{\Parameter}$ on $\StableSpace \times \SetParameter$, we have, for all $k \geq 2$,
		\begin{equation}\label{e-controle-phi-2-local-estimates}
		\normesub{\DiffPointSup{\SMTGraphMapLocal}{\VarPhaseSpaceStable,\Parameter}{k}} \leq
		\ConstII[k] \left(\gphs{\OriginDiff}^2 \ConstanteControleExpMatrix{\OriginDiff}^2\max\left(\gphs{\OriginDiff} \ConstanteControleExpMatrix{\OriginDiff} \NormeSupHigherDerivativesLocMax{2}{\VFParam,\TailleBoule},\TailleBoule^{-1} \right)^{k-1}
		\NormeSupHigherDerivativesLocMax{k+1}{\VFParam,\TailleBoule} \right)^{2k-1}
		\end{equation}
	\end{subequations}
\end{enumerate}
\end{thm}

\begin{rem}
If the singularity is hyperbolic ($\SpectreDown{\OriginDiff_{|\UnstableSpace}} > 0$), then the global $\ExpDecaySpeed$-stable set $\SMexp{0,\VFParam_\Parameter}{\ExpDecaySpeed}$ coincide with the global stable set $\SM{0,\VFParam_\Parameter}$ (for $\Parameter$ sufficiently small). 
\end{rem}

\begin{rem}
	If one is working with a different norm than the Euclidean one, one will have the same result but with different constants $\ConstII,\ConstII[0], \ConstII[1],\dots$
\end{rem}

\begin{proof}[Proof of theorem~\ref{thm.variete-gamma-stable-local-version}]
Fix a smooth family of vector fields $(\VFParam_\Parameter)_{\Parameter \in \SetParameter}$ satisfying the hypotheses~\ref{hypo.singularity} and~\ref{hypo.PHS}, a partially hyperbolic splitting $(\StableSpace,\UnstableSpace)$ of $\OriginDiff=\DiffPartielPoint{\VFParam}{0,0}{\VariablePhaseSpace}$ and $\TailleBoule >0$.
	Fix a smooth "plateau" map $\chi:[0,+\infty] \to [0,1]$ such that
	\begin{equation*}
		\chi(u)=\begin{cases}
		1 & \text{if $0 \leq u \leq 1$} \\
		0 & \text{if $u \geq 2$}
		\end{cases}
	\end{equation*}
For every $k \geq 1$, let $a_k=\max\left(1,\sup_{u \geq 0}\abs{\chi^{(k)}(u)}\right)$.
For any $0 <\TailleTronc \leq 1$, let us define the "truncated" smooth family of vector fields $(\TroncVFParam{\TailleTronc}_\Parameter)_{\Parameter \in \SetParameter}$ by
\begin{equation*}
\forall (\VariablePhaseSpace,\Parameter) \in \PhaseSpaceVF \times \SetParameter,
	\TroncVFParam{\TailleTronc}(\VariablePhaseSpace,\Parameter) = \OriginDiff \VariablePhaseSpace + \chi\left(\frac{\norme{(\VariablePhaseSpace,\Parameter)}^2}{\TailleTronc^2}\right)\theta(\VariablePhaseSpace,\Parameter)
\end{equation*}
where $\theta(\VariablePhaseSpace,\Parameter)= \VFParam(\VariablePhaseSpace,\Parameter)-\OriginDiff \VariablePhaseSpace$.
We now state a lemma about $\TroncVFParam{\TailleTronc}$.
\begin{lemme}\label{lemma.Tronc-VFParam}
	There exists a sequence of constants $(c_k)_{k \geq 1}$, $c_k \geq 1$, independant of $\VFParam$, $(\StableSpace,\UnstableSpace)$ and $\TailleBoule$, such that for every $0 <\TailleTronc \leq \min(1,\TailleBoule/\sqrt{2})$,
\begin{enumerate}
	\item $\TroncVFParam{\TailleTronc}$ is a smooth family of vector fields satisfying the hypotheses~\ref{hypo.singularity},~\ref{hypo.PHS} and~\ref{hypo.Hyp-derivatives-bounded}.
	
	\item\label{e.same-origin-diff} $\DiffPartielPoint{\TroncVFParam{\TailleTronc}}{0,0}{\VariablePhaseSpace}= \OriginDiff$.
	
	\item We have the following estimates on the derivatives of $\TroncVFParam{\TailleTronc}$:
\begin{subequations}
	\label{e.controle-derivatives-tronc-VFParam}
	\begin{align}
	\label{e.controle-champ-tronc-derivee-1}
	\NormeSupHigherDerivativesBounded{1}{\TroncVFParam{\TailleTronc}} &\leq c_1 \TailleTronc \NormeSupHigherDerivativesLoc{2}{\VFParam,\TailleBoule} \\
	\label{e.controle-champ-tronc-derivee-superieures}
	\forall k \geq 2,	\NormeSupHigherDerivativesBounded{k}{\TroncVFParam{\TailleTronc}} &\leq c_k \TailleTronc^{2-k} \NormeSupHigherDerivativesLoc{k}{\VFParam,\TailleBoule}
	\end{align}
\end{subequations}
\end{enumerate}
Moreover, for every $\ExpDecaySpeed \in \ExpDecaySpeedInterval$, for every $0< \TailleTronc \leq \TailleTronc(\ExpDecaySpeed)$ where
\begin{equation}\label{e.Taille-tronc-associated-with-expdecayspeed}
\TailleTronc(\ExpDecaySpeed) \egaldef \min\left(\frac{1}{c_1\ConstI}\frac{\DistanceBordSpectre{\OriginDiff}{\ExpDecaySpeed}}{\ConstanteControleExpMatrix{\OriginDiff} \NormeSupHigherDerivativesLocMax{2}{\VFParam,\TailleBoule}}, \frac{\TailleBoule}{\sqrt{2}}\right)  \in \left(0,\min(1,\TailleBoule/\sqrt{2})\right]
\end{equation}
for every $\Parameter \in \SetParameter$ such that $\norme{\Parameter} \leq \TailleTronc/2$, for every $0 < \eta \leq \TailleTronc/2$, for every $0 < \delta \leq \frac{\eta}{\ConstantEstimateStableManifoldThm \ConstanteControleExpMatrix{\OriginDiff} \gphs{\OriginDiff}}$, we have
\begin{equation}\label{e.LSM-tronc-VFParam}
\LSMexp{0,\VFParam_\Parameter}{\ExpDecaySpeed}{\eta} \cap B_{\PhaseSpaceVF}(0,\delta) = \SMexp{0,\TroncVFParam{\TailleTronc}_\Parameter}{\ExpDecaySpeed} \cap B_{\PhaseSpaceVF}(0,\delta)
\end{equation}
\end{lemme}

\begin{proof}[Proof of lemma~\ref{lemma.Tronc-VFParam}]
Fix $0 <\TailleTronc \leq \min(1,\TailleBoule/\sqrt{2})$. Let $(\VariablePhaseSpace,\Parameter) \in \PhaseSpaceVF \times \SetParameter$. By definition of $\chi$, we have
\begin{equation}\label{e.valeurs-truncated-VF}
	\TroncVFParam{\TailleTronc}(\VariablePhaseSpace,\Parameter) =
	\begin{cases}
	\VFParam(\VariablePhaseSpace,\Parameter) & \text{if $\norme{(\VariablePhaseSpace,\Parameter)} \leq \TailleTronc$}\\
	\OriginDiff \VariablePhaseSpace & \text{if $\norme{(\VariablePhaseSpace,\Parameter)} \geq \TailleTronc \sqrt{2}$}
	\end{cases}
\end{equation}
It follows from~\eqref{e.valeurs-truncated-VF} that $\DiffPartielPoint{\TroncVFParam{\TailleTronc}}{0,0}{\VariablePhaseSpace}= \OriginDiff$ and $\TroncVFParam{\TailleTronc}$ satisfies the hypotheses~\ref{hypo.singularity},~\ref{hypo.PHS} and~\ref{hypo.Hyp-derivatives-bounded}.

We are now going to prove estimates~\eqref{e.controle-derivatives-tronc-VFParam}. According to~\eqref{e.valeurs-truncated-VF}, we only need estimates on the derivatives of $\TroncVFParam{\TailleTronc}$ on $B_{\PhaseSpaceVF \times \SetParameter}(0,\TailleTronc \sqrt{2})$.
As in the proof of lemma~\ref{lemma.FP-smooth-estimates} (see~\eqref{e.binary-relationship-constant-free}), we introduce a notation to avoid clutter with constants independant of $\VFParam,\TailleTronc,\TailleBoule,\VariablePhaseSpace$ and $\Parameter$ in the following estimates:
for any real positive functions $\delta_1,\delta_2$ depending on $(\VFParam,\TailleTronc,\TailleBoule,\VariablePhaseSpace,\Parameter)$ where $0 < \TailleTronc \leq \min(1,\TailleBoule/\sqrt{2})$ and $(\VariablePhaseSpace,\Parameter) \in B(0,\TailleTronc \sqrt{2})$, we define the binary relationship $\precsim$ by
\begin{equation}\label{e.binary-relationship-constant-free-2}
\delta_1 \precsim \delta_2 \iff \exists C >0, \; \delta_1 \leq C \delta_2
\end{equation}
We will use the abuse of notation $\delta_1(\VFParam,\TailleTronc,\TailleBoule,\VariablePhaseSpace,\Parameter) \precsim \delta_2(\VFParam,\TailleTronc,\TailleBoule,\VariablePhaseSpace,\Parameter)$.
Using the fact that $\theta(0,0) = 0$, $\DiffPoint{\theta}{0,0}=0$ and $\DiffSup{\theta}{k}=\DiffSup{\VFParam}{k}$ for all $k \geq 2$, if follows from the mean value theorem that
\begin{equation}\label{e.controle-part-non-linear-lemme-troncVFParam}
\begin{split}
	\norme{\theta(\VariablePhaseSpace,\Parameter)} &\precsim \TailleTronc^2 \NormeSupHigherDerivativesLoc{2}{\VFParam,\TailleBoule} \\
	\normesub{\DiffPoint{\theta}{\VariablePhaseSpace,\Parameter}} & \precsim \TailleTronc \NormeSupHigherDerivativesLoc{2}{\VFParam,\TailleBoule} \\
\forall k \geq 2,	\NormeSupHigherDerivativesLoc{k}{\theta} &= \NormeSupHigherDerivativesLoc{k}{\VFParam,\TailleBoule}
\end{split}
\end{equation}
For every $(\VariablePhaseSpace,\Parameter) \in B(0,\TailleTronc \sqrt{2})$, let
\begin{equation*}
	N^{\TailleTronc}(\VariablePhaseSpace,\Parameter) = \frac{\norme{(\VariablePhaseSpace,\Parameter)}^2}{\TailleTronc^2}
\end{equation*}
and let $\chi^\TailleTronc = \chi \circ N^{\TailleTronc}$. Using the standard Fa\`a di Bruno's formula, we have, for all $j \geq 1$,
\begin{equation}\label{e.control-chi-taille-tronc}
	\normesub{\DiffPointSup{\chi^\TailleTronc}{\VariablePhaseSpace,\Parameter}{j}} \precsim \TailleTronc^{-j}
\end{equation}
Using estimates~\eqref{e.controle-part-non-linear-lemme-troncVFParam} and~\eqref{e.control-chi-taille-tronc}, we have
\begin{equation*}
	\normesub{\DiffPoint{\TroncVFParam{\TailleTronc}}{\VariablePhaseSpace,\Parameter}-A} \precsim \TailleTronc^{-1} \TailleTronc^2 \NormeSupHigherDerivativesLoc{2}{\VFParam,\TailleBoule}+\TailleTronc \NormeSupHigherDerivativesLoc{2}{\VFParam,\TailleBoule} \precsim \TailleTronc \NormeSupHigherDerivativesLoc{2}{\VFParam,\TailleBoule}
\end{equation*}
Since $A = \DiffPartielPoint{\TroncVFParam{\TailleTronc}}{0,0}{\VariablePhaseSpace}$, it follows that estimate~\eqref{e.controle-champ-tronc-derivee-1} holds true for some constant $c_1 \geq 1$ independant of $\VFParam$, $(\StableSpace,\UnstableSpace)$, $\TailleTronc$ and $\TailleBoule$.
Using Leibniz formula and estimates~\eqref{e.controle-part-non-linear-lemme-troncVFParam},~\eqref{e.control-chi-taille-tronc}, we have, for all $k\geq 2$,
\begin{equation*}
	\normesub{\DiffPointSup{(\chi^\TailleTronc \theta)}{\VariablePhaseSpace,\Parameter}{k}} \precsim \TailleTronc^{2-k}\NormeSupHigherDerivativesLoc{k}{\VFParam,\TailleBoule}
\end{equation*}
Since $\DiffSup{(\chi^\TailleTronc \theta)}{k} = \DiffSup{\TroncVFParam{\TailleTronc}}{k}$ for all $k \geq 2$, it follows that~\eqref{e.controle-champ-tronc-derivee-superieures} holds true for some constant $c_k \geq 1$ independant of $\VFParam$, $(\StableSpace,\UnstableSpace)$, $\TailleTronc$ and $\TailleBoule$.

Now, let us fix $\ExpDecaySpeed \in \ExpDecaySpeedInterval$. Let $0< \TailleTronc \leq \TailleTronc(\ExpDecaySpeed)$ (see~\eqref{e.Taille-tronc-associated-with-expdecayspeed}). According to~\eqref{e.controle-champ-tronc-derivee-1}, condition~\eqref{e-hyp-norme-differential-exp-decay} is satisfied for $\ExpDecaySpeed$ and $\TroncVFParam{\TailleTronc}$ so according to item~\eqref{item.local-stable-set-cas-global} of proposition~\ref{prop.variete-gamma-stable-global-version}, we have, for every $\Parameter \in \SetParameter$, for every $\eta >0$ and for every $0 < \delta \leq \frac{\eta}{\ConstantEstimateStableManifoldThm \ConstanteControleExpMatrix{\OriginDiff} \gphs{\OriginDiff}}$,
\begin{equation*}
	\LSMexp{0,\TroncVFParam{\TailleTronc}_\Parameter}{\ExpDecaySpeed}{\eta} \cap B_{\PhaseSpaceVF}(0,\delta) = \SMexp{0,\TroncVFParam{\TailleTronc}_\Parameter}{\ExpDecaySpeed} \cap B_{\PhaseSpaceVF}(0,\delta)
\end{equation*}
According to~\eqref{e.valeurs-truncated-VF}, for every $\Parameter \in \SetParameter$ such that $\norme{\Parameter} \leq \TailleTronc/2$ and for every $0 < \eta \leq \TailleTronc/2$, we have
\begin{equation*}
	\LSMexp{0,\VFParam_\Parameter}{\ExpDecaySpeed}{\eta} = \LSMexp{0,\TroncVFParam{\TailleTronc}_\Parameter}{\ExpDecaySpeed}{\eta}
\end{equation*}
so~\eqref{e.LSM-tronc-VFParam} holds true.
This concludes the proof of lemma~\ref{lemma.Tronc-VFParam}.
\end{proof}

According to lemma~\ref{lemma.Tronc-VFParam}, for every $0 <\TailleTronc \leq \min(1,\TailleBoule/\sqrt{2})$, $\TroncVFParam{\TailleTronc}$ is a smooth family of vector fields satisfying the hypotheses~\ref{hypo.singularity},~\ref{hypo.PHS} and~\ref{hypo.Hyp-derivatives-bounded} and $(\StableSpace,\UnstableSpace)$ is a partially hyperbolic splitting of $\DiffPartielPoint{\TroncVFParam{\TailleTronc}}{0,0}{\VariablePhaseSpace}=\OriginDiff$. Denote by $\SMTGraphMapII^\TailleTronc$ the smooth map associated with $\TroncVFParam{\TailleTronc}$ and $(\StableSpace,\UnstableSpace)$ by proposition~\ref{prop.variete-gamma-stable-global-version} (well defined for all $\TailleTronc$ small enough by~\eqref{e.controle-champ-tronc-derivee-1}).

Let $\ExpDecaySpeed, \ExpDecaySpeed' \in \ExpDecaySpeedInterval$. Let $\TailleTronc = \min(\TailleTronc(\ExpDecaySpeed),\TailleTronc(\ExpDecaySpeed'))$ (see~\eqref{e.Taille-tronc-associated-with-expdecayspeed}).
Estimate~\eqref{e.controle-champ-tronc-derivee-1} implies that $\ExpDecaySpeed$ and $\ExpDecaySpeed'$ satisfy~\eqref{e-hyp-norme-differential-exp-decay} for $\TroncVFParam{\TailleTronc}$. In particular $\SMTGraphMapII^\TailleTronc$ is well defined. Let $\Parameter \in \SetParameter$ such that $\norme{\Parameter} \leq \TailleTronc/2$, let $0 < \eta \leq \TailleTronc/2$ and let $0 < \delta \leq \frac{\eta}{\ConstantEstimateStableManifoldThm \ConstanteControleExpMatrix{\OriginDiff} \gphs{\OriginDiff}}$.
We have
\begin{align*}
\LSMexp{0,\VFParam_\Parameter}{\ExpDecaySpeed}{\eta} \cap B_{\PhaseSpaceVF}(0,\delta)
&= \SMexp{0,\TroncVFParam{\TailleTronc}_\Parameter}{\ExpDecaySpeed} \cap B_{\PhaseSpaceVF}(0,\delta) & \text{using~\eqref{e.LSM-tronc-VFParam}}\\
&= \Graph \left(\SMTGraphMapII^\TailleTronc_\Parameter\right) \cap B_{\PhaseSpaceVF}(0,\delta) & \text{using item~\eqref{item.graph-structure-global-estimate} of proposition~\ref{prop.variete-gamma-stable-global-version}}\\
\end{align*}
and since the above computation holds true for $\ExpDecaySpeed'$ as well, it follows that
\begin{equation*}
	\LSMexp{0,\VFParam_\Parameter}{\ExpDecaySpeed}{\eta} \cap B_{\PhaseSpaceVF}(0,\delta) = \LSMexp{0,\VFParam_\Parameter}{\ExpDecaySpeed'}{\eta} \cap B_{\PhaseSpaceVF}(0,\delta)
\end{equation*}
and finally,
\begin{equation}\label{e.uniqueness-property}
\SMexp{0,\VFParam_\Parameter}{\ExpDecaySpeed} = \SMexp{0,\VFParam_\Parameter}{\ExpDecaySpeed'}
\end{equation}
It follows that item~\eqref{item.uniqueness} of theorem~\ref{thm.variete-gamma-stable-local-version} holds true.

Let
\begin{equation*}
	\tilde{\ExpDecaySpeed} \egaldef \frac{\SpectreUp{\OriginDiff_{|\StableSpace}}+\min\left(0,\SpectreDown{\OriginDiff_{|\UnstableSpace}}\right)}{2}
\end{equation*}
One can remark that
\begin{equation*}
	\DistanceBordSpectre{\OriginDiff}{\tilde{\ExpDecaySpeed}} \geq \left(2^{n-1}\gphs{\OriginDiff}\right)^{-1}
\end{equation*}
Let
\begin{equation}\label{e.Taille-tronc-canonique}
\tilde{\TailleTronc} \egaldef  \min\left(\left(c_1\ConstI2^{n-1}\gphs{\OriginDiff} \ConstanteControleExpMatrix{\OriginDiff} \NormeSupHigherDerivativesLocMax{2}{\VFParam,\TailleBoule}\right)^{-1},\frac{\TailleBoule}{\sqrt{2}}\right) \leq \TailleTronc(\tilde{\ExpDecaySpeed})
\end{equation}
Let $\SMTGraphMapLocal \egaldef \SMTGraphMapII^{\tilde{\TailleTronc}}$.
Estimate~\eqref{e.controle-champ-tronc-derivee-1} implies that $\tilde{\ExpDecaySpeed}$ satisfies~\eqref{e-hyp-norme-differential-exp-decay} for $\TroncVFParam{\tilde{\TailleTronc}}$ so $\SMTGraphMapLocal$ is well defined.
According to proposition~\ref{prop.variete-gamma-stable-global-version} and lemma~\ref{lemma.Tronc-VFParam}, for every $\Parameter \in \SetParameter$ such that $\norme{\Parameter} \leq \tilde{\TailleTronc}/2$, for every $0< \eta \leq \tilde{\TailleTronc}/2$ and for every $0 < \delta \leq \frac{\eta}{\ConstantEstimateStableManifoldThm \ConstanteControleExpMatrix{\OriginDiff} \gphs{\OriginDiff}}$, we have
\begin{equation*}
\LSMexp{0,\VFParam_\Parameter}{\tilde{\ExpDecaySpeed}}{\eta} \cap B_{\PhaseSpaceVF}(0,\delta) = \Graph \left(\SMTGraphMapLocal_\Parameter\right) \cap B_{\PhaseSpaceVF}(0,\delta)
\end{equation*}
According to~\eqref{e.uniqueness-property}, for every $\ExpDecaySpeed \in \ExpDecaySpeedInterval$, for every $\Parameter \in \SetParameter$ such that $\norme{\Parameter} \leq \min(\tilde{\TailleTronc},\TailleTronc(\ExpDecaySpeed))/2$, for every $0< \eta \leq \tilde{\TailleTronc}/2$ and for every $0 < \delta \leq \frac{\eta}{\ConstantEstimateStableManifoldThm \ConstanteControleExpMatrix{\OriginDiff} \gphs{\OriginDiff}}$, we have
\begin{equation}\label{e.LSM-graph-phi}
\LSMexp{0,\VFParam_\Parameter}{\ExpDecaySpeed}{\eta} \cap B_{\PhaseSpaceVF}(0,\delta) = \Graph \left(\SMTGraphMapLocal_\Parameter\right) \cap B_{\PhaseSpaceVF}(0,\delta)
\end{equation}
so item~\eqref{item.graph-structure-local-estimate} of theorem~\ref{thm.variete-gamma-stable-local-version} holds true.

We are now going to prove estimates~\eqref{e.estimates-SMT-graph-local-estimates}.
Using~\eqref{e.LSM-tronc-VFParam}, one can remark that for every $0 < \TailleTronc \leq \tilde{\TailleTronc}$, for every $\norme{\Parameter} \leq \TailleTronc/2$, we have
\begin{equation*}\label{e.graph-locally-coincide}
	\Graph \left(\SMTGraphMapLocal_\Parameter\right) \cap B_{\PhaseSpaceVF}(0,\delta(\TailleTronc)) = \Graph \left(\SMTGraphMapII^\TailleTronc_\Parameter\right) \cap B_{\PhaseSpaceVF}(0,\delta(\TailleTronc))
\end{equation*}
where
\begin{equation*}
	 \delta(\TailleTronc) \egaldef \frac{\TailleTronc}{2\ConstantEstimateStableManifoldThm \ConstanteControleExpMatrix{\OriginDiff} \gphs{\OriginDiff}}
\end{equation*}
It follows that for every $0 < \TailleTronc \leq \tilde{\TailleTronc}$, for every $\norme{\Parameter} \leq \TailleTronc/2$ and for every $\VarPhaseSpaceStable \in \StableSpace$ such that $\norme{\VarPhaseSpaceStable + \SMTGraphMapLocal_\Parameter(\VarPhaseSpaceStable)} <  \delta(\TailleTronc)$, we have
\begin{equation*}
\SMTGraphMapLocal_\Parameter(\VarPhaseSpaceStable) = \SMTGraphMapII^\TailleTronc_\Parameter(\VarPhaseSpaceStable)
\end{equation*}
In order to obtain the estimates about $\SMTGraphMapLocal$ and its derivatives at a given point $(\VarPhaseSpaceStable,\Parameter)$, the idea is to remark that it will be the same estimates for $\SMTGraphMapII^\TailleTronc$ for some well chosen $\TailleTronc=\TailleTronc(\VarPhaseSpaceStable,\Parameter)$. Plugging~\eqref{e.controle-champ-tronc-derivee-1} into~\eqref{e-controle-phi-0}, we obtain, for every $(\VarPhaseSpaceStable,\Parameter) \in \StableSpace \times \SetParameter$,
\begin{equation*}\label{e.estimate-smtgraphmaplocal-derivee-0-first}
	\norme{\SMTGraphMapLocal(\VarPhaseSpaceStable,\Parameter)} \leq \frac{\ConstI[0]}{2^{n-1}\ConstI}  \norme{\VarPhaseSpaceStable}
\end{equation*}
It follows from the previous estimate that for every $(\VarPhaseSpaceStable,\Parameter) \in \StableSpace \times \SetParameter \setminus \{(0,0)\}$ such that
\begin{equation}\label{e.condition-on-varstablespace}
\begin{split}
\norme{\VarPhaseSpaceStable} &<  \left(\ConstantEstimateStableManifoldThm \ConstanteControleExpMatrix{\OriginDiff} \gphs{\OriginDiff}\right)^{-1} \min \left( \frac{1}{8c_1(\ConstI+\ConstI[0])2^{n-1}} \left(\gphs{\OriginDiff} \ConstanteControleExpMatrix{\OriginDiff} \NormeSupHigherDerivativesLocMax{2}{\VFParam,\TailleBoule}\right)^{-1},\frac{\TailleBoule}{8\sqrt{2}\left(1+\frac{\ConstI[0]}{2^{n-1}\ConstI}\right)} \right)  \\
\norme{\Parameter} &< \left(\ConstantEstimateStableManifoldThm \ConstanteControleExpMatrix{\OriginDiff} \gphs{\OriginDiff}\right)^{-1} \min \left( \frac{1}{8c_1\ConstI2^{n-1}} \left(\gphs{\OriginDiff} \ConstanteControleExpMatrix{\OriginDiff} \NormeSupHigherDerivativesLocMax{2}{\VFParam,\TailleBoule}\right)^{-1}, \frac{\TailleBoule}{8\sqrt{2}} \right) 
\end{split}
\end{equation}
the number
\begin{equation}\label{e.defin-taille-tronc-adaptee-point}
	\TailleTronc(\VarPhaseSpaceStable,\Parameter) \egaldef 4\ConstantEstimateStableManifoldThm \ConstanteControleExpMatrix{\OriginDiff} \gphs{\OriginDiff}\left(  \left(1+\frac{\ConstI[0]}{2^{n-1}\ConstI}\right) \norme{\VarPhaseSpaceStable} + \norme{\Parameter} \right)
\end{equation}
satisfies $0 < \TailleTronc(\VarPhaseSpaceStable,\Parameter) \leq \tilde{\TailleTronc}$ and the following property: for every $(\VarPhaseSpaceStable',\Parameter') \in \StableSpace \times \SetParameter$,
\begin{equation*}
 \left( \norme{\VarPhaseSpaceStable'} < 2 \norme{\VarPhaseSpaceStable} \quad \text{and} \quad
	\norme{\Parameter'} < 2 \norme{\Parameter} \right) \implies \left(\norme{\VarPhaseSpaceStable'+\SMTGraphMapLocal_{\Parameter'}(\VarPhaseSpaceStable')} < \delta(\TailleTronc(\VarPhaseSpaceStable,\Parameter)) \quad \text{and} \quad
	\norme{\Parameter'} \leq \TailleTronc(\VarPhaseSpaceStable,\Parameter)/2\right)
\end{equation*}
Let us now fix $(\VarPhaseSpaceStable,\Parameter) \in \StableSpace \times \SetParameter \setminus \{(0,0)\}$ satisfying~\eqref{e.condition-on-varstablespace}. According to the above reasoning, the maps $\SMTGraphMapLocal$ and $\SMTGraphMapII^{\TailleTronc(\VarPhaseSpaceStable,\Parameter)}$ coincide on $B_{\StableSpace}(0,2\norme{\VarPhaseSpaceStable}) \times B_{\SetParameter}(0,2\norme{\Parameter})$, in particular all their derivatives at the point $(\VarPhaseSpaceStable,\Parameter)$ coincide. We have
\begin{align*}
	\norme{\SMTGraphMapLocal(\VarPhaseSpaceStable,\Parameter)} &= \norme{\SMTGraphMapII^{\TailleTronc(\VarPhaseSpaceStable,\Parameter)}(\VarPhaseSpaceStable,\Parameter)} \\
	& \leq \ConstI[0] \gphs{\OriginDiff}\ConstanteControleExpMatrix{\OriginDiff} \NormeSupHigherDerivativesBounded{1}{\TroncVFParam{\TailleTronc(\VarPhaseSpaceStable,\Parameter)}} \norme{\VarPhaseSpaceStable} & \text{using~\eqref{e-controle-phi-0}}\\
	& \leq c_1\ConstI[0] \gphs{\OriginDiff}\ConstanteControleExpMatrix{\OriginDiff} \TailleTronc(\VarPhaseSpaceStable,\Parameter) \NormeSupHigherDerivativesLoc{2}{\VFParam,\TailleBoule} \norme{\VarPhaseSpaceStable} & \text{using~\eqref{e.controle-champ-tronc-derivee-1}} \\
	& \leq 4c_1 \ConstI[0] \ConstantEstimateStableManifoldThm \left(1+\frac{\ConstI[0]}{2^{n-1}\ConstI}\right) \gphs{\OriginDiff}^2 \ConstanteControleExpMatrix{\OriginDiff}^2 \NormeSupHigherDerivativesLoc{2}{\VFParam,\TailleBoule} \norme{\VarPhaseSpaceStable} (\norme{\VarPhaseSpaceStable}+\norme{\Parameter}) & \text{using~\eqref{e.defin-taille-tronc-adaptee-point}}
\end{align*}
so estimate~\eqref{e-controle-phi-0-local-estimates} holds true (for some different constants). By the same reasoning, we obtain estimates~\eqref{e-controle-phi-1-derivative-stable-variable-local-estimates} and~\eqref{e-controle-phi-1-derivative-parameter-variable-local-estimates}.
Using~\eqref{e.controle-champ-tronc-derivee-superieures}, we get, for all $k \geq 2$,
\begin{align*}
	\NormeSupHigherDerivativesBoundedMax{k+1}{\TroncVFParam{\tilde{\TailleTronc}}} &\leq \max\left(1, c_{k+1} \tilde{\TailleTronc}^{1-k} \NormeSupHigherDerivativesLoc{k+1}{\VFParam}\right) \\
	& \leq c_{k+1} \tilde{\TailleTronc}^{1-k}\NormeSupHigherDerivativesLocMax{k+1}{\VFParam,\TailleBoule} \\
\end{align*}
Using~\eqref{e.Taille-tronc-canonique}, we get, for all $k \geq 2$,
\begin{align*}
	\NormeSupHigherDerivativesBoundedMax{k+1}{\TroncVFParam{\tilde{\TailleTronc}}} 	& \leq c_{k+1} \max\left(\left(c_1\ConstI2^{n-1}\gphs{\OriginDiff} \ConstanteControleExpMatrix{\OriginDiff} \NormeSupHigherDerivativesLocMax{2}{\VFParam,\TailleBoule}\right)^{k-1},\left(\frac{\sqrt{2}}{\TailleBoule}\right)^{k-1}\right)  \NormeSupHigherDerivativesLocMax{k+1}{\VFParam,\TailleBoule}\\
	& \leq c_{k+1} \max\left(\left(c_1\ConstI2^{n-1}\right)^{k-1},\sqrt{2}^{k-1}\right)
	\max\left(\gphs{\OriginDiff} \ConstanteControleExpMatrix{\OriginDiff} \NormeSupHigherDerivativesLocMax{2}{\VFParam,\TailleBoule},\TailleBoule^{-1} \right)^{k-1}
	\NormeSupHigherDerivativesLocMax{k+1}{\VFParam,\TailleBoule}
\end{align*}
Plugging this estimate into~\eqref{e-controle-phi-2} applied to $\SMTGraphMapII^{\tilde{\TailleTronc}}=\SMTGraphMapLocal$, it follows that~\eqref{e-controle-phi-2-local-estimates} holds true. This concludes the proof of theorem~\ref{thm.variete-gamma-stable-local-version}.
\end{proof}

\section{Estimates for vector fields vanishing on submanifolds}
\label{e.SMF}
Fix $n \in \N$ and a linear subspace $\UnstableSpace$ of $\PhaseSpaceVF$. Denote by $\norme{.}$ the Euclidean norm on $\PhaseSpaceVF$. Let $\OpenSetDefinVF$ be an open neighbourhood of $0$ in $\PhaseSpaceVF$. Fix a smooth vector field $\FixedVF: \OpenSetDefinVF \to \PhaseSpaceVF$. Assume that
\begin{enumerate}
	\item \label{e.ligne-singularities} $\FixedVF$ vanishes on $\TraceOpenSetDefinVF:= \OpenSetDefinVF \cap \UnstableSpace$;
	\item \label{e.PHS-ligne-singularities} For every $\Parameter \in \TraceOpenSetDefinVF$, there exists a decomposition $\StableSpace_{\Parameter} \oplus \UnstableSpace = \R^n$ stabilized by $\OriginDiff_{\Parameter}:=\DiffPoint{\FixedVF}{\Parameter}$ and such that
	\begin{equation*}
		\SpectreUp{(\OriginDiff_{\Parameter})_{|\StableSpace_{\Parameter}}} <0
	\end{equation*}
\end{enumerate}

For every $\Parameter \in \TraceOpenSetDefinVF$, let
\begin{equation}
	\EVII{\Parameter} \egaldef \min\left(1,\abs{\SpectreUp{(\OriginDiff_{\Parameter})_{|\StableSpace_{\Parameter}}}}\right)^{\DimPhaseSpace-1} \quad \text{(see~\eqref{e.max-spectre})}
\end{equation}

Let $\SMF$ be the stable foliation associated with the contracted subspace $\UnstableSpace$ on which the vector field $\FixedVF$ vanishes, that is, the partition
\begin{equation*}
\SMF \egaldef \MyBigSet{\SM{\Parameter,\FixedVF}}{\Parameter \in \TraceOpenSetDefinVF}
\end{equation*}
where the stable manifolds $\SM{\Parameter,\FixedVF}$ are called the leaves of the foliation $\SMF$.

For every integer $k \geq 2$, every $\Parameter \in \TraceOpenSetDefinVF$ and every $\TailleBoule >0$ such that $\overline{B}_{\PhaseSpaceVF}(\Parameter,\TailleBoule) \subset \OpenSetDefinVF$, let
\begin{equation*}
\NormeSupHigherDerivativesLoc{k}{\FixedVF,\Parameter,\TailleBoule} \egaldef  \sup_{2 \leq j \leq k} \sup_{y \in \overline{B}_{\PhaseSpaceVF}(\Parameter,\TailleBoule)} \normesub{\DiffPointSup{\FixedVF}{y}{j}}
\end{equation*}
and let
\begin{equation*}
\NormeSupHigherDerivativesLocMax{k}{\FixedVF,\Parameter,\TailleBoule} \egaldef \max\left(1, \NormeSupHigherDerivativesLoc{k}{\FixedVF,\Parameter,\TailleBoule}\right)
\end{equation*}

The next theorem states that in this context, the foliation $\SMF$ can be locally smoothly straightened in the neighbourhood of any point $\Parameter \in \TraceOpenSetDefinVF$.

\begin{figure}
	\begin{center}
		\begin{tikzpicture}[scale=2]	
		\begin{scope}[xshift=100]
		\begin{scope}[yshift=12.5]
		\draw (0,0) node {$\bullet$};
		\draw (0,0) node[right] {$\Parameter_0$};
		\end{scope}	
		\begin{scope}[rotate=25]
		\draw[dashed] (-1,0) -- (1,0);
		\end{scope}
		\begin{scope}[yshift=25,rotate=25]
		\draw[dashed] (-1,0) -- (1,0);
		\end{scope}
		\begin{scope}[xshift=25,yshift=11.5]
		\draw[dashed] (0,0) -- (0,0.88);
		\end{scope}
		\begin{scope}[xshift=-26,yshift=-12.3]
		\draw[dashed] (0,0) -- (0,0.88);
		\end{scope}
		\draw (-0.8,0.7) node {$V_{\Parameter_0}$};
		\begin{scope}[yshift=5,rotate=25]
		\draw[purple] (-1,0) -- (1,0);
		\draw (0,0) node {$\bullet$};
		\draw (0,0) node[right] {$\Parameter$};
		\end{scope}
		\begin{scope}[xshift=-23,yshift=-10]
		\draw[purple] (0,0) node[right] {($\Parameter + \StableSpace_{\Parameter_0}) \cap V_{\Parameter_0}$};
		\end{scope}
		\begin{scope}[yshift=10,rotate=25]
		\draw[blue] (-1,0) -- (1,0);
		\end{scope}
		\begin{scope}[yshift=15,rotate=25]
		\draw[cyan] (-1,0) -- (1,0);
		\end{scope}
		\begin{scope}[yshift=20,rotate=25]
		\draw[orange] (-1,0) -- (1,0);
		\end{scope}
		\draw (0,-0.25) -- (0,1);
		\draw (0,1) node[above] {$\UnstableSpace$};
		\end{scope}		
		
		\begin{scope}[xshift=50,yshift=15]
		\draw (0,0.25) node {$\SLC_{\Parameter_0}$};
		\draw[->] (-0.5,0) to[bend left] (0.5,0);
		\end{scope}
		
		\begin{scope}[yshift=12.5]
		\draw (0,0) node {$\bullet$};
		\draw (0,0) node[right] {$\Parameter_0$};
		\end{scope}
		\begin{scope}[rotate=25]
		\draw[dashed] [domain=-1.2:0.865] plot (\x,{2/3*cos(\x r)-2/3});
		\end{scope}
		\begin{scope}[yshift=25,rotate=25]
		\draw[dashed] [domain=-1.055:0.92] plot (\x,{1/4*cos(\x r)-1/4});
		\end{scope}
		\begin{scope}[xshift=25,yshift=5]
		\draw[dashed] (0,0) -- (0,1);
		\end{scope}
		\begin{scope}[xshift=-26,yshift=-25]
		\draw[dashed] (0,0) -- (0,1.2);
		\end{scope}
		\draw (-0.8,0.7) node {$U_{\Parameter_0}$};
		\begin{scope}[yshift=5,rotate=25]
		\draw[purple] [domain=-1.2:0.865] plot (\x,{2/3*cos(\x r)-2/3});
		\draw (0,0) node {$\bullet$};
		\draw (0,0) node[right] {$\Parameter$};
		\end{scope}
		\begin{scope}[xshift=-20,yshift=-16]
		\draw[purple] (0,0) node[right] {$\SM{\Parameter,\FixedVF} \cap U_{\Parameter_0}$};
		\end{scope}
		\begin{scope}[yshift=10,rotate=25]
		\draw[blue] [domain=-1.13:0.875] plot (\x,{1/2*cos(\x r)-1/2});
		\end{scope}
		\begin{scope}[yshift=15,rotate=25]
		\draw[cyan] [domain=-1.08:0.9] plot (\x,{1/3*cos(\x r)-1/3});
		\end{scope}
		\begin{scope}[yshift=20,rotate=25]
		\draw[orange] [domain=-1.055:0.92] plot (\x,{1/4*cos(\x r)-1/4});
		\end{scope}
		\draw (0,-0.25) -- (0,1);
		\draw (0,1) node[above] {$\UnstableSpace$};
		\end{tikzpicture}
	\end{center}
	\caption{The local coordinate system $\SLC_{\Parameter_0}$ straightens the stable foliation induced by $\SMF$ on $U_{\Parameter_0}$.}
	\label{fig.param-locale-adaptée}
\end{figure}

\begin{thm}[Local straightening of the stable foliation of a vector field]\label{thm.redressement-local-feuilletage}
	There exists two positive constants $\ConstIII \geq \ConstIII' \geq 1$, a sequence of positive constants $(\ConstIII[k])_{k\geq 2}$ and a sequence of integers $(N_k)_{k \geq 2}$ (all independant of $\FixedVF$) such that for every map $\TailleBoule: \TraceOpenSetDefinVF \to \left(0,1\right]$ satisfying
	\begin{equation*}
		\forall \Parameter \in \TraceOpenSetDefinVF, \; \overline{B}_{\PhaseSpaceVF}(\Parameter,\TailleBoule(\Parameter)) \subset \OpenSetDefinVF
	\end{equation*}
	there exists
	\begin{itemize}
		\item two families $(U_{\Parameter})_{\Parameter \in \TraceOpenSetDefinVF}$ and $(V_{\Parameter})_{\Parameter \in \TraceOpenSetDefinVF}$ of open sets of $\PhaseSpaceVF$;
		\item a family of smooth diffeomorphisms
		\begin{equation*}
		(\SLC_{\Parameter}: U_{\Parameter} \to V_{\Parameter})_{\Parameter \in \TraceOpenSetDefinVF}
		\end{equation*}
	\end{itemize}
satisfying the following properties. Given $\Parameter_0 \in \TraceOpenSetDefinVF$:
	\begin{enumerate}
		\item\label{thm.size-neighbourhoods} Both $U_{\Parameter_0}$ and $V_{\Parameter_0}$ are neighbourhoods of $\Parameter_0$. More precisely, they both contain the open ball $B_{\PhaseSpaceVF}\left(\Parameter_0, R_{\Parameter_0}\right)$ where
		\begin{equation}\label{e.definition-rayon-minimal}
		R_{\Parameter_0} \egaldef
		\frac{\EVII{\Parameter_0}}{\ConstIII \ConstanteInegTriangInverse{\StableSpace_{\Parameter_0},\UnstableSpace}^2 \ConstanteControleExpMatrix{\OriginDiff_{\Parameter_0}}} \min\left( \frac{\EVII{\Parameter_0}}{ \ConstanteControleExpMatrix{\OriginDiff_{\Parameter_0}} \NormeSupHigherDerivativesLocMax{2}{\FixedVF,\Parameter_0,\TailleBoule(\Parameter_0)}}, \TailleBoule(\Parameter_0) \right)  \quad \text{(see~\eqref{e.ConstanteControleExpMatrixNorme})}
		\end{equation}
				
		\item\label{thm.straightening} $\SMF$ foliates $U_{\Parameter_0}$ and $\SLC_{\Parameter_0}$ is a local coordinate system straightening the stable foliation $\SMF$ (see figure~\ref{fig.param-locale-adaptée}). More precisely,
		\begin{equation*}
		U_{\Parameter_0}  = \bigsqcup_{\Parameter \in \TraceOpenSetDefinVF \cap U_{\Parameter_0}} \SM{\Parameter,\FixedVF} \cap U_{\Parameter_0}
		\end{equation*}
		and, for every $\Parameter \in \TraceOpenSetDefinVF \cap U_{\Parameter_0}$,
		\begin{equation*}
		\SLC_{\Parameter_0}\left( \SM{\Parameter,\FixedVF} \cap U_{\Parameter_0} \right) = \left(\Parameter + \StableSpace_{\Parameter_0}\right) \cap V_{\Parameter_0}
		\end{equation*} 
		Moreover, for every $\Parameter \in \TraceOpenSetDefinVF \cap U_{\Parameter_0}$,
		\begin{equation*}
		\SM{\Parameter,\FixedVF} \cap U_{\Parameter_0} = 	\LSMexp{\Parameter,\FixedVF}{\ExpDecaySpeed}{\eta} \cap   U_{\Parameter_0}
		\end{equation*}
		where
		\begin{align*}
		\ExpDecaySpeed &= -\frac{\abs{\SpectreUp{(\OriginDiff_{\Parameter_0})_{|\StableSpace_{\Parameter_0}}}}}{2} \\
		\eta &= \frac{1}{\ConstIII'} \min\left( \frac{\EVII{\Parameter_0}}{ \ConstanteControleExpMatrix{\OriginDiff_{\Parameter_0}} \NormeSupHigherDerivativesLocMax{2}{\FixedVF,\Parameter_0,\TailleBoule(\Parameter_0)}}, \TailleBoule(\Parameter_0) \right)
		\end{align*}
		
		\item\label{thm.local-coord-form} Identifying $\PhaseSpaceVF$ and $\StableSpace_{\Parameter_0} \times \UnstableSpace$, the local coordinate system has the following form:
		\begin{equation*}\label{e.form-local-coord-system}
		\SLC_{\Parameter_0}(z,\Parameter) = (z,\Parameter) +  \left(0,\SLCsmall_{\Parameter_0}(z,\Parameter)\right)
		\end{equation*}
		where $\SLCsmall_{\Parameter_0}(0,\Parameter)=0$.
		
		\item\label{thm.close-Id} For every $0<\epsilon \leq 1$, $\SLC_{\Parameter_0}$ restricted to $B_{\PhaseSpaceVF}\left(\Parameter_0, \epsilon R_{\Parameter_0}\right)$ is $\epsilon$-close to the identity with respect to the $C^1$-norm:
		\begin{equation*}
		\begin{split}
		\norme{\SLC_{\Parameter_0} - \Id}_{C^1} &\leq \epsilon \quad \text{in restriction to} \:\: B_{\PhaseSpaceVF}\left(\Parameter_0, \epsilon R_{\Parameter_0}\right) \\
		\norme{\SLC_{\Parameter_0}^{-1} - \Id}_{C^1} &\leq \epsilon \quad \text{in restriction to} \:\: B_{\PhaseSpaceVF}\left(\Parameter_0, \epsilon R_{\Parameter_0}\right)
		\end{split}
		\end{equation*}
		
		\item\label{thm.Ck-norms} The $C^k$-norms have a sub-polynomial growth with respect to $\EVII{\Parameter_0}^{-1}$: more precisely, for every $k \geq 2$,
		\begin{equation*}
		\norme{\SLC_{\Parameter_0}}_{C^k}, \norme{\SLC_{\Parameter_0}^{-1}}_{C^k}	\leq \ConstIII[k] \left( \frac{\ConstanteControleExpMatrix{\OriginDiff_{\Parameter_0}} \NormeSupHigherDerivativesLocMax{k+1}{\FixedVF,\Parameter_0,\TailleBoule(\Parameter_0)}}{\EVII{\Parameter_0} \TailleBoule(\Parameter_0)} \right)^{N_k}
		\end{equation*}
		
		\item\label{thm.local-unicity} For every $\Parameter_1 \in \TraceOpenSetDefinVF$, $\SLC_{\Parameter_0}$ and $\SLC_{\Parameter_1}$ "coincide" on $U_{\Parameter_0} \cap U_{\Parameter_1}$ modulo the choice of the direction on which the stable manifolds are projected. More precisely, if we denote by $\pi_\Parameter$ the linear projection along $\UnstableSpace$ onto $\StableSpace_{\Parameter}$ for every $\Parameter \in \TraceOpenSetDefinVF$, we have
		\begin{equation*}
			\SLC_{\Parameter_0}- \SLC_{\Parameter_1} = \pi_{\Parameter_0} - \pi_{\Parameter_1} \quad \text{in restriction to $U_{\Parameter_0} \cap U_{\Parameter_1}$}
		\end{equation*}
	\end{enumerate}
\end{thm}

\begin{rem}
	The charts $(\SLC_{\Parameter_0})_{\Parameter_0 \in \TraceOpenSetDefinVF}$ do not form a foliation coordinate atlas because $\SLC_{\Parameter_0}$ straightens the leaf $\SM{\Parameter,\FixedVF} \cap U_{\Parameter_0}$ onto the affine subspace $\Parameter + \StableSpace_{\Parameter_0}$ which depends on $\Parameter_0$. Nevertheless, identifying $\PhaseSpaceVF$ and $\StableSpace_{\Parameter_0} \times \UnstableSpace$, one only needs to compose $\SLC_{\Parameter_0}$ with $(\pi_{|\StableSpace_{\Parameter_0}},\Id_{\UnstableSpace})$ where $\pi$ denotes a linear projection along $\UnstableSpace$ onto a fixed complement of $\UnstableSpace$ (for example $\UnstableSpace^{\perp}$) to obtain a foliation coordinate atlas. This would change the estimates on the norms of the derivatives of $\SLC_{\Parameter_0}^{-1}$ by a factor $\ConstanteInegTriangInverse{\StableSpace_{\Parameter_0},\UnstableSpace}$ and would make $\SLC_{\Parameter_0}$ close to $(\pi_{|\StableSpace_{\Parameter_0}},\Id_{\UnstableSpace})$ in item~\eqref{thm.close-Id}. We did not make this choice for two reasons: there is no canonical complement of $\UnstableSpace$ and we want to obtain the fact that $\SLC_{\Parameter_0}$ can be made arbitrarily close to $\Id$ with respect to the $C^1$-norm.
\end{rem}

\begin{proof}
\proofstep{Presentation of the proof as a consequence of theorem~\ref{thm.variete-gamma-stable-local-version}}
Fix a map $\TailleBoule: \TraceOpenSetDefinVF \to \left(0,1\right]$ satisfying
\begin{equation*}
\forall \Parameter \in \TraceOpenSetDefinVF, \; \overline{B}(\Parameter,\TailleBoule(\Parameter)) \subset \OpenSetDefinVF
\end{equation*}
Fix $\Parameter_0 \in \TraceOpenSetDefinVF$. Even if it means translating the vector field $\FixedVF$, one can assume that $\Parameter_0 = 0$. We will then prove the desired result in the neighbourhood of $0$. 
Recall that we want to straighten, for all $\Parameter \in \UnstableSpace$ small enough, the local stable manifold $\LSM{\Parameter,\FixedVF}{\eta}$ for some $\eta$ depending on $\Parameter$. This leads us to define, for every $\Parameter \in \UnstableSpace$ and $\VariablePhaseSpace \in \PhaseSpaceVF$ such that $\Parameter + \VariablePhaseSpace \in \OpenSetDefinVF$,
\begin{equation*}
\VFParam(\VariablePhaseSpace,\Parameter) \egaldef \VFParam_\Parameter (\VariablePhaseSpace) = \FixedVF( \Parameter + \VariablePhaseSpace)
\end{equation*}
We will prove later on that the local stable manifolds of $\FixedVF$ coincide with the local $\ExpDecaySpeed$-stable manifolds for some $\ExpDecaySpeed <0$ well chosen (see~\eqref{e.choice-gamma}). We will then focus on describing those local $\ExpDecaySpeed$-stable manifolds.
The local $\ExpDecaySpeed$-stable manifold of $\Parameter \in \TraceOpenSetDefinVF$ for $\FixedVF$ is exactly the translation of the local $\ExpDecaySpeed$-stable manifold of $0$ for $\VFParam_\Parameter$ by $t_{\Parameter}:\VariablePhaseSpace \mapsto \Parameter+\VariablePhaseSpace$. More precisely, for every $\Parameter \in \TraceOpenSetDefinVF$ and for every $0<\delta \leq \TailleBoule(\Parameter)$, we have
\begin{equation}\label{e.stable-manifold-locally-coincide-1}
\Parameter+\LSMexp{0,\VFParam_\Parameter}{\ExpDecaySpeed}{\delta}= \LSMexp{\Parameter,\FixedVF}{\ExpDecaySpeed}{\delta}
\end{equation}
According to the above equation, we want to straighten the local $\ExpDecaySpeed$-stable manifolds $\LSMexp{0,\VFParam_\Parameter}{\ExpDecaySpeed}{\delta}$ for $\Parameter$ small enough.
\proofstep{Construction of $\SLC_{\Parameter_0}$}
We are now going to extend $\VFParam$ so that we can apply theorem~\ref{thm.variete-gamma-stable-local-version}. One can remark that $\VFParam$ is well defined on a neighbourhood of the closed ball $\overline{B}_{\PhaseSpaceVF \times \UnstableSpace}((0,0),\TailleBoule(\Parameter_0)/2)$. Multiplying $\VFParam$ by a smooth plateau map equal to $1$ on $\overline{B}_{\PhaseSpaceVF \times \UnstableSpace}((0,0),\TailleBoule(\Parameter_0)/2)$ and vanishing outside of a small neighbourhood of $\overline{B}_{\PhaseSpaceVF \times \UnstableSpace}((0,0),\TailleBoule(\Parameter_0)/2)$, we obtain a smooth family of vector fields (as defined in section~\ref{section.stable-manifold-theorem-notations-setup}) defined on $\PhaseSpaceVF \times \UnstableSpace$, still denoted by $\VFParam$. With this new smooth family of vector fields, equation~\eqref{e.stable-manifold-locally-coincide-1} implies: for every $\Parameter \in \UnstableSpace$ such that $\norme{\Parameter} \leq \TailleBoule(\Parameter_0)/4$ and for every $0<\delta \leq \TailleBoule(\Parameter_0)/4$, we have
\begin{equation}\label{e.stable-manifold-locally-coincide-2}
\Parameter+\LSMexp{0,\VFParam_\Parameter}{\ExpDecaySpeed}{\delta}= \LSMexp{\Parameter,\FixedVF}{\ExpDecaySpeed}{\delta}
\end{equation}
By hypothesis~\eqref{e.ligne-singularities} on $\FixedVF$, for every $\Parameter \in \UnstableSpace$, $\VFParam_\Parameter(0,0) =(0,0)$. By hypothesis~\eqref{e.PHS-ligne-singularities} on $\FixedVF$, $(\StableSpace_{\Parameter_0},\UnstableSpace)$ is a partially hyperbolic splitting of $\OriginDiff_{\Parameter_0}=\DiffPartielPoint{\VFParam}{0,0,0}{z,v}$ so $\VFParam:=(\VFParam_\Parameter)_{\Parameter \in \UnstableSpace}$ is a smooth family of vector fields satisfying the hypotheses~\ref{hypo.singularity} and~\ref{hypo.PHS}. Using the estimate
\begin{equation*}
\forall k \geq 2, \; \NormeSupHigherDerivativesLocMax{k}{\VFParam,\TailleBoule(\Parameter_0)/2} \leq \NormeSupHigherDerivativesLocMax{k}{\FixedVF,\Parameter_0,\TailleBoule(\Parameter_0)} \quad \text{(see~\eqref{e.norme-max-cas-local})}
\end{equation*}
it follows from theorem~\ref{thm.variete-gamma-stable-local-version} applied to $(\VFParam,\StableSpace_{\Parameter_0},\UnstableSpace)$ with $\TailleBoule = \TailleBoule(\Parameter_0)/2$ that there exists a smooth map
	\begin{equation*}
	\fonction{\SMTGraphMapLocal}{\StableSpace_{\Parameter_0} \times \UnstableSpace}{\UnstableSpace}{(\VarPhaseSpaceStable,\Parameter)}{\SMTGraphMapLocal_\Parameter(\VarPhaseSpaceStable)}
	\end{equation*}
such that for every $\Parameter \in B_{\UnstableSpace}\left(0,\tilde{\eta} \right)$, for every $0< \eta \leq \tilde{\eta}$ and for every $0 < \delta \leq \frac{\eta \EVII{\Parameter_0}}{\ConstantEstimateStableManifoldThm \ConstanteControleExpMatrix{\OriginDiff_{\Parameter_0}}}$, we have
	\begin{equation}\label{e.corollary-graph-local-structure}
	\LSMexp{0,\VFParam_\Parameter}{\ExpDecaySpeed}{\eta} \cap B_{\PhaseSpaceVF}(0,\delta) = \Graph \left(\SMTGraphMapLocal_\Parameter\right) \cap B_{\PhaseSpaceVF}(0,\delta)
	\end{equation}
where
\begin{align}
\label{e.choice-gamma}
	\ExpDecaySpeed & \egaldef -\frac{\abs{\SpectreUp{(\OriginDiff_{\Parameter_0})_{|\StableSpace_{\Parameter_0}}}}}{2} \\
	\tilde{\eta} & \egaldef \frac{1}{4\ConstII} \min\left( \frac{\EVII{\Parameter_0}}{ \ConstanteControleExpMatrix{\OriginDiff_{\Parameter_0}} \NormeSupHigherDerivativesLocMax{2}{\FixedVF,\Parameter_0,\TailleBoule(\Parameter_0)}}, \TailleBoule(\Parameter_0) \right) \quad \text{(see~\eqref{e.ConstanteControleExpMatrixNorme})}
\end{align}
Moreover, for every $(\VarPhaseSpaceStable,\Parameter) \in B_{\StableSpace_{\Parameter_0}}\left(0,\tilde{\delta}\right) \times B_{\UnstableSpace}\left(0,\tilde{\delta}\right)$,
	\begin{subequations}
		\label{e.estimates-SMT-graph-local-estimates-corollary}
		\begin{align}
		\label{e-controle-phi-0-local-estimates-corollary}
		\norme{\SMTGraphMapLocal(\VarPhaseSpaceStable,\Parameter)} &\leq\ConstII[0] \EVII{\Parameter_0}^{-2} \ConstanteControleExpMatrix{\OriginDiff_{\Parameter_0}}^2 \NormeSupHigherDerivativesLocMax{2}{\FixedVF,\Parameter_0,\TailleBoule(\Parameter_0)} (\norme{\VarPhaseSpaceStable}+\norme{\Parameter}) \norme{\VarPhaseSpaceStable} \\
		\label{e-controle-phi-1-derivative-stable-variable-local-estimates-corollary}
		\normesub{\DiffPartielPoint{\SMTGraphMapLocal}{\VarPhaseSpaceStable,\Parameter}{z}} &\leq\ConstII[1]
		\EVII{\Parameter_0}^{-2} \ConstanteControleExpMatrix{\OriginDiff_{\Parameter_0}}^2 \NormeSupHigherDerivativesLocMax{2}{\FixedVF,\Parameter_0,\TailleBoule(\Parameter_0)} (\norme{\VarPhaseSpaceStable}+\norme{\Parameter})\\
		\label{e-controle-phi-1-derivative-parameter-variable-local-estimates-corollary}
		\normesub{\DiffPartielPoint{\SMTGraphMapLocal}{\VarPhaseSpaceStable,\Parameter}{\Parameter}} &\leq	\ConstII[1]
		\EVII{\Parameter_0}^{-1} \ConstanteControleExpMatrix{\OriginDiff_{\Parameter_0}} \NormeSupHigherDerivativesLocMax{2}{\FixedVF,\Parameter_0,\TailleBoule(\Parameter_0)} \norme{\VarPhaseSpaceStable}
		\end{align}
where
\begin{equation*}
	\tilde{\delta} \egaldef \frac{\EVII{\Parameter_0}}{4\ConstII\ConstanteControleExpMatrix{\OriginDiff_{\Parameter_0}}} \min\left( \frac{\EVII{\Parameter_0}}{ \ConstanteControleExpMatrix{\OriginDiff_{\Parameter_0}} \NormeSupHigherDerivativesLocMax{2}{\FixedVF,\Parameter_0,\TailleBoule(\Parameter_0)}}, \TailleBoule(\Parameter_0) \right)
\end{equation*}
		and more generally, using the norm $\norme{(\VarPhaseSpaceStable,\Parameter)}=\norme{\VarPhaseSpaceStable}+\norme{\Parameter}$ on $\StableSpace_{\Parameter_0} \times \SetParameter$, we have, for all $k \geq 2$,
		\begin{multline}\label{e-controle-phi-2-local-estimates-corollary}
		\normesub{\DiffPointSup{\SMTGraphMapLocal}{\VarPhaseSpaceStable,\Parameter}{k}} \leq
		\ConstII[k] \left[ \left(\frac{\ConstanteControleExpMatrix{\OriginDiff_{\Parameter_0}}}{\EVII{\Parameter_0}}\right)^2 \right. \\
		\left. \max\left(\frac{\ConstanteControleExpMatrix{\OriginDiff_{\Parameter_0}} \NormeSupHigherDerivativesLocMax{2}{\FixedVF,\Parameter_0,\TailleBoule(\Parameter_0)}}{\EVII{\Parameter_0}}, \frac{2}{\TailleBoule(\Parameter_0)} \right)^{k-1} \NormeSupHigherDerivativesLocMax{k+1}{\FixedVF,\Parameter_0,\TailleBoule(\Parameter_0)}\right]^{2k-1}
		\end{multline}
	\end{subequations}

Let
\begin{equation}\label{e.defin-psi}
\fonction{\psi}{\StableSpace_{\Parameter_0} \times \UnstableSpace}{\StableSpace_{\Parameter_0} \times \UnstableSpace}{(z,\Parameter)}{(z,\Parameter + \phi(z,\Parameter))}
\end{equation}
One can rewrite~\eqref{e.defin-psi} as $\psi = \Id + h$ where $h(z,\Parameter) = (0,\phi(z,\Parameter))$. According to~\eqref{e-controle-phi-0-local-estimates-corollary}, \eqref{e-controle-phi-1-derivative-stable-variable-local-estimates-corollary}, \eqref{e-controle-phi-1-derivative-parameter-variable-local-estimates-corollary} and the mean value theorem, there exists a constant $K \geq 1$ (independant of $\FixedVF$, $\TailleBoule$ and $\Parameter_0$) such that for every $0< \epsilon \leq 1$ and for every $(\VarPhaseSpaceStable,\Parameter) \in B_{\StableSpace_{\Parameter_0}}\left(0,\frac{\epsilon \tilde{\delta}}{K}\right) \times B_{\UnstableSpace}\left(0,\frac{\epsilon \tilde{\delta}}{K}\right)$, we have
\begin{equation}\label{e.estimate-diff-h} 
\normesub{\DiffPoint{h}{z,\Parameter}} \leq \normesub{\DiffPoint{\phi}{z,\Parameter}} \leq \frac{\epsilon}{2}
\end{equation}
and
\begin{subequations}
	\label{e.neighbourhood-small-enough-proof}
\begin{align}
\label{e.neighbourhood-small-enough-proof-depart}
V_{\Parameter_0}  & \subset B_{\PhaseSpaceVF}\left(0,\frac{\hat{\eta}}{2}\right) \\
\label{e.neighbourhood-small-enough-proof-arrivee}
U_{\Parameter_0} & \subset B_{\PhaseSpaceVF}\left(0,\frac{\hat{\eta}}{2}\right)
\end{align}
\end{subequations}
where
\begin{align*}
V_{\Parameter_0} &\egaldef L \left(\tilde{V}_{\Parameter_0} \right) \\
\tilde{V}_{\Parameter_0} &\egaldef B_{\StableSpace_{\Parameter_0}}\left(0,\frac{\tilde{\delta}}{K}\right) \times B_{\UnstableSpace}\left(0,\frac{\tilde{\delta}}{K}\right) \\
U_{\Parameter_0}& \egaldef L\left(\tilde{U}_{\Parameter_0}\right) \\
\tilde{U}_{\Parameter_0} &\egaldef \psi\left(B_{\StableSpace_{\Parameter_0}}\left(0,\frac{\tilde{\delta}}{K}\right) \times B_{\UnstableSpace}\left(0,\frac{\tilde{\delta}}{K}\right)\right) \\
\hat{\eta} &\egaldef \frac{\tilde{\eta} \EVII{\Parameter_0}}{\ConstantEstimateStableManifoldThm \ConstanteControleExpMatrix{\OriginDiff_{\Parameter_0}}}
\end{align*}
and $L:\StableSpace_{\Parameter_0} \times \UnstableSpace \to \PhaseSpaceVF$ is the canonical isomorphism $(z,\Parameter) \mapsto z+\Parameter$. According to lemma~\ref{lemma.triangular-inequality-inverse-angle-two-spaces}, we have
\begin{equation}\label{e.control-identification-canonique}
	\normesub{L} \leq 1, \quad \normesub{L^{-1}} \leq \ConstanteInegTriangInverse{\StableSpace_{\Parameter_0},\UnstableSpace},
\end{equation}
the linear spaces $\PhaseSpaceVF$, $\StableSpace_{\Parameter_0}$ and $\UnstableSpace$ being equipped with the Euclidean norm and the linear space $\StableSpace_{\Parameter_0} \times \UnstableSpace$ being equipped with $\norme{(z,\Parameter)} = \norme{z}+\norme{\Parameter}$.
Using~\eqref{e.estimate-diff-h} with $\epsilon=1$, we get that $\psi$ is injective on $\tilde{V}_{\Parameter_0}$ and according to the global inverse function theorem, $\psi$ is invertible on $\tilde{V}_{\Parameter_0}$. Let us define $\tilde{\SLC}_{\Parameter_0}$ as the local inverse of $\psi$:
\begin{equation*}
	\tilde{\SLC}_{\Parameter_0} \egaldef \left(\psi_{|\tilde{V}_{\Parameter_0}}\right)^{-1}: \tilde{U}_{\Parameter_0} \to \tilde{V}_{\Parameter_0}
\end{equation*}
and let
\begin{equation}\label{e.defin-local-coord-straight-proof}
	\SLC_{\Parameter_0} = L \circ \tilde{\SLC}_{\Parameter_0} \circ L^{-1} : U_{\Parameter_0} \to V_{\Parameter_0}
\end{equation}
\proofstep{Proof of item~\eqref{thm.straightening}}
The first thing to remark is the fact that $\psi$ is constructed so that it maps straight lines to the graphs induced by $\SMTGraphMap$: more precisely, we have, for every $\Parameter \in \UnstableSpace$,
\begin{equation}\label{e.straight-lines-to-graphs}
	\psi\left(\StableSpace_{\Parameter_0}  \times \{\Parameter\} \right) = (0,\Parameter) + \Graph \SMTGraphMap_\Parameter
\end{equation}
Identifying $\Graph \SMTGraphMap_\Parameter \subset \StableSpace_{\Parameter_0} \times \UnstableSpace$ with its image in $\PhaseSpaceVF$, we have, for every $\Parameter \in B_{\UnstableSpace}\left(0,\frac{\tilde{\delta}}{K}\right)$, 
\begin{align*} 
\left(\Parameter + \Graph \SMTGraphMap_\Parameter\right) \cap U_{\Parameter_0} 
&= \Parameter +  \left(\Graph \SMTGraphMap_\Parameter\right) \cap \underbracket [1pt][7pt] {\left( U_{\Parameter_0} - \Parameter\right)}_{\subset B_{\PhaseSpaceVF}(0,\hat{\eta})} & \text{using~\eqref{e.neighbourhood-small-enough-proof-arrivee}}  \\
&=\left(\Parameter+ \LSMexp{0,\VFParam_\Parameter}{\ExpDecaySpeed}{\tilde{\eta}} \right) \cap U_{\Parameter_0} & \text{using~\eqref{e.corollary-graph-local-structure}}
\end{align*}
so, using~\eqref{e.stable-manifold-locally-coincide-2}, we get that
\begin{equation}\label{e.equality-graph-local-stable-set}
	\left(\Parameter + \Graph \SMTGraphMap_\Parameter\right) \cap U_{\Parameter_0} =\LSMexp{\Parameter,\FixedVF}{\ExpDecaySpeed}{\tilde{\eta}} \cap U_{\Parameter_0}
\end{equation}
Since the family $(\LSMexp{\Parameter,\FixedVF}{\ExpDecaySpeed}{\tilde{\eta}})_{\Parameter \in B_{\UnstableSpace}\left(0,\frac{\tilde{\delta}}{K}\right)}$ is pairwise disjoint, the family \begin{equation*}
	\left(\left(\Parameter + \Graph \SMTGraphMap_\Parameter\right) \cap U_{\Parameter_0}\right)_{\Parameter \in B_{\UnstableSpace}\left(0,\frac{\tilde{\delta}}{K}\right)}
\end{equation*}
is also pairwise disjoint. The preceding remark allows us to write
\begin{align*}
U_{\Parameter_0} 
&= \SLC_{\Parameter_0}^{-1}(V_{\Parameter_0}) \\
&= \SLC_{\Parameter_0}^{-1}\left(\bigsqcup_{\Parameter \in B_{\UnstableSpace}\left(0,\frac{\tilde{\delta}}{K}\right)} \left(\Parameter + \StableSpace_{\Parameter_0} \right) \cap V_{\Parameter_0} \right)  \\
&= \bigsqcup_{\Parameter \in B_{\UnstableSpace}\left(0,\frac{\tilde{\delta}}{K}\right)} \SLC_{\Parameter_0}^{-1}\left(\left( \Parameter + \StableSpace_{\Parameter_0} \right) \cap V_{\Parameter_0}\right) & \text{by injectivity of $\SLC_{\Parameter_0}^{-1}$} \\
& \subset \bigsqcup_{\Parameter \in B_{\UnstableSpace}\left(0,\frac{\tilde{\delta}}{K}\right)} \left(\Parameter + \Graph \SMTGraphMap_\Parameter\right) \cap U_{\Parameter_0} & \text{using~\eqref{e.straight-lines-to-graphs}} \\
& \subset \bigsqcup_{\Parameter \in B_{\UnstableSpace}\left(0,\frac{\tilde{\delta}}{K}\right)} \LSMexp{\Parameter,\FixedVF}{\ExpDecaySpeed}{\tilde{\eta}} \cap U_{\Parameter_0} & \text{using~\eqref{e.equality-graph-local-stable-set}} \\
& \subset U_{\Parameter_0} 
\end{align*}
where $\sqcup$ denotes a disjoint union. It follows that all the preceding inclusions must be instead equality. As consequences, we get that the family $(\LSMexp{\Parameter,\FixedVF}{\ExpDecaySpeed}{\tilde{\eta}})_{\Parameter \in B_{\UnstableSpace}\left(0,\frac{\tilde{\delta}}{K}\right)}$ foliates $U_{\Parameter_0}$:
	\begin{equation}\label{e.local-manifolds-foliate-neighbourhood}
		U_{\Parameter_0}  = \bigsqcup_{\Parameter \in B_{\UnstableSpace}\left(0,\frac{\tilde{\delta}}{K}\right)} \LSMexp{\Parameter,\FixedVF}{\ExpDecaySpeed}{\tilde{\eta}} \cap U_{\Parameter_0}
	\end{equation}
and for every $\Parameter \in B_{\UnstableSpace}\left(0,\frac{\tilde{\delta}}{K}\right)$,
\begin{equation} \label{e.lines-maps-to-stable-manifolds-1}
	\SLC_{\Parameter_0}^{-1}\left(\left( \Parameter + \StableSpace_{\Parameter_0} \right) \cap V_{\Parameter_0}\right) = \LSMexp{\Parameter,\FixedVF}{\ExpDecaySpeed}{\tilde{\eta}} \cap U_{\Parameter_0}
\end{equation}
Let $\Parameter \in B_{\UnstableSpace}\left(0,\frac{\tilde{\delta}}{K}\right)$. Since any orbit contained in $\SM{\Parameter,\FixedVF}$ must eventually enter $U_{\Parameter_0}$, it follows from~\eqref{e.local-manifolds-foliate-neighbourhood} and the fact that such an orbit converges to $\Parameter$, that
\begin{equation}\label{e.lines-maps-to-stable-manifolds-2}
	\LSMexp{\Parameter,\FixedVF}{\ExpDecaySpeed}{\tilde{\eta}} \cap U_{\Parameter_0} = \SM{\Parameter,\FixedVF} \cap U_{\Parameter_0}
\end{equation}
According to~\eqref{e.lines-maps-to-stable-manifolds-1} and~\eqref{e.lines-maps-to-stable-manifolds-2}, we have, for every $\Parameter \in B_{\UnstableSpace}\left(0,\frac{\tilde{\delta}}{K}\right)$,
\begin{equation*}
	\SLC_{\Parameter_0}^{-1}\left(\left( \Parameter + \StableSpace_{\Parameter_0}\right) \cap V_{\Parameter_0}\right) = \SM{\Parameter,\FixedVF} \cap U_{\Parameter_0}
\end{equation*}
so item~\eqref{thm.straightening} holds true.
\proofstep{Proof of item~\eqref{thm.local-coord-form}} This is a direct consequence of~\eqref{e.defin-psi} and the fact that $\SMTGraphMap(0,\Parameter)=0$ (see~\eqref{e-controle-phi-0-local-estimates-corollary}).
\proofstep{Proof of items~\eqref{thm.size-neighbourhoods} and~\eqref{thm.close-Id}}
We have $\Diff{\tilde{\SLC}_{\Parameter_0}^{-1}} = \Id + \Diff{h}$ so
\begin{equation}\label{e-expression-inverse-Dpsi}
\Diff{\tilde{\SLC}_{\Parameter_0}} = \Id + \sum_{k \geq 1}(-1)^k \left(\Diff{h}\right)^k
\end{equation}
Using~\eqref{e.estimate-diff-h}, it follows that for every $0< \epsilon \leq 1$ and for every $(\VarPhaseSpaceStable,\Parameter) \in B_{\StableSpace_{\Parameter_0}}\left(0,\frac{\epsilon\tilde{\delta}}{K}\right) \times B_{\UnstableSpace}\left(0,\frac{\epsilon\tilde{\delta}}{K}\right)$,
\begin{equation}\label{e.comparaison-Id-LCS}
\normesub{\DiffPoint{\tilde{\SLC}_{\Parameter_0}}{z,\Parameter}-\Id} \leq  \epsilon
\end{equation}
According to~\eqref{e.neighbourhood-small-enough-proof-arrivee},~\eqref{e.comparaison-Id-LCS}, the mean value theorem and the fact that $\SMTGraphMap(0,0)=(0,0)$, there exists a contant $K' \geq K$ (independant of $\FixedVF$, $\TailleBoule$ and $\Parameter_0$)) such that
\begin{equation*}
B_{\StableSpace_{\Parameter_0}}\left(0,\frac{\tilde{\delta}}{K'}\right) \times B_{\UnstableSpace}\left(0,\frac{\tilde{\delta}}{K'}\right) \subset \tilde{U}_{\Parameter_0}  \cap \tilde{V}_{\Parameter_0}
\end{equation*}
We then use~\eqref{e.control-identification-canonique} to obtain that, for every $0 < \epsilon \leq 1$,
\begin{equation*}
	B_{\PhaseSpaceVF}(0,\epsilon R_{\Parameter_0}) \subset L\left( B_{\StableSpace_{\Parameter_0}}\left(0,\frac{\epsilon}{\ConstanteInegTriangInverse{\StableSpace_{\Parameter_0},\UnstableSpace}} \frac{\tilde{\delta}}{K'}\right) \times B_{\UnstableSpace}\left(0,\frac{\epsilon}{\ConstanteInegTriangInverse{\StableSpace_{\Parameter_0},\UnstableSpace}} \frac{\tilde{\delta}}{K'}\right) \right) \subset U_{\Parameter_0} \cap V_{\Parameter_0}
\end{equation*}
where $R_{\Parameter_0}$ is defined by~\eqref{e.definition-rayon-minimal} for some constant $\ConstIII$ large enough (independant of $\FixedVF$, $\TailleBoule$ and $\Parameter_0$). The above inclusion with $\epsilon=1$ proves that item~\eqref{thm.size-neighbourhoods} holds true. Even if it means taking $\ConstIII$ larger, item~\eqref{thm.close-Id} holds true as well, using~\eqref{e.estimate-diff-h},~\eqref{e.control-identification-canonique},~\eqref{e.comparaison-Id-LCS}, the mean value theorem and the fact that $\SMTGraphMap(0,0)=(0,0)$.

\proofstep{Proof of item~\eqref{thm.Ck-norms}} This is a consequence of~\eqref{e-expression-inverse-Dpsi},~\eqref{e.defin-psi},~\eqref{e-controle-phi-2-local-estimates-corollary},~\eqref{e.control-identification-canonique} and the fact that the sequence $(\NormeSupHigherDerivativesLocMax{k}{\FixedVF,\Parameter_0,\TailleBoule(\Parameter_0)})_{k \geq 2}$ is increasing.
\proofstep{Proof of item~\eqref{thm.local-unicity}}
One can construct, for any $\Parameter \in \TraceOpenSetDefinVF$, a local coordinate system $\SLC_{\Parameter}$ satisfying items~\eqref{thm.size-neighbourhoods}-\eqref{thm.Ck-norms} in the same way than $\SLC_{\Parameter_0}$ (and with the same constants). Let $\Parameter_1 \in \TraceOpenSetDefinVF$. By construction of $\SLC_{\Parameter_0}$ and $\SLC_{\Parameter_1}$ (see~\eqref{e.defin-psi} and~\eqref{e.defin-local-coord-straight-proof}), for all $\Parameter \in \TraceOpenSetDefinVF \cap U_{\Parameter_0} \cap U_{\Parameter_1}$ and for all $y \in \LSMexp{\Parameter,\FixedVF}{\ExpDecaySpeed}{\tilde{\eta}} \cap U_{\Parameter_0} \cap U_{\Parameter_1}$, we have
\begin{equation*}
	\begin{cases}
	\SLC_{\Parameter_0}(y) &= \Parameter + \pi_{\Parameter_0}(y) \\
	\SLC_{\Parameter_1}(y) &= \Parameter + \pi_{\Parameter_1}(y)
	\end{cases}
\end{equation*}
where $\pi_{\Parameter}$ denotes the linear projection along $\UnstableSpace$ onto $\StableSpace_{\Parameter}$. It follows that for all $y \in U_{\Parameter_0} \cap U_{\Parameter_1}$, we have
\begin{equation*}
	\SLC_{\Parameter_0}(y) - \SLC_{\Parameter_1}(y) = \pi_{\Parameter_0}(y) - \pi_{\Parameter_1}(y)
\end{equation*}
so item~\eqref{thm.local-unicity} holds true.
\end{proof}

\appendix

\section{Some linear algebra lemmas}
\label{section.appendix-lemmas-algebra}
We recall here some elementary facts of linear algebra that will be of great importance throughout this paper. We refer to section~\ref{section.notations} for the notations.
\begin{lemme}\label{lemma.triangular-inequality-inverse-angle-two-spaces}
	Let $n \in \N$ and let $F,G$ be two linear subspaces of $\R^n$ such that $F \cap G = \{ 0\}$. For every $x=x_F + x_G \in F \oplus G$,
	\begin{equation*}
	\norme{x_F}_2 + \norme{x_G}_2 \leq \ConstanteInegTriangInverse{F,G}\norme{x}_2
	\end{equation*}
\end{lemme}

\begin{proof}
	Recall that $\ConstanteInegTriangInverse{F,G} = \left(\frac{2}{1-\cos \Angle{F,G}}\right)^{\frac{1}{2}}$. Let $x=x_F + x_G \in F \oplus G$. It is sufficient to prove the straightforward inequality
	\begin{equation*}
		\frac{a^2+b^2+2ab}{a^2+b^2-2ab\epsilon} \leq \frac{2}{1-\epsilon}
	\end{equation*}
	where $a=\norme{x_F}_2^2$, $b=\norme{x_G}_2^2$ and $ \epsilon = \cos \Angle{F,G} \in [0,1[$.
\end{proof}

\begin{lemme}\label{lemma.triangular-inequality-inverse-angles}
	Let $n \in \N$ and $A \in \MatrixSet_n(\R)$. Let $\R^n = \oplus_{1 \leq i \leq r} E_i$ be the decomposition of $\R^n$ as the direct sum of the generalized eigenspaces of $A$. Accordingly, for any $x \in \R^n$, we will use the decomposition $x = \sum_{i=1}^{r} x_i$ where $x_i \in E_i$. The following control holds true for every $x \in \R^n$:
	\begin{equation*}
	\sum_{i=1}^{r}\norme{x_i}_2 \leq \ConstanteInegTriangInverse{A}\norme{x}_2
	\end{equation*}
\end{lemme}

\begin{proof}
The proof is a straightforward induction on the number $r$ of generalized eigenspaces of $A$, using lemma~\ref{lemma.triangular-inequality-inverse-angle-two-spaces}.
\end{proof}

\begin{lemme}\label{lemma.controle-polynomial-exp-matrices}
	Let $n \in \N$ and $A \in \MatrixSet_n(\R)$. We have, for every $\BorneVPStable > \SpectreUp{A}$ and for every $s \geq 0$,
	\begin{equation*}\label{e.controle-matrice-exp-general}
	\normesub{e^{sA}}_2 \leq 2^{n-1}(n-1)^{n-1}  \frac{\max\left(1,\normesub{A}_2\right)^{n-1}\ConstanteInegTriangInverse{A}}{\min\left(1,\BorneVPStable - \SpectreUp{A}\right)^{n-1}} e^{\BorneVPStable s} 
	\end{equation*}
\end{lemme}
\begin{proof}
	Fix $ \SpectreUp{A} < \BorneVPStable \leq \SpectreUp{A}+1$ and $s \geq 0$.
	Let 
\begin{equation*}
	\R^n = \oplus_{1 \leq i \leq r}  \Ker \left(A - \mu_i \Id\right)^{d_i}
\end{equation*}
be the decomposition of $\R^n$ as the direct sum of the generalized eigenspaces of $A$.
Fix $x = \sum_{i=1}^{r} x_i \in \R^n$, where $x_i \in \Ker \left(A - \mu_i \Id\right)^{d_i}$. For every $1 \leq i \leq r$.
\begin{align*}
	\norme{e^{sA}x_i}_2 & = \norme{e^{s\mu_i \Id}e^{s(A - \mu_i \Id)} x_i}_2 \\
	& = \norme{e^{s\mu_i \Id} \sum_{j=0}^{d_i-1}\frac{s^j}{j!}(A-\mu_i \Id)^jx_i}_2 \\
	& \leq e^{s \RealPart(\mu_i)}\norme{x_i}_2 \sum_{j=0}^{d_i-1}\frac{s^j}{j!} (2\normesub{A}_2)^j   \\
	& \leq e^{s\BorneVPStable}\norme{x_i}_2 2^{d_i-1} \max\left(1,\normesub{A}_2\right)^{d_i-1} e^{s(\RealPart(\mu_i)-\BorneVPStable)}(1+s)^{d_i-1}
\end{align*}
where we used the fact that $\abs{\mu_i} \leq \normesub{A}_2$ by Browne theorem.
By a straightforward computation, we obtain
\begin{equation*}
	\max_{t \geq 0} e^{t(\RealPart(\mu_i)-\BorneVPStable)}(1+t)^{d_i-1} \leq
\begin{cases}
\frac{(d_i-1)^{d_i-1}}{\left(\BorneVPStable - \RealPart(\mu_i)\right)^{d_i-1}} & \text{if $\BorneVPStable - \RealPart(\mu_i) \leq d_i-1$} \\
1 & \text{if $\BorneVPStable - \RealPart(\mu_i) > d_i-1$}
\end{cases}
\end{equation*}
It follows that
\begin{equation*}
	\norme{e^{sA}x}_2 \leq e^{s\BorneVPStable}2^{n-1}(n-1)^{n-1}  \frac{\max\left(1,\normesub{A}_2\right)^{n-1}}{\min\left(1,\BorneVPStable - \SpectreUp{A}\right)^{n-1}} \sum_{i=1}^{r}\norme{x_i}_2
\end{equation*}
Using lemma~\ref{lemma.triangular-inequality-inverse-angles}, we obtain the desired inequality.
\end{proof}

\printbibliography
\end{document}